\documentclass[12pt, twoside]{article}

\usepackage{tocloft}
\renewcommand{\cftsecleader}{\cftdotfill{\cftdotsep}}
\usepackage{titletoc}
\usepackage{setspace}
\usepackage{tabularx}
\usepackage{pdfpages}
\usepackage{amsmath,amssymb}
\usepackage{amsthm}
\usepackage{epsfig}
\usepackage{graphics}
\usepackage{graphicx}
\usepackage{colordvi}
\usepackage{mathrsfs} 
\usepackage{stmaryrd}
\usepackage{geometry}
\usepackage{dsfont}
\usepackage{tikz}
\usetikzlibrary{automata}
\usepackage{url}
\usepackage[hidelinks]
{hyperref}
\hypersetup{
    colorlinks=false,
   linkcolor=blue!75!black,
    filecolor=blue!75!black,      
   urlcolor=blue!75!black, 
   citecolor=blue!75!black,
}
\usepackage{enumitem}
\usepackage[textsize=tiny]{todonotes}
\usepackage[margin=1cm, size=small]{caption}
\usepackage{multirow}
\allowdisplaybreaks[4]
\oddsidemargin
0in \evensidemargin 0in\marginparwidth 0.4in
\usepackage{color,soul}
\usepackage{xcolor}
\usepackage{colortbl}

\usepackage{etoolbox}
\patchcmd{\thebibliography}{\section*}{\section}{}{}
\usepackage{fancyvrb}
\usepackage{fvextra}
\usepackage{fancyhdr}
\usepackage{comment}
\usepackage{chngcntr}
\counterwithin{figure}{section}

\DeclareFontFamily{U}{mathx}{}
\DeclareFontShape{U}{mathx}{m}{n}{<-> mathx10}{}
\DeclareSymbolFont{mathx}{U}{mathx}{m}{n}
\DeclareMathAccent{\widehat}{0}{mathx}{"70}
\DeclareMathAccent{\widecheck}{0}{mathx}{"71}
\DeclareMathAlphabet{\little}{U}{dutchcal}{m}{n}

\usepackage{empheq}
\usetikzlibrary{automata}
\definecolor{blue2}{cmyk}{.94,.11,0,0}
\usepackage{enumitem}
\setitemize{itemsep=0pt}
\definecolor{myblue}{rgb}{.8, .8, 1}
\newlength\mytemplen
\newsavebox\mytempbox

\newcommand{\ol}{\overline}

\newcommand{\BES}{{\rm BES}}

\newcommand{\e}{{\mathrm e}}
\newcommand{\1}{\mathds 1}
\newcommand{\G}{\mathscr G}
\newcommand{\C}{{\mathscr C}}

\newcommand{\B}{\mathscr B}

\renewcommand{\O}{{\mathcal O}}

\newcommand{\N}{{\mathbf N}}
\newcommand{\vep}{{\varepsilon}}

\newcommand{\sgn}{{\rm sgn}}

\newcommand{\dint}{{\int\!\!\!\int}}

\newcommand{\supp}{{\rm supp}}

\newcommand{\bs}{\boldsymbol}
\newcommand{\ms}{\mathscr}
\renewcommand{\P}{{\mathbb P}}
\newcommand{\E}{{\mathbb E}}

\newcommand{\mc}{\mathcal}
\renewcommand{\d}{{\mathrm d}}

\newcommand{\R}{{\Bbb R}}

\renewcommand{\i}{{\mathtt i}}
\newcommand{\defeq}{{\stackrel{\rm def}{=}}}

\renewenvironment{proof}[1][\proofname]{\noindent {\bfseries #1.}\;}{\hfill\ensuremath{\blacksquare}\\}

\newcommand{\EM}{\gamma_{\mathsf E\mathsf M}}
\newcommand{\cc}{{^\circ}}

\newcommand{\two}{{\sqrt{2}}}

\newcommand{\oovarphi}{\overline{\overline{\varphi}}}
\newcommand{\oophi}{\overline{\overline{\varphi}}{\vphantom{\overline{\varphi}}}}

\newcommand{\lv}{\Lambda_\vep}
\newcommand{\llv}{\Lambda_\vep}

\newcommand{\uvep}{{}^{\,\vep}}

\usepackage{multirow}
\usepackage[margin=2cm]{caption}

\newcommand{\slla}{\langle\!\langle}
\newcommand{\srra}{\rangle\!\rangle}
\newcommand{\wt}{\widetilde}

\makeatletter
\newsavebox{\@brx}
\newcommand{\lla}[1][]{\savebox{\@brx}{\(\m@th{#1\langle}\)}
  \mathopen{\copy\@brx\kern-0.5\wd\@brx\usebox{\@brx}}}
\newcommand{\rra}[1][]{\savebox{\@brx}{\(\m@th{#1\rangle}\)}
  \mathclose{\copy\@brx\kern-0.5\wd\@brx\usebox{\@brx}}}
\makeatother

\newtheoremstyle{slantthm}{10pt}{10pt}{\slshape}{}{\bfseries}{}{.5em}{\thmname{#1}\thmnumber{ #2}\thmnote{ (#3)}.}
\newtheoremstyle{slantrmk}{10pt}{10pt}{\rmfamily}{}{\bfseries}{}{.5em}{\thmname{#1}\thmnumber{ #2}\thmnote{ (#3)}.}

\begin{document}
\theoremstyle{slantthm}
\newtheorem*{mthm}{Main Theorem}
\newtheorem{thm}{Theorem}[section]
\newtheorem{prop}[thm]{Proposition}
\newtheorem{lem}[thm]{Lemma}
\newtheorem{cor}[thm]{Corollary}
\newtheorem{defi}[thm]{Definition}
\newtheorem{disc}[thm]{Discussion}
\newtheorem{conj}[thm]{Conjecture}

\theoremstyle{slantrmk}
\newtheorem{ass}[thm]{Assumption}
\newtheorem{rmk}[thm]{Remark}
\newtheorem{eg}[thm]{Example}
\newtheorem{que}[thm]{Question}
\numberwithin{equation}{section}
\newtheorem{quest}[thm]{Quest}
\newtheorem{prob}[thm]{Problem}
 \newtheorem{nota}[thm]{Notation}
\newcommand{\thetitle}{Asymptotic exponential moments of two-dimensional additive functionals above subcriticality}

\thispagestyle{empty}
\setlength{\cftbeforesecskip}{1pt}
\renewcommand{\cftsecleader}{\cftdotfill{\cftdotsep}}

\setcounter{page}{1}

\title{\vspace{-1.5cm}
\bf \thetitle\footnote{Support from an NSERC Discovery grant is gratefully acknowledged.}}

\author{Yu-Ting Chen\,\footnote{Department of Mathematics and Statistics, University of Victoria, British Columbia, Canada.}\,\,\footnote{Email: \url{chenyuting@uvic.ca}}\vspace{-.4cm}}

\date{\today \vspace{-1cm}
}

\maketitle
\abstract{We introduce new asymptotic formulas for the exponential moments of rescaled additive functionals of two-dimensional Brownian motion, extending the analysis to parameter regimes that go beyond the subcritical case underlying distributional limits. As an application, we develop a probabilistic–analytic approximation for the two-dimensional Laplacian with a singular perturbation at the origin. This method describes the perturbed Laplacian by combining analytic characteristics of the Kallianpur–Robbins and Kasahara–Kotani laws. \medskip  

\noindent \emph{Keywords:} Brownian motion; additive functionals; large deviations; excursions; local times; Schr\"odinger operators. 
\smallskip 

\noindent \emph{Mathematics Subject Classification (2020):} 60F10, 60J55, 60J65, 60H30.

\setlength{\cftbeforesecskip}{0pt}
\setlength\cftaftertoctitleskip{0pt}
\renewcommand{\cftsecleader}{\cftdotfill{\cftdotsep}}
\setcounter{tocdepth}{2}
\tableofcontents

\section{Introduction}\label{sec:intro}
The study of limit theorems for additive functionals has a long history centred on large-time asymptotic laws. In this context, two-dimensional Brownian motion has gathered significant attention; see \cite{PY:86} for a survey. This is owing to the special feature that for a two-dimensional standard Brownian motion $W$, points are \emph{polar} in the sense that for any point $z$, $\P(W_t\neq z\;\forall\; t>0)=1$, whereas
\begin{align}\label{def:Af}
A_\varphi(t)\,\defeq\,\int_0^t \varphi(W_s)\d s,
\end{align}
which is an additive functional, tends to infinity in distribution as $t\to\infty$. More specifically, such a property of $A_\varphi(t)$ holds,  for example, when the support of the function $\varphi$ is compact and has a nonempty interior, and $\lla \varphi\rra>0$, where
\begin{align}
\lla \varphi\rra\, \defeq\int \varphi(z')\d z'.\label{def:phig}
\end{align} 
We refer to $\varphi$ in $A_{\varphi}(t)$ as the {\bf occupation function} because when $\varphi$ is an indicator function, $A_{\varphi}(t)$ recovers the occupation measure.

In this paper, our main focus is on the framework of $A_{\varphi_\vep}(t)$ for $\vep\searrow 0$, where
\begin{align}\label{def:ffvep}
\varphi_\vep(z')\,\defeq\,\vep^{-2}\varphi(\vep^{-1}z'),
\end{align}
and $W_0=\vep z$ for fixed $z$. {\bf Here and in what follows, the functions $\varphi$ are assumed to be bounded, have compact support, and satisfy $\lla |\varphi|\rra > 0$, although $\varphi$ need not be nonnegative.} The frameworks of $A_{\varphi_\vep}(t)$ for $\vep\searrow 0$ and $A_\varphi(t)$ for $t\to\infty$ are essentially interchangeable, due to the following identity:
\begin{align}\label{Brownianscaling}
\P^W_{\vep z}(A_{\varphi_\vep}(t)\in \d x)=
\P^W_z(A_{\varphi}(\vep^{-2}t)\in \d x),
\end{align}
which uses the notation $W_0=z'$ under $\P^W_{z'}$ and follows from the Brownian scaling property. Nevertheless, the additive functionals $A_{\varphi_\vep}(t)$ provide a more straightforward framework for the viewpoints presented later. More importantly, despite the aforementioned polarity of points, $A_{\varphi_\vep}(t)$ is not identically equal to zero and can serve as an approximate local time when the occupation functions $\varphi_\vep$ are approximations to the identity, which is the case when $0\leq \varphi\in \C_c$ and $\lla \varphi\rra=1$. Such approximate local times are central to certain mathematical approaches in statistical physics and naturally emerge as probabilistic analogues of quantum-solvable models using delta-function potentials. We will briefly discuss these connections later in this introduction. 
 
The primary goal of this paper is to extend the asymptotic law of Kallianpur and Robbins \cite{KR:53}, which generalizes the blow-up of additive functionals mentioned below \eqref{def:Af}, and another asymptotic law by Kasahara and Kotani~\cite{KK} concerning the fluctuations of additive functionals. Under the conditions on $\varphi$ specified above, these two laws apply and can be restated as the following two limits in distribution, respectively, using the setting of $A_{\varphi_\vep}(t)$ subject to $\P^W_{\vep z}$ with $\vep\searrow 0$:
\begin{align}
\frac{\pi }{\log \vep^{-1}}A_{\varphi_\vep}(t)&\xrightarrow[\vep\to 0]{\rm (d)}\mathbf e\quad\mbox{if }\lla \varphi\rra=1,\label{def:KR}\\
\label{def:KK}
\sqrt{\frac{\pi}{\lla \mc E(\varphi)\rra \log \vep^{-1}}}A_{\varphi_\vep}(t)&\xrightarrow[\vep\to 0]{\rm (d)}\two \, B_{\mathbf e}\quad\mbox{if }\lla \varphi\rra=0.
\end{align}
Here, we adopt the following notations:
\begin{itemize}
\item In \eqref{def:KR} and \eqref{def:KK}, $\mathbf e$ is an exponential random variable with mean $1$. 
\item In \eqref{def:KK}, $B$ is a one-dimensional Brownian motion  starting from $0$ and independent of $\mathbf e$,  and $\lla \mc E(\varphi)\rra$ is called the {\bf logarithmic energy} as we set
\begin{align} \label{def:energy}
\mathcal E(\varphi)(z')&\,\defeq\, \varphi(z')\int \kappa(z'-z'')\varphi(z'')\d z'',\quad  
\kappa(z')\,\defeq\, \frac{1}{\pi}\log\frac{1}{ |z'|}.
\end{align}
A method for verifying the strict positivity of $\lla\mathcal E(\varphi)\rra$ is presented in \cite[Example~3.3]{Mattner}.
\end{itemize}
In the sequel, we will refer to \eqref{def:KR} and \eqref{def:KK} as the {\bf Kallianpur--Robbins law} and the {\bf Kasahara--Kotani law}, respectively, and to both collectively as {\bf the two laws}. 

Our extensions of the two laws address exponential moments by establishing asymptotic representations of the following Feynman–Kac formulas and their variants as $\vep \searrow 0$:
\begin{align}\label{def:FK0}
\E^{W}_{\vep z}[\e^{\mu_\vep A_{\varphi_\vep}(t)}g(W_t)]\quad (\E^W_{z'}\,\defeq\, \E^{\P^{W}_{z'}}).
\end{align}
Here, $g$ are bounded nonnegative with $\lla g\rra>0$, and
\begin{align}\label{def:muvep}
\mu_\vep\equiv \begin{cases}
\lv,&\mbox{if } \lla\varphi\rra=1\\
\lv^{1/2},&\mbox{if }\lla \varphi\rra=0
\end{cases},\quad 
\mbox{ for $\lv\sim {\rm cnst}\cdot\log^{-1}\vep^{-1}$.}
\end{align}
That is, we will work with $\mu_\vep$ such that the leading orders obey the scalings $\log^{-1}\vep^{-1}$ and $\log^{-1/2}\vep^{-1}$ for $\lla\varphi\rra=1$ and $\lla\varphi\rra=0$, respectively, as in \eqref{def:KR} and \eqref{def:KK}. The variants introduced in the main theorems are tailored to the regime $\lla \varphi \rra = 1$, thus extending the established formulas in \eqref{finalgoal00}.

While the scalings from the two laws are essential to our approach, the problem of establishing asymptotic representations for the Feynman--Kac formulas in \eqref{def:FK0} is in fact broader. Through these formulas, the choice of $\mu_\vep$ allows us to analyze two-dimensional Brownian motion conditioned on the exponentially rare events where $A_{\varphi_\vep}(t)$ are large. Specifically, these exponentially rare events exist because $\E^{W}_{\vep z}[\e^{\mu_\vep A_{\varphi_\vep}(t)}]$ may either converge or diverge to infinity as $\vep \searrow 0$, depending on the following criteria for the limiting distributions in the two cases, as determined by \eqref{def:KR} and \eqref{def:KK}:
\begin{align}
\E[\exp\{\mu\mathbf e\}]<\infty&\Longleftrightarrow \mu<1,\label{def:murange}\\
\E[\exp\{\sqrt{\mu}\two B_{\mathbf e}\}]<\infty&\Longleftrightarrow \mu<1.\label{KK:expmom}
\end{align}
The corresponding pre-limit expressions for $\E^{W}_{\vep z}[\e^{\mu_\vep A_{\varphi_\vep}(t)}]$ use the choices of $\mu_\vep$ below:
\begin{align}
\mu_\vep &\sim \frac{\mu \pi}{\log\vep^{-1}}\quad \mbox{if }\lla \varphi\rra=1;\label{lvasymp:KR}\\
\mu_\vep &\sim\sqrt{ \frac{\mu \pi}{\lla \mc E(\varphi)\rra\log\vep^{-1}}}\quad \mbox{if }\lla \varphi\rra=0.\label{lvasymp:KK}
\end{align}
Due to \eqref{def:murange} and \eqref{KK:expmom}, the {\bf subcritical}, {\bf critical} and {\bf supercritcal} regimes of $\mu_\vep$ will refer to $\mu\in (0,1)$, $\mu=1$ and $\mu>1$, respectively, in the sense of \eqref{lvasymp:KR}--\eqref{lvasymp:KK}. 

Our main results focus on the critical and supercritical regimes. In these settings, the asymptotic behavior of two-dimensional Brownian motion conditioned on large values of $A_{\varphi_\vep}(t)$ is governed by large deviation theory. Moreover, the characterizations of these asymptotic laws may be initiated by studying the asymptotic representations of the Feynman--Kac formulas in \eqref{def:FK0}. A direct way to clarify these relationships is to consider the following Esscher transforms, which are concentrated on the events that $\e^{\mu_\vep A_{\varphi_\vep(t)}}$ are large:
\begin{align}\label{Esscher}
\frac{\e^{\mu_\vep A_{\varphi_\vep}(t)}}{\E^{W}_{\vep z}[\e^{\mu_\vep A_{\varphi_\vep}(t)}]}\d \P^W_{\vep z}.
\end{align}
For similar Esscher transforms and discussions on additional connections to large deviation theory, see \cite{CT:15, Dagallier:25, Jack:20}.  
 
Explicit asymptotic representations of the Feynman–Kac formulas in \eqref{def:FK0} provide a basis for analyzing the limiting laws associated with \eqref{Esscher} and establish a broader framework by circumventing the need for time-dependent normalizing factors. The asymptotic results, discussed below, accommodate general initial conditions by taking
 \begin{align}\label{def:FK1}
\E^W_{ z}\big[\e^{\lv A_{\varphi_\vep}(t)}g(W_t)\big];
\end{align}
conversely, the expectation in \eqref{def:FK1} generates \eqref{def:FK0} when $\mu_\vep \equiv \lv$ via expansion (see \eqref{exp:12proof0} for details). The most significant of these earlier results pertains to the case
\begin{align}\label{nontrivial}
 \lla \varphi\rra=1 \quad \&\quad \mu=1.
\end{align} 
In this context, \cite{C:BES} establishes precise Feynman–Kac-type formulas for the limits and introduces a singular diffusion that demonstrates how an extremal property of two-dimensional Brownian motion emerges above subcriticality, as suggested by \eqref{Esscher}. For the construction of a conditional version of this diffusion, see Clark and Mian~\cite{CM:25}. 

Explicit asymptotic representations of the Feynman–Kac formulas discussed above were initially derived within the functional analytic framework associated with the Feynman–Kac semigroup in \eqref{def:FK1}. The corresponding operator serves as an approximation for constructing quantum-solvable models of the two-dimensional Laplacian subject to a singular perturbation at the origin, as reviewed in \cite{AGHH:Solvable}; the analysis also accounts for rescaled functions $\varphi_\vep$ with zero integral. In both cases $\lla \varphi \rra = 1$ and $\lla \varphi \rra = 0$, the limiting resolvents are explicitly solvable using 
\begin{align}
\lv\equiv  \frac{\mu\pi}{\log \vep^{-1}}+\frac{\lambda\pi }{\log^2 \vep^{-1}}.
\label{def:lvoointro}
\end{align}
In particular, the formula corresponding to \eqref{nontrivial}, restated in \eqref{finalgoal00}, shows that the limiting resolvents depend on $(\lambda, \varphi)$ through $\beta$, defined by 
\begin{align}\label{lim:beta}
\frac{\log \beta}{2}&=\pi\lla \mc E(\varphi)\rra +\frac{\log 2}{2}+\lambda - \gamma_{\sf EM},
\end{align}
where $\EM$ denotes the Euler–Mascheroni constant (see, for example, \cite{AGHH:Solvable, BC:2D, DR:Schrodinger, GQT}). Notably, Berini and Cancrini~\cite{BC:2D} utilized these results as essential tools in developing the singular SPDE known as the two-dimensional stochastic heat equation at criticality. For further background and recent developments on applying the Feynman–Kac formulas under consideration to the analysis of this SPDE, see \cite{CSZ:19, C:DBG, GQT}, among others.

A central approach in our analysis of the Feynman--Kac formulas in \eqref{def:FK0} beyond subcriticality is to compare the existing asymptotic representations with the two laws. For $\lla \varphi\rra=0$, it is noteworthy that the dichotomy in \eqref{def:muvep} is \emph{not} applied in those existing results, which are based solely on $\mu_\vep\equiv \lv$ for $\lv$ in \eqref{def:lvoointro}. Also, for $\lla\varphi\rra =1$, our starting point for invoking the two laws is that the blow-up of exponential moments arises only in the limiting case, and distributional descriptions of the two laws in \eqref{def:KR} and \eqref{def:KK} do not, by themselves, involve criticality. This observation suggests that the detailed structure of the additive functionals, ``just before collapsing to the limits,'' might still account for the asymptotic representations associated with criticality. In particular, the Kallianpur--Robbins law \eqref{def:KR} is expected to be central given the criterion in \eqref{def:murange}. Additionally, the special appearance of the logarithmic energy term $\lla \mc E(\varphi)\rra$ in \eqref{lim:beta} points to the relevance of the Kasahara--Kotani law in \eqref{def:KK}. Further heuristics are provided at the beginning of Section~\ref{sec:KRreg}.

In this paper, we establish extensions of the two laws by proving asymptotic representations of the Feynman--Kac formulas in \eqref{def:FK0} and suitable variants, using the following Green's function formulation with bounded nonnegative $g$ satisfying $\lla g\rra>0$:
\begin{align}
\label{def:Gf}
G^f_q\{g\}(z)&\,\defeq \int_0^\infty \e^{-q t}\E^W_z[\e^{A_f(t)}g(W_t)]\d t,\\
G_q\{g\}(z)&\,\defeq\,G_q^0\{g\}(z).\label{def:Gq0}
\end{align}
The following gives an overview of the main theorems of this paper: 
\begin{itemize}
\item For $\lla \varphi\rra=0$, Theorem~\ref{thm:=0} proves asymptotic representations of $G^{\lv^{1/2}\varphi_\vep}_q\{g\}(\vep z)$ as $\vep\to 0$ for all sufficiently large $q>0$, where, for $\mu\in (0,1]$ and $\lambda',\lambda\in \R$,
\begin{gather}
\llv \equiv  \frac{\mu\pi }{\lla \mc E(\varphi)\rra \log \vep^{-1}}+\frac{\lambda'\pi}{\lla \mc E(\varphi)\rra \log^{3/2} \vep^{-1}}+\frac{\lambda\pi}{\lla \mc E(\varphi)\rra \log^{2} \vep^{-1}},\label{def:llvintro}
\end{gather}
and we identify three parameter regimes arranged in order of increasing $\mu$ and $\lambda'$ such that each yields a different form of asymptotic representation. More specifically,
the first regime is defined by $\mu<1$, and the remaining two regimes correspond to $\mu=1$ and differ by the magnitude of $\lambda'$. 

A key feature of Theorem~\ref{thm:=0} is the adoption of the scaling $\log^{-1/2}\vep^{-1}$ through $\lv^{1/2}$, which stands in contrast to those used in previous solutions of related problems. Accordingly, while the first regime mentioned above recovers the Kasahara--Kotani law in \eqref{def:KK}, new limits emerge in the other two regimes. Moreover, when the third regime is applied, we obtain limiting expressions that capture essential features of the established limits under \eqref{nontrivial}. Thus, we interpret the third regime as the regime of \emph{criticality}, marked by two levels of phase transitions. For further details, see the discussion between Proposition~\ref{prop:Fvep=0} and Theorem~\ref{thm:=0}. 

The proof of Theorem~\ref{thm:=0}  employs a specialized method of \emph{asymptotic recursions}, described in Section~\ref{sec:KKregime}. This method also provides an alternative proof of the asymptotic representations for the Feynman--Kac formulas in \eqref{def:FK1} under \eqref{nontrivial} (see Section~\ref{sec:another}). Another proof of this kind in this paper is the proof of (\ref{eq:finalI:rad}) of Proposition~\ref{prop:5} for $\lla\varphi\rra=1$.

\item For $\lla \varphi\rra=1$, our main theorems are Theorems~\ref{thm:>0} and~\ref{thm:critres}: 
\begin{itemize}
\item Theorem~\ref{thm:critres} proves that, assuming \eqref{nontrivial},
\begin{align}\label{intro:thm:>0}
\lim_{\vep\to 0}G^{\lv(\oovarphi\uvep)_\vep}_q\{g\}(z)=\lim_{\vep\to 0}G^{\lv\varphi_\vep}_q\{g\}(z)
\end{align}
for all $z\neq 0$ and suitably large $q>0$, as the first limits recover the right-hand side of the following known formula: 
\begin{align}
\lim_{\vep\to 0} G_{q}^{\lv\varphi_\vep}\{g\} ( z)=G_{q}\{g\}(z)+\frac{2\pi}{\log (q/\beta)}G_{q}(z)G_{q}\{g\}(0).\label{finalgoal00}
\end{align}
Here, $(\oovarphi\uvep)_\vep$ rescales $\oovarphi\uvep$ as in \eqref{def:ffvep}, and $\oovarphi\uvep$, defined in \eqref{def:phiexp}, is a radially symmetric function transforming $\varphi$ and showing a first-order expansion in $\lv$ as a combination of the central mechanisms for proving the two laws. By these functions  $\oovarphi\uvep$, we regard Theorem~\ref{thm:critres} as a new approximation of the two-dimensional Laplacian singularly perturbed at the origin by using analytic characteristics of the two laws simultaneously. 

\item Theorem~\ref{thm:>0} presents the main tools used to prove Theorem~\ref{thm:critres}. The former theorem provides asymptotic formulas for the following time-changed exponential moments, which act as the ``disintegrated forms'' of $G^{\lv(\oovarphi\uvep)_\vep}_q\{g\}(z)$:
\begin{align}\label{multtime}
&\E^W_z[\e^{-q\vep^2 \tau^b_\ell+\lv A_{ \oovarphi\uvep}(\tau^b_\ell)}].
\end{align}
Here, $\lv$ satisfies \eqref{def:lvoointro} for \emph{any} given $\mu\in (0,\infty)$ and $\lambda\in \R$, and the time change $\tau^b_\ell$ is the inverse local time at level $b>0$ of the radial part $|W_t|$, which is a
two-dimensional Bessel process (${\bf BES}^{\bs 2}$).
\end{itemize} 
\end{itemize}

The theorems presented here have already spurred progress on related problems and opened new avenues of inquiry. In particular, the asymptotic recursion method introduced in this paper inspired the approach in~\cite{C:SHEAC}, where precise and explicit formulas were derived for the semimartingale characteristics of the two-dimensional stochastic heat equation at criticality. Furthermore, the radial symmetry technique for deriving the limits in \eqref{intro:thm:>0} and \eqref{finalgoal00} originated in \cite{C:MixDBG}. This approach subsequently inspired the skew-product diffusion framework in \cite{C:BES}, leading to the Feynman--Kac-type formula mentioned below \eqref{nontrivial}. Here, we complete and streamline that method. For future research, an interesting application would be to analyze the conditional laws in \eqref{Esscher} at the process level, focusing on the results for $\lla \varphi\rra=0$ and $\mu=1$ from Theorem~\ref{thm:=0}. An initial step could involve inverting the limiting Laplace transforms, as was done in \cite{ABD:Schrodinger,C:BES} for the case of \eqref{nontrivial}. Another question is whether our asymptotic results for the exponential moments in \eqref{multtime} within the supercritical regime ($\mu > 1$) remain valid upon restoring the original time variable $t$, as these results allow various scaling limits; see \eqref{lim:aux} for the corresponding asymptotic formula. For this, it is important to note a key distinction between the framework in \eqref{multtime} and those in \eqref{def:FK0} and \eqref{def:FK1}: the ``clock" $\tau^b_\ell$ stops whenever $|W_t|$ exists $b$. Nevertheless, our proof of Theorem~\ref{thm:critres} shows that allowing all $b > 0$ enables this restoration in the setting of \eqref{nontrivial}. \medskip 

\noindent {\bf Remainder of this paper.} Section~\ref{sec:mainresults} discusses in more detail the main theorems of this paper and the methods of proof. Most of the proofs are presented in Sections~\ref{sec:=0} and~\ref{sec:>0} according to $\lla \varphi\rra=0$ and $\lla\varphi\rra>0$, respectively.\qed  \medskip 

\noindent {\bf Frequently used notation.} 
We use $C(\Theta)$ to denote a strictly positive constant that depends only on $\Theta$, with such constants potentially varying between terms. In contrast, constants written as $\mathfrak{C}$ are special, and each carries a specific, fixed meaning. We write $A=\mathcal{O}(B)$ if $|A| \leq C|B|$ for some universal constant $C > 0$, where the upper bound involves $|B|$ rather than $B$ itself. Similarly, $A = \mathcal{O}_{\Theta}(B)$ if $|A| \leq C(\Theta)|B|$. Little-$\little{o}$ notations, denoted in the same style, are interpreted analogously. Finally, $\log^ab\,\defeq\,(\log b)^a$, where $\log$ uses base $\e$, and
\begin{align}
\|H\|_{h}&\,\defeq \,\sup\{|H(z)|;z\in D_h\}\quad (h:\R^2\to \R),
\label{def:Hh}
\end{align}
where $D_h$ is the smallest closed centered disc containing $\supp(h)$. \qed 

\section{Statement of main theorems}\label{sec:mainresults}
The main theorems in this paper study the cases of $\lla \varphi\rra=0$ and $\lla \varphi\rra=1$ separately, employing the scalings $\log^{-1/2}\vep^{-1}$ from the Kasahara--Kotani law in \eqref{def:KK} and $\log^{-1}\vep^{-1}$ from the Kallianpur--Robbins law in \eqref{def:KR}, respectively.  While our proofs for the two cases are largely independent, several relations exist between the methods, and each method features distinct characteristics, which we will discuss below. 

For the following discussion, recall once again that {\bf $\varphi$ is assumed to be bounded with compact support and satisfy $\lla |\varphi|\rra>0$}. {\bf Also, functions $g$, used as terminal conditions of Feynman--Kac formulas and the Green's functions, are assumed to be bounded nonnegative with $\lla g\rra> 0$}.  

\subsection{Occupation functions with zero integrals}\label{sec:KKregime}
We find this case particularly delicate. Beyond various technical issues arising from the lack of nonnegativity of $\varphi$, the nature of criticality is more intricate than anticipated. Specifically, based on the Kasahara--Kotani law in \eqref{def:KK} and the criterion in \eqref{KK:expmom}, one may expect that the critical limits of the associated Feynman--Kac formulas in \eqref{def:FK0} will follow from taking
\[
\mu_\vep\sim \sqrt{\frac{\pi}{\lla \mc E(\varphi)\rra \log \vep^{-1}}}.
\]
Recall that by the Kallianpur--Robbins law in \eqref{def:KR}, the choice of $\mu_\vep$ through \eqref{nontrivial} and \eqref{def:lvoointro} is an analogue that has been established for the critical limits under $\lla\varphi\rra=1$. It may thus seem that only a single phase transition occurs when $\lla \varphi\rra=0$. Nevertheless, our main theorem in this case reveals the existence of two phase transitions, together with asymptotic representations for the following Green's functions:
\begin{align}\label{def:F}
F_\infty(z)=F^{\vep,q}_\infty(z)\;\defeq \;G^{\llv^{1/2}\varphi_\vep}_q\{g\}(\vep z),
\end{align} 
where $\varphi_\vep$ and $G^f_\nu\{g\}(z)$ are defined by \eqref{def:ffvep} and \eqref{def:Gf}, and $\lv$ obeys \eqref{def:llvintro}.  The proof applies the method of asymptotic recursions discussed in Section~\ref{sec:intro}. 

In the remainder of Section~\ref{sec:KKregime}, we first explain how the asymptotic recursions emerge and how they are used to reveal the fundamental mechanisms behind the Green's functions, as the derivation may have broad implications, and we expect more to follow. After presenting the asymptotic recursions, we turn to the asymptotic representations in the main theorem. The exact nature of the criticality and phase transitions, however, involves additional technical details, which we address subsequently.

\subsubsection{Derivation of asymptotic recursions}\label{sec:ar-der}
The following lemmas, Lemma~\ref{lem:expansion=0} and Lemma~\ref{lem:1replace00}, play a central role in deriving the asymptotic recursions satisfied by $F_\infty$. For the statement of the first lemma, we set
\begin{align}
\mc K\{f\}(z)&\,\defeq \;\kappa \star (\varphi f)(z)=\int \kappa (z-z')\varphi(z')f(z')\d z',\label{def:K}\\
P_t(z,z')=P_t(z-z')&\,\defeq\, \frac{1}{2\pi t}\exp\left(-\frac{|z-z'|^2}{2t}\right),\label{def:Pt}\\
G_q(z,z')=G_q(z-z')&\,\defeq \int_0^\infty \e^{-q t}P_t(z,z')\d t, \label{def:G}
\end{align}
where $\kappa(\cdot)$ is defined in \eqref{def:energy}, $P_t(z,z')$ is the two-dimensional Brownian transition density, and with respect to $\chi\,\defeq\, (\mu,\lambda',\lambda)$ for given $\mu,\lambda',\lambda$ defining $\lv$ in \eqref{def:llvintro}, we fix
\[
\mbox{$\ol{\vep}(\chi)\in (0,1/4]$ such that }\lv>0,\; \forall\;\vep\in (0,\ol{\vep}(\chi)].
\]

\begin{lem}\label{lem:expansion=0}
For all $\vep\in (0,\ol{\vep}(\chi)]$ and $q\in(0,\infty)$ and $\|F^{\vep,q}_\infty\|_\varphi<\infty$, $F_\infty=F_\infty^{\vep,q}$ satisfies 
\begin{align}\label{exp:10}
F_\infty(z)=G_q\{g\}(\vep z)+ 
\mc T_\vep\{F_{\infty}\}(z),
\end{align}
where
\begin{align}\label{def:Tvep}
 \mc T_\vep\{h\}(z)\,\defeq\, 
\lv^{1/2}\int G_q( \vep z,\vep z')\varphi(z')h( z')\d z'.
\end{align}
Also, under Assumption~\ref{ass:q} stated below on the choice of $q$, for all small $\vep\in (0,\ol{\vep}(\chi)]$,
\begin{align}
\begin{split}\label{Finfty:expansion}
F_\infty(z)&=G_q\{g\}(0)
+
\frac{\lv^{1/2}}{\pi}\Biggl(\log \frac{2^{1/2}}{q^{1/2}\vep}-\EM\Biggr)\lla \varphi F_\infty\rra
+\lv^{1/2}\mc K\{ F_\infty\}(z)
+\mathcal O^F_\infty
\end{split}
\end{align}
uniformly in all $z\in \supp(\varphi)$,
where, for $\mathcal O_\Theta(B)$ defined at the end of Section~\ref{sec:intro},
\begin{align}
\mathcal O^F_\infty&\,\defeq\, \mc O_{\varphi,\chi}
[(1+q^{-1})\vep^{2/3}\|g\|_\infty]+\mathcal O_{\varphi,\chi}\{q\vep^{2}[\log (q^{1/2}\vep)^{-1}]\}\|F_{\infty}\|_{\varphi}.\label{def:Oinfty}
\end{align}
\end{lem}

\begin{lem}\label{lem:1replace00}
Under Assumption~\ref{ass:q}, the following holds for all small $\vep\in (0,\ol{\vep}(\chi)]$:
\begin{align}
\lla \varphi F_\infty\rra&=\llv^{1/2}\lla \mathcal E(\varphi)F_\infty\rra 
+\mathcal O^F_\infty,
\label{1replace0}\\
\begin{split}
F_\infty( z'')-F_\infty( z')&=\llv^{1/2}\mc K\{F_\infty\}(z'')-\llv^{1/2}\mc K\{ F_\infty\}(z') +\mathcal O^F_\infty,
\label{F1replace0}
\end{split}
\end{align}
uniformly in $z'',z'\in \supp(\varphi)$. Here, $\mathcal E\{f\}(z')\,\defeq\,\varphi(z')\int (\pi^{-1}\log|z'-z''|^{-1})\varphi(z'')\d z''$, as defined in \eqref{def:energy}.
\end{lem}

See Section~\ref{sec:expansion} for the proofs of these two lemmas. In particular, the proof of Lemma~\ref{lem:1replace00} is covered in Lemma~\ref{lem:1replace0} by taking $X\equiv F$ and $n\equiv \infty$ therein. Additionally, the assumption regarding the choice of $q$, stated below, pertains to a priori bounds on the Green's functions, employing the notation $\|H\|_h$ as defined at the end of Section~\ref{sec:intro}. Note that it does \emph{not} assume bounds on the limiting values:

\begin{ass}\label{ass:q} 
$q\in (0,\infty)$ is such that $\|F^{\vep,q}_\infty\|_{\varphi}<\infty$ for all small $\vep\in (0,\ol{\vep}(\chi)]$.  \qed 
\end{ass}

We use Lemma~\ref{lem:expansion=0} mainly for the first-order asymptotic expansion of $F_\infty^{\vep,q}$ in $\vep$, given by \eqref{Finfty:expansion}, while Lemma~\ref{lem:1replace00} allows us to asymptotically \emph{close} this first-order expansion. Specifically, to achieve the asymptotic closure, we use \eqref{1replace0} to substitute $\lla \varphi F_\infty\rra$ in \eqref{Finfty:expansion} with $\lv^{1/2}\lla \mc E(\varphi)F_\infty\rra$. Next, we apply \eqref{F1replace0} for a further asymptotic replacement, so that for all $z \in \supp(\varphi)$,
 \begin{align*}
 \lla \mathcal E(\varphi)F_\infty\rra&=
\lla \mathcal E(\varphi)\rra  F_\infty(z)+\lv^{1/2}\lla \mathcal E(\varphi)\mc K\{ F_\infty\}\rra-\llv^{1/2}\lla \mathcal E(\varphi)\rra\,\!\mc K\{F_\infty\}(z)+\mathcal O^F_\infty.
 \end{align*}
Moreover, $F_\infty$ in the second and third terms on the right-hand side can be replaced by the \emph{number} $F_\infty(z)$ by using \eqref{F1replace0} again. Such applications of \eqref{F1replace0}
can be repeated for the term $\lv^{1/2}\mc K\{ F_\infty\}(z)$ on the right-hand side of \eqref{Finfty:expansion}. 

The overall effect of the above replacements is the following asymptotic recursion:
for all $z\in \supp(\varphi)$,
 \begin{align}\label{Finfty:X}
F_\infty( z)=G_q\{g\}(0)+ (\mu-s_\vep  \mathfrak C) F_\infty( z)+\mathcal O_{\varphi,\chi,q}(s'_\vep \|F_\infty\|_\varphi)+\mathcal O_{\varphi,\chi,q,g}(s''_\vep).
\end{align}
Here, $\mathfrak C>0$ is an explicitly expressible constant, and $s_\vep,s'_\vep,s''_\vep\to 0$ with
\begin{align}\label{s'svep}
s'_\vep/s_\vep\to 0\quad\mbox{ as $\vep\to 0$.}
\end{align} 
In particular, for $\mu=1$, \eqref{Finfty:X} after rearrangement gives
 \begin{align}
 \begin{split}\label{Finfty:X2}
&( s_\vep \mathfrak C) F_\infty(z)= [G_q\{g\}(0)+\mathcal O_{\varphi,\chi,q,g}(s''_\vep)]+\O_{\varphi,\chi,q}(s'_\vep \|F_\infty\|_\varphi)\\
&\Longrightarrow 
\lim_{\vep\to 0} s_\vep F_\infty(z)=\frac{G_q\{g\}(0)}{\mathfrak C}\mbox{ uniformly over } z\in \supp(\varphi).
\end{split}
\end{align}
In more detail, the deduction in \eqref{Finfty:X2} requires the following consideration, since Assumption~\ref{ass:q} is weaker than $\limsup_{\vep\to 0}\|F^{\vep,q}_\infty\|_{\varphi}<\infty$: 

\begin{lem}\label{lem:ss}
Under \eqref{s'svep} and the sufficient condition in \eqref{Finfty:X2}, we have the convergence
$\O_{\varphi,\chi,q}(s'_\vep \|F_\infty\|_\varphi)\to 0$ as $\vep\searrow 0$.
\end{lem}
\begin{proof} 
Since $g\geq 0$ and $\lla g\rra> 0$ by assumption, it holds that $G_q\{g\}>0$ and $F_\infty> 0$. Hence, by \eqref{s'svep}, the equation in the sufficient condition of \eqref{Finfty:X2}, after rearrangement, implies $\limsup_{\vep\to 0}s_\vep\|F_\infty\|_\varphi<\infty$. This property leads to  $\O_{\varphi,\chi,q}(s'_\vep \|F_\infty\|_\varphi)=\O_{\varphi,\chi,q}((s'_\vep/s_\vep)\cdot s_\vep \|F_\infty\|_\varphi)\to 0$ by \eqref{s'svep} again. 
\end{proof}

Proposition~\ref{prop:Fvep=0} below proves asymptotic representations of the Green's functions by summarizing the precise values of $\mathfrak C$ and orders of $s_\vep$ in \eqref{Finfty:X} for two cases assuming $\mu=1$. The precise values of $\mathfrak C$ for $\mu\in (0,1)$ are also obtained. 
 See Section~\ref{sec:multi} for the proof. In particular, this proof is where we find the form of criticality more delicate than anticipated, as taking $\mu=1$ can lead to two different explicit expressions for the asymptotic expressions, given by \eqref{lim:=0-2} and \eqref{lim:=0-3}. 

\begin{prop}\label{prop:Fvep=0}
Set
\begin{align}
\mathfrak C(\varphi)&\,\defeq \,\frac{\pi^{1/2}\lla \mc E(\varphi)\mc K\{\1\}\rra}{\lla \mc E(\varphi)\rra^{3/2}},\label{def:Cphi}\\
\mathfrak C_{\lambda'}&\,\defeq -\lambda'-\mathfrak C(\varphi),\label{def:Clambda'}\\
\mathfrak C_{\lambda,q}&\,\defeq -\lambda -\log \frac{2^{1/2}}{q^{1/2}}+\EM+\frac{3\pi\lla \mc E(\varphi)\mc K\{\1\}\rra^2}{2\lla \mc E(\varphi)\rra^{3} }-\frac{\pi\lla \mc E(\varphi)\mc K^2\{\1\}\rra }{\lla \mc E(\varphi)\rra^2 }.\label{def:Clambdaq}
\end{align} 
Then under Assumption~\ref{ass:q}, the following limits hold uniformly in $z\in \supp(\varphi)$: 
\begin{subequations}\label{lim:=0}
\begin{align}
\begin{split}\label{lim:=0-1}
\displaystyle \lim_{\vep\to 0}G^{\lv^{1/2}\varphi_\vep}_q\{g\}(\vep z)=\frac{G_q\{g\}(0)}{1-\mu},&\quad  \mbox{if $\mu\in (0,1)$},
\end{split}\\
\begin{split}
\displaystyle \lim_{\vep\to 0}\frac{G^{\lv^{1/2}\varphi_\vep}_q\{g\}(\vep z)}{\log^{1/2} \vep^{-1}}
=\frac{G_q\{g\}(0)}{\mathfrak C_{\lambda'}},&\quad  \mbox{if $\mu=1$ }\mathfrak C_{\lambda'}>0,\label{lim:=0-2}
\end{split}\\
\begin{split}
\displaystyle \lim_{\vep\to 0}\frac{G^{\lv^{1/2}\varphi_\vep}_q\{g\}(\vep z)}{\log \vep^{-1}}=\frac{G_q\{g\}(0)}{\mathfrak C_{\lambda,q}}, &\quad \mbox{if $\mu=1$, }\mathfrak C_{\lambda'}=0,
\label{lim:=0-3}
\end{split}
\end{align}
\end{subequations}
where the limit in \eqref{lim:=0-3} also requires that $q$ is large enough such that $\mathfrak C_{\lambda,q}>0$. 
\end{prop}

The three limits in \eqref{lim:=0} have the following features. The last two limits \eqref{lim:=0-2} and \eqref{lim:=0-3} are new and impose requirements on $\mathfrak C_{\lambda'}$ that depend in a nontrivial way on $\varphi$. This is because the values of $\mathfrak C(\varphi)$ for all $\varphi$ with $\lla\varphi\rra=0$ and $\lla |\varphi|\rra=0$ span $\R$, as we demonstrate in Section~\ref{sec:range}. Additionally,
the first two limits \eqref{lim:=0-1} and \eqref{lim:=0-2} both show the essential role of the exponential moments with $g\equiv 1$. For example, the limit in \eqref{lim:=0-1} can be restated as
\begin{align}\label{eq:expmom-g}
\lim_{\vep\to 0}G^{\lv^{1/2}\varphi_\vep}_q\{g\}(\vep z)= \frac{G_q\{g\}(0)}{1-\mu}=\int_0^\infty \e^{-qt} \E^W_0[\e^{\sqrt{\mu}\sqrt{2} B_{\mathbf e}}]\E^W_0[g(W_t)]\d t.
\end{align}
In contrast, the third limit in \eqref{lim:=0} is of a different nature. The prefactor $1/\mathfrak C_{\lambda,q}$ is time-dependent due to its explicit dependence on $q$. Moreover, we interpret this limit as critical, since $1/\mathfrak C_{\lambda,q}$ in \eqref{lim:=0-3} can be expressed as $2/(\log q+\tilde{a})$. This limit mirrors the nontrivial limits previously established under \eqref{nontrivial}, to be recalled in \eqref{finalgoal00}, where the prefactors $2/\log (q/\beta)=2/(\log q+\log 1/\beta)$ are the main characteristics.

\subsubsection{Main theorem for $\bs \lla\bs \varphi\bs \rra\bs =\bs 0$}
The theorem is presented below. The three conditions on $\mu$, $\mathfrak C_{\lambda'}$, and $\mathfrak C_{\lambda,q}$ in \eqref{lim:=0} correspond to those in this theorem, with the added requirement that $\lla \mc E(\varphi)\mc K\{\1\}\rra\geq 0$. Furthermore, the two phase transitions introduced at the start of Section~\ref{sec:KKregime} are now characterized by the conditions on the three limits in this theorem.

\begin{thm}\label{thm:=0}
Let $\varphi$ be bounded with a compact support and satisfy $\lla \varphi\rra=0$ with $\lla |\varphi|\rra> 0$, $g$ be nonnegative bounded with $\lla g\rra> 0$, and $\llv$ be given by
\begin{align}\label{def:llv}
\begin{split}
\llv \equiv \frac{\mu\pi }{\lla \mc E(\varphi)\rra \log \vep^{-1}}+\frac{\lambda'\pi}{\lla \mc E(\varphi)\rra \log^{3/2} \vep^{-1}}+\frac{\lambda\pi}{\lla \mc E(\varphi)\rra \log^{2} \vep^{-1}}
\end{split}
\end{align}
for fixed $\mu\in (0,1]$ and $\lambda',\lambda\in \R$. Then  for all large enough $q>0$, the following holds:
\begin{itemize}
\item \eqref{lim:=0-1} if $\mu<1$;
\item \eqref{lim:=0-2} if $\mu=1$, $\mathfrak C_{\lambda'}>0$, and $\lla \mc E(\varphi)K\{\1\}\rra\geq 0$;
\item \eqref{lim:=0-3} if $\mu=1$, $\mathfrak C_{\lambda'}=0$, and $\lla \mc E(\varphi)K\{\1\}\rra\geq  0$.
\end{itemize}
Moreover, these limits
hold uniformly on compacts in $z\in \R^2$.
\end{thm}

Theorem~\ref{thm:=0} does not follow immediately from Proposition~\ref{prop:Fvep=0} because we must still validate the limits in Proposition~\ref{prop:Fvep=0} by proving the existence of $q$ satisfying Assumption~\ref{ass:q}. The issue in verifying this assumption is that $\varphi$ is not nonnegative. For example, the Gr\"onwall-type bound for an analogue of $F_\infty$ in \cite[Lemma~5.5]{C:DBG} does not appear to extend to the case of Theorem~\ref{thm:=0} to yield sharp bounds. To address this, we revisit \eqref{exp:10}, considering the following Picard iteration for approximate Green's functions which are easier to bound:
\begin{align}
\begin{split}
F_n(z)\equiv 0,\;n=-1,-2,\cdots,\mbox{ and }
 F_n(z)\,\defeq\,
G_q\{g\}(\vep z)+\mc T_\vep\{F_{n-1}\}(z),\; n\in \Bbb Z_+,\label{def:Fc}
\end{split}
\end{align}
where the definitions for $n\in \Bbb Z_+$ are in the inductive sense.
Then, observe the following:
\begin{itemize}
\item For all $n\in \Bbb Z_+$, $\|F_n\|_{\varphi}<\infty$ and $F_{n}(z)=\sum_{j=0}^n\mc T_\vep^j \{z'\mapsto G_q\{g\}(\vep z')\}( z)$.
\item If, for fixed $\vep>0$, $\{F_n\}$ converges to $F$ as $n\to\infty$ under $\|\cdot\|_{\varphi}$, then by dominated convergence, $F(z)=G_q\{g\}(\vep z)+\mc T_\vep\{ F\}(z)$ on $\supp(\varphi)$. Note that this equation takes the same form as \eqref{exp:10}.
\end{itemize}

The following proposition completes our quest to bound $F_\infty$; see Section~\ref{sec:geo} for the proof. The condition $\lla \mc E(\varphi)\mc K\{\1\}\rra\geq 0$ mentioned above arises in one step of its proof (see Lemma~\ref{lem:recF}). However, we did not explore the consequence of $\lla \mc E(\varphi)K\{\1\}\rra< 0$, as this inequality could be reversed by substituting $\varphi$ with $-\varphi$.

\begin{prop}\label{prop:gbdd}
{\rm (1$\cc$)} Set $D_n\,\defeq\, F_n-F_{n-1}$ for all $n\in \Bbb Z$, and $\chi\,\defeq\, (\mu,\lambda',\lambda)$.
There exists $q_1(\varphi,\chi)\in(0,\infty)$ 
such that for all $q\geq q_1(\varphi,\chi)$,   
$\mathfrak C_{\lambda,q}>0$ and the following bounds hold for all sufficiently small $\vep\in (0,\ol{\vep}(\chi)]$ depending on $q$ and for all $n\in \Bbb Z_+$: 
\begin{align*}
\|D_n\|_\varphi\leq 
\begin{cases}
\displaystyle C_0\left|1-\left(\frac{1-\mu}{8}-\frac{C_1}{\log^{1/2} \vep^{-1}}
\right)\right|^n,&\mbox{if $\mu\in (0,1)$, }\\
\vspace{-.2cm}\\
\displaystyle C_0\left|1-\frac{C_2}{\log^{1/2} \vep^{-1}}
\right|^n,&\mbox{if $\mu=1$, $\lla \mc E(\varphi)\mc K\{\1\}\rra\geq 0 $, $\mathfrak C_{\lambda'}>0$},\\
\vspace{-.2cm}\\
\displaystyle C_0\left|1-\frac{C_3}{\log \vep^{-1}}
\right|^n,&\mbox{if }\mu=1,\,\lla \mc E(\varphi)\mc K\{\1\}\rra\geq 0,\, \mathfrak C_{\lambda'}=0,
\end{cases}
\end{align*}
where $C_0=C(\varphi,\chi,q,g)$, and $C_j=C(\varphi,\chi,q)$ for $j=1,2,3$. \medskip

\noindent {\rm (2$\cc$)} Moreover, for all $q\geq q_1(\varphi,\chi)$ and sufficiently small $\vep\in (0,1)$ depending on $q$, $\|F_\infty\|_\varphi<\infty$ and the $\|\cdot\|_\varphi$-limit $F$ of $\{F_n\}$ exists and satisfies $F=F_\infty=\sum_{j=0}^\infty\mc T_\vep^j \{z'\mapsto G_q\{g\}(\vep z')\}$ on $\supp(\varphi)$.  
\end{prop}

\begin{proof}[End of the proof of Theorem~\ref{thm:=0}]
By the strong Markov property of Brownian motion at the first hitting time of $\supp(\varphi_\vep)$ when the initial condition is $\vep z\notin \supp(\varphi_\vep)$, it suffices to prove that the convergences in \eqref{lim:=0} hold uniformly in $z\in \supp(\varphi)$. In this case, for $q_1(\varphi,\chi)$ from Proposition~\ref{prop:gbdd}, any $q\geq q_1(\varphi,\chi)$ satisfies Assumption~\ref{ass:q} by the geometric bounds for $\|D_n\|_\varphi$ and the formula  $\sum_{n=0}^\infty(1-\rho_0)^n=\rho_0^{-1}$, $\rho_0\in (0,1)$. The required result now follows from Proposition~\ref{prop:Fvep=0}.
\end{proof}

\subsection{Occupation functions with unit integrals}\label{sec:KRreg}
To make the possible connections between the two laws and the limits of $G_{q}^{\lv\varphi_\vep}\{g\} (z)$ under the condition of \eqref{nontrivial} more precise, we focus on the {\bf radial part} $\overline{f}$ and the {\bf radially asymmetric part} $\widehat{f}$ of a function $f$. Here,
\begin{gather}
\overline{f}(z)\,\defeq\, \frac{1}{2\pi}\int_{-\pi}^\pi f(|z|\e^{\i \theta})\d \theta\quad (\i\,\defeq\,\sqrt{-1}),\label{def:ophi}\\
\widehat{f}(z)\,\defeq\, f(z)-\overline{f}(z),\label{def:hphi}
\end{gather}
where we view the two-dimensional space as $\Bbb C$. 
A straightforward approach to involving these parts may be to examine scaling limits of the additive functionals $A_{\overline{\varphi}_\vep}(t)$ and $A_{\widehat{\varphi}_\vep}(t)$ from the following decomposition:
 \begin{align}\label{CFH}
 A_{\varphi_\vep}(t)=A_{\overline{\varphi}_\vep}(t)+A_{\widehat{\varphi}_\vep}(t).
\end{align}
This approach is motivated by the condition $\lla \varphi\rra=1$ from \eqref{nontrivial}, which implies
$\lla\overline{\varphi}\rra =1$ and $ \lla \widehat{\varphi}\rra=0$, so $\overline{\varphi}$ and $\widehat{\varphi}$ become suitable for applying the two laws. The main difficulty in deriving the asymptotic forms of $G_{q}^{\lv\varphi_\vep}\{g\} ( z)$ using \eqref{CFH}, however, is that the associated $\lv$ scales according to \eqref{lvasymp:KR}, whereas demonstrating the convergence of $A_{\widehat{\varphi}_\vep}(t)$ requires a fundamentally different scaling as in \eqref{lvasymp:KK}; see also \cite{CFH}. 

Rather than relying on \eqref{CFH}, our main theorems for the case $\lla \varphi\rra=1$, as outlined  at the end of Section~\ref{sec:intro}, show that the limits are based on the following functions:
\begin{align}\label{def:phiexp}
\oovarphi\uvep\,\defeq\; \overline{\varphi}+\lv\widecheck{\varphi},
\end{align} 
where, by using $\mathcal E(\cdot)$ defined in \eqref{def:energy}, 
\begin{align}
\widecheck{f}(z)&\;\defeq \;\overline{\mathcal E(\widehat{f})}(z)=\frac{1}{2\pi}\int_{-\pi}^\pi \left(-\frac{1}{\pi}\widehat{f}(|z|\e^{\i \theta})\int_{\Bbb C}\widehat{f}(z')\log ||z|\e^{\i \theta}-z'|\d z'\right)\d \theta.\label{def:cphi}
\end{align}
The above formulation of $\oovarphi\uvep$ is motivated by the fact that Kasahara and Kotani~\cite{KK} used the components $\overline{\varphi}$ and $\widecheck{\varphi}$ as occupation functions in their original process-level proofs of the two laws \eqref{def:KR} and \eqref{def:KK}. In contrast, we will use \eqref{def:phiexp} by combining $\overline{\varphi}$ and $\widecheck{\varphi}$, rather than applying them separately. Our analysis will also be conducted at the expectation level and proceed independently of \cite{KK}.

\subsubsection{Disintegration of Green's functions}\label{sec:DGF}
Theorem~\ref{thm:excursion}, presented below, is the central tool for establishing our main results for the case of $\lla \varphi\rra=1$, and its applicability extends beyond the critical case. This theorem is based on excursion theory and applies to any general Markov process $\{X_t\}$ that meets Assumption~\ref{ass:X}; see \cite[Chapter~IV]{Bertoin} and \cite{Blumenthal} for excursion theory. To state the theorem and the assumption, we define
\begin{align}
U^f_\nu\{g\}(a)\;\defeq \int_0^\infty \e^{-\nu t}\E^X_a[\e^{A^X_f(t)}g(X_t)]\d t\label{def:Uf}
\end{align}
for the Green's function associated with the following additive functional:
\begin{align}
A^X_f(t)\,\defeq \int_0^t f(X_s)\d s.\label{def:AfX}
\end{align}
Here and in what follows, $\P^X_a[\cdot]\,\defeq\,\P[\cdot|X_0=a]$, and $\E^X_a$ is similarly defined. The resolvents of the transition probabilities are given by the following operators: 
\begin{align}\label{def:U0}
U_\nu\,\defeq\, U^0_\nu.
\end{align}

\begin{ass}\label{ass:X}
With respect to a filtration $\{\G_t\}$ satisfying the usual conditions, $\{X_t\}$ is a $\{\G_t\}$-strong Markov process with c\`adl\`ag paths such that the state space $E$ is 
a complete separable metrizable topological space, and the following conditions hold:
\begin{itemize}
\item [(a)] Any point $b\in E$ is regular: $\P^X_b(\inf\{t>0;X_t=b\}=0)=1$.
\item [(b)] Any point $b\in E$ is instantaneous: $\P^X_b(\inf\{t\geq 0;X_t\neq b\}=0)=1$.
\item [(c)] There exists a $\sigma$-finite measure $\mathbf m$ on $(E,\B(E))$ such that for all $a\in E$ and $\nu\in(0,\infty)$, the $\nu$-resolvent measures 
$\Gamma\mapsto  U_\nu\{\1_\Gamma\}(a)$
are absolutely continuous with respect to $\mathbf m$.
Moreover, we can choose the densities $b\mapsto U_\nu(a,b)$ of $\Gamma\mapsto  U_\nu\{\1_\Gamma\}(a)$ such that $U_\nu(a,b)$ is jointly continuous in $(a,b)\in E\times E$.
\item [(d)] For any initial condition $a\in E$, there exists a family $\{L^b_t;b\in E,t\geq 0\}$ of random variables jointly continuous in $(b,t)$ such that the occupation times formula holds: 
\begin{align}\label{oct}
\int_0^t f(X_s)\d s=\int_E f(b) L^b_t\mathbf m(\d b)
\end{align}
for all nonnegative $f\in \B(E)$,
with $\P^X_a$-probability one.\hfill $\blacksquare$
\end{itemize}
\end{ass}

Assumption~\ref{ass:X} ensures that the excursion theory applies to $\{X_t\}$ and that local times are well-defined at every level $b$. This assumption also enables the following properties:
\begin{itemize}
\item At any given level $b\in E$, the excursions are ``Brownian-motion-like'' in the sense of no isolated visits of $b$ by condition (a) or holding periods by condition (b).
\item Condition (c) implies that the inverse local time under $\P_b$ is a subordinator with Laplace exponent $U_\nu(b,b)^{-1}$ \cite[Theorem~3.6.3, p.85]{MR:Markov}. See Theorem~\ref{thm:excursion} (2$\cc$) for the strict positivity of $U_\nu(b,b)$.
\item Condition (d)  implies that $\{L^b_t;t\geq 0\}$ is a Markovian local time at level $b$  \cite[Proposition~5 on p.111]{Bertoin}. See \cite[Section~3.7, pp.98--105]{MR:Markov} for sufficient conditions of (d). 
\end{itemize}

To state Theorem~\ref{thm:excursion}, we will also use the inverse local times:
\begin{align}
\tau^b_\ell\,\defeq\inf\{t\geq 0;L^b_t>\ell\},\quad \ell\geq 0,\label{def:invLT}
\end{align}
and the first passage functionals:
\begin{align}\label{def:Pf}
Q^f_{\nu,b}(a)&\,\defeq \,\E^X_a[\e^{-\nu T_b+A^X_f(T_b)};T_b<\infty]\quad(Q_{\nu,b}\,\defeq\, Q^0_{\nu,b}).
\end{align}
Here, $\E[Z;\Gamma]\,\defeq\, \E[Z\1_\Gamma]$, and
\begin{align}\label{def:TbX}
T_b=T_b(X)&\,\defeq \inf\{t\geq 0;X_t=b\}.
\end{align}

\begin{thm}\label{thm:excursion}
Let $\{X_t\}$ satisfy Assumption~\ref{ass:X}. Fix a real-valued, locally bounded function $f\in \B(E)$, a nonnegative function $g\in \B(E)$, and $\nu\in (0,\infty)$. 
\begin{itemize}
\item [\rm (1$\cc$)] For all $a\in E$, it holds that
\begin{align}
U^f_\nu\{g\}(a)
&=\int_E \frac{Q^f_{\nu,b}(a)U_\nu(b,b)}{\max\big\{1-U_\nu\{fQ^f_{\nu,b}\}(b),0\big\}}g(b)\mathbf m(\d b),\label{formula:excursion}\\
\E_a^X[\e^{-\nu\tau^b_\ell+A^X_f(\tau^b_\ell)}]&=Q^f_{\nu,b}(a)\exp\Biggl\{-\ell\Biggl(\frac{1-U_{\nu}\{f Q^{f}_{\nu,b}\}(b)}{U_\nu(b,b)}\Biggr)\Biggr\},\quad \ell\geq 0,\label{kernel:expexp}
\end{align}
where the usual convention of $1/0=\infty$ and $\infty\cdot 0=0$ applies to the right-hand side of \eqref{formula:excursion} when $\max\{1-U_\nu\{fQ^f_{\nu,b}\}(b),0\}=0$. 

\item [\rm (2$\cc$)] For all $b\in E$, $Q^f_{\nu,b}$ in \eqref{formula:excursion} satisfies the following recursive equation: 
\begin{align}\label{psiab:dec0}
Q^f_{\nu,b}(a)+Q_{\nu,b}(a)U_\nu\{f Q^f_{\nu,b}\}(b)=Q_{\nu,b}(a)+U_\nu\{fQ^f_{\nu,b}\}(a),\quad a\in E.
\end{align}
Also, $Q_{\nu,b}(a)=U_\nu(a,b)/U_\nu(b,b)$, where $U_\nu(b,b)>0$.
\end{itemize}
Here, \eqref{formula:excursion} and \eqref{psiab:dec0} use the property that $U_\nu\{fQ^f_{\nu,b}\}(a)$ is well-defined and $(-\infty,\infty]$-valued  for all $a,b\in E$. Hence, whenever $U_\nu\{fQ^f_{\nu,b}\}(b)<\infty$, 
\begin{align}\label{psiab:dec}
Q^f_{\nu,b}(a)=Q_{\nu,b}(a)\times [1-U_\nu\{f Q^f_{\nu,b}\}(b)]+U_\nu\{fQ^f_{\nu,b}\}(a),
\end{align}
which follows by rearranging \eqref{psiab:dec0}.
\end{thm}

The relations below support parts (1$\cc$) and (2$\cc$) of Theorem~\ref{thm:excursion}. For part (1$\cc$), \eqref{formula:excursion} follows from \eqref{kernel:expexp} via the following distintegration formula:
\begin{align}
U^f_\nu\{g\}(a)&=\int_E\int_0^\infty \E^X_{a}[\e^{-\nu\tau^b_\ell+A^X_f(\tau^b_\ell)}]\d \ell g(b)\mathbf m(\d b).\label{dec:Gf}
\end{align}
In part (2$\cc$), \eqref{psiab:dec0} gives an equation for $Q^f_{\nu,b}$ in terms of the resolvents and $Q_{\nu,b}$ as coefficients, with $Q_{\nu,b}(a)$ expressible as ratios of resolvent densities. This relation allows us to view Theorem~\ref{thm:excursion} as providing explicit solutions to the Green's functions $U^f_\nu\{g\}(a)$ in terms of the resolvents $U_\nu$.

\begin{rmk}[Kac's original theorem]
\label{rmk:excursion}
For the special case of a one-dimensional Brownian motion $\{B_t\}$, \eqref{formula:excursion} may be reminiscent of Kac's original theorem~\cite{Kac} for the Feynman--Kac formula. The theorem from \cite{Kac} can be reformulated as follows: let $\nu \geq 0$, and let $f$ and $g$ be nonnegative with $f$ locally bounded:
\begin{align}\label{Kac}
\int_0^\infty \e^{-\nu t}\E^B_0[\e^{-\int_0^t f(B_s)\d s}g(B_t)]\d t=\int_{-\infty}^\infty G^{-f}_\nu(0,b)g(b)\d b.
\end{align}
The Green's function $b\mapsto G^{-f}_\nu(0,b)$ is the unique solution to
\[
\frac{1}{2}G''(b)-[\nu+f(b)]G(b)=0,\quad b\neq 0,
\]
subject to the following three conditions: (i) $G'(b)$ exists for all $b\neq 0$ and is uniformly bounded; (ii) $G$ vanishes at $\pm \infty$; and (iii) the following jump condition holds:
\begin{align}\label{cond:jump}
\lim_{b\searrow 0}G'(b)-\lim_{b\nearrow 0}G'(b)=-2.
\end{align}
Moreover, in the context of \eqref{Kac}, by Tanaka's formula, one may choose $\mathbf m(\d b)$ in \eqref{formula:excursion} to be $\d b$. While excursion theory provides an alternative proof of \eqref{Kac} in \cite{JPY}, we will clarify at the start of Section~\ref{sec:level} how our proof of Theorem~\ref{thm:excursion} differs.
\hfill $\blacksquare$ 
\end{rmk}

\subsubsection{Two main theorems for $\bs \lla\bs \varphi\bs \rra\bs =\bs 1$}\label{sec:=1main}
Theorem~\ref{thm:>0}, which builds on Theorem~\ref{thm:excursion}, provides the first main theorem and serves as our tool for proving the second main theorem, Theorem~\ref{thm:critres}, discussed later. The central part of Theorem~\ref{thm:>0}, specifically part (2$\cc$), is established for bivariate exponential moments of time-changed additive functionals as in \eqref{kernel:expexp} beyond the critical case, using the two-dimensional Bessel process ($\BES^2$). For the statement of Theorem~\ref{thm:>0}, we specialize the notation introduced before Theorem~\ref{thm:excursion} as follows: $R_t(a,b)$ denotes the transition densities of $ \BES^2$, and $V_\nu(a,b)$ denotes the densities of the resolvents. Hence,
\begin{align}\label{UG:polar}
\begin{split}
R_t(a,b)&= b\int_{-\pi}^\pi P_t(a,b\e^{\i \theta})\d \theta,\quad
V_{\nu}(a,b)=\int_0^\infty \e^{-\nu t}R_t(a,b)\d t,\quad a,b\geq 0,
\end{split}
\end{align}
where $P_t(z,z')$ is defined in \eqref{def:Pt}. We also set 
\begin{gather}
V_\nu\{h\}(a)\,\defeq \int_0^\infty V_\nu(a,b)h(b)\d b,\label{def:V}\\
T_b=T_b(\rho)\,\defeq\inf\{t\geq 0;\rho_t=b\},\quad 
R^f_{\nu,b}(a)\,\defeq\,\E^\rho_a[\e^{-\nu T_b+A_f(T_b)}],\quad R_{\nu,b}\,\defeq\, R^0_{\nu,b},\label{def:TR}
\end{gather}
where $\{\rho_t = |W_t|\}$ is a version of the $\BES^2$ process, and $f(\cdot)$ will be chosen to be radially symmetric so that, 
for $A_f(t)$ as defined in \eqref{def:Af}, we also have
\[
A_f(t)=\int_0^t f(\rho_s)\d s. 
\]
In this context, the semimartingale local times $\{L^b_t\}$ of $\{\rho_t\}$ at levels $b \in (0, \infty)$ exist as nonzero processes. Specifically, these local times are defined via Tanaka's formula \cite[(1.2) Theorem on p.222]{RY} for the SDE $\mathrm{d}\rho_t = (2\rho_t)^{-1}\mathrm{d} t + \mathrm{d} B_t$, where $B$ denotes a one-dimensional standard Brownian motion.

\begin{thm}[First main theorem for $\bs \lla\bs \varphi\bs \rra\bs =\bs 1$]\label{thm:>0}
{\rm (1$\cc$)} Assumption~\ref{ass:X} is satisfied by $\{X_t\}\equiv \{\rho_t\}$ and $E\equiv (0,\infty)$. Hence, Theorem~\ref{thm:excursion} applies. \medskip 

\noindent {\rm (2$\cc$)}
Fix $q\in (0,\infty)$ and a bounded $\varphi$ with compact support with $\lla \varphi\rra=1$, and set
 \begin{align}
\Lambda_\vep\equiv  \frac{\mu\pi}{\log \vep^{-1}}+\frac{ \lambda\pi}{\log^2 \vep^{-1}},\quad \mu\in (0,\infty),\;\lambda\in \R.
\label{def:lvo}
\end{align}
Then for fixed $a,b,\mu,\ell\in (0,\infty)$ and $\lambda\in \R$,
the following holds as $\vep\to 0$:
\begin{align}\label{lim:aux}
\begin{split}
&\E^{\rho}_{a}[\e^{-q\vep^2\tau^b_\ell+A_{\lv \oovarphi\uvep}(\tau^b_\ell)}]\\
&\quad =[1+\little{o}(1)]
\exp\Biggl\{-\frac{\ell}{b} \cdot \frac{1-\mu-\frac{A_{\mu,\lambda,q}(b,b)}{\log \vep^{-1}}+\little{o}\big(\frac{1}{\log\vep^{-1}}\big) }{
2\log \vep^{-1}+\mathcal O(1)}\Biggr\},
\end{split}
\end{align}
where 
\begin{align}\label{def:Azb}
\begin{split}
A_{\mu,\lambda,q}(a,b)
 &\,\defeq\,(\mu^2-\mu)
\int \log^-(b/|z'|)  \overline{\varphi}(z')\d z'+\int \biggl(\log \frac{|b-z'|^{\mu^2}}{|a-z'|^\mu}\biggr)
\overline{\varphi}(z')\d z'\\
&\quad +\mu\biggl(\log \frac{2^{1/2}}{q^{1/2}} - \gamma_{\sf EM}\biggr)+\lambda+\mu^2\pi \lla \mathcal E(\varphi)\rra,
\end{split}
 \end{align}
using $\lla \cdot\rra$ and $\mathcal E(f)$ defined in \eqref{def:phig} and \eqref{def:energy}, and the Euler--Mascheroni constant $\EM$.
\end{thm}

\begin{proof}[Proof of Theorem~\ref{thm:>0} (1$\cc$)]
Let $\{X_t\}\equiv \{\rho_t\}$ with state space $E = (0, \infty)$. This choice of state space, as opposed to $E \equiv \R_+$, is justified because the origin is polar for $\BES^2$. To meet conditions (c) and (d) of Assumption~\ref{ass:X}, set $\mathbf{m}(\mathrm{d}b) \equiv \mathrm{d}b$, $U_{\nu}(a, b) \equiv V_{\nu}(a, b)$ as in \eqref{UG:polar}, and define the local times $\{L^b_t\}$ according to Tanaka's formula as previously mentioned.
\end{proof}

Section~\ref{sec:3asymp} presents a more detailed version of Theorem~\ref{thm:>0} (2$\cc$) as Proposition~\ref{prop:5}, which also contains the proof. For a discussion on the connection between the proof of Proposition~\ref{prop:5} and the asymptotic excursion method, see Remark~\ref{rmk:AR}.

The asymptotic form in \eqref{lim:aux} highlights the special significance of $\mu=1$, identified in \eqref{KK:expmom} as the critical case for exponential moments of the Kallianpur--Robbins law \eqref{def:KK}. Notably, there is a phase transition in the exponential term on the right-hand side: the exponent scales as $\log^{-1}\vep^{-1}$ for $\mu \neq 1$, but as $\log^{-2}\vep^{-1}$ when $\mu=1$. In this critical case, the next-order term $-A_{\mu,\lambda}(b,b)/\log \vep^{-1}$ becomes dominant, and $-A_{\mu,\lambda}(b,b)$ encapsulates the dependence on both parameters $\mu$ and $\lambda$ of $\lv$. Recall as well formula \eqref{lim:beta} for $\beta$. Furthermore, \eqref{lim:aux} provides asymptotic representations that extend throughout the entire supercritical regime ($\mu>1$).

Theorem~\ref{thm:critres}, as an application of Theorem~\ref{thm:>0}, demonstrates that the established critical limits for $\lla \varphi\rra=1$, stated in \eqref{finalgoal0}, can be derived from the radially symmetric function $\oovarphi\uvep$ introduced before Section~\ref{sec:DGF}. This theorem introduces a new approximation for the two-dimensional Laplacian with a singular perturbation at the origin, as discussed in Section~\ref{sec:intro}. The proof is given in Section~\ref{sec:radial}.

\begin{thm}[Second main theorem for $\bs\lla\bs \varphi\bs \rra\bs =\bs 1$]\label{thm:critres}
Assume $\varphi\in \C_c(\R^2)$ with $\varphi\geq 0$ and $\lla \varphi\rra=1/2$.
Let $\oovarphi\uvep$ be defined by \eqref{def:phiexp} and $(\oovarphi\uvep)_\vep$ be defined as in \eqref{def:ffvep}, and take $\mu=1$ in \eqref{def:lvo}. For any $0\leq g\in \B_b(\R^2)$ and $z\neq 0$, 
\begin{align}
\lim_{\vep\to 0} G_{q}^{\lv(\oovarphi\uvep)_\vep}\{g\} ( z)=G_{q}\{g\}(z)+\frac{2\pi}{\log (q/\beta)}G_{q}(z)G_{q}\{g\}(0),\label{finalgoal0}
\end{align}
where $q$ satisfies $q>\beta$ and is large enough, and $\beta$ is defined by \eqref{lim:beta}.
\end{thm}

\begin{rmk}\label{rmk:betadiff}
The prefactor $2\pi/[\log (q/\beta)]$ in \eqref{finalgoal0}, defined by \eqref{lim:beta}, matches the expression $2\pi/[\log (q/\beta)]$  found in \cite[(5.12)]{C:DBG}, though the parameter $\beta$ in the latter case satisfies the following formula: 
\begin{align}\label{lim:beta1}
\frac{\log \beta}{2}&=-\dint \phi(z')\phi(z'')\log |z'-z''|\d z'\d z'+\log 2 +\lambda- \gamma_{\sf EM}.
\end{align}
To address these differences, consider $\varphi(z) \equiv 2\phi(\sqrt{2}z)$ with $\int \phi = 1$ in the context of \eqref{finalgoal0}. Then, the expression for $\beta$ in \eqref{lim:beta1} follows from the formula for $\beta$ in \eqref{lim:beta}.
\qed 
\end{rmk}

Let us point out that while the asymptotic recursion approach outlined in Section~\ref{sec:KKregime} extends to the present case $\lla \varphi\rra=1$, this approach and the one based on Theorem~\ref{thm:excursion} have distinct strengths. Specifically, in Section~\ref{sec:another}, we demonstrate that the asymptotic recursion method enables a shorter derivation of \eqref{finalgoal00}. This technique could potentially provide an alternative proof of Theorem~\ref{thm:critres} as well. However, such a proof would be more limited than the approach based on Theorem~\ref{thm:excursion}, as it would only apply to the subcritical and critical cases.

\section{Asymptotic recursions}\label{sec:=0}
\subsection{First-order expansions}\label{sec:expansion}
In this subsection, we prove Lemmas~\ref{lem:expansion=0} and~\ref{lem:1replace00} on  expansions of $F_\infty$. To facilitate the proof of Proposition~\ref{prop:gbdd} (2$\cc$), to be presented in Section~\ref{sec:geo}, we also generalize part of these lemmas to $\{F_n\}_{n\in \mathbb{Z}}$ and $\{D_n\}_{n\in \mathbb{Z}\cup \{\infty\}}$. Here, the new term $D_\infty$ refers to any $\widetilde{F}$ satisfying $\|\widetilde{F}\|_\varphi<\infty$ and, for some $q\in (0,\infty)$,
 $\widetilde{F}(z)=\mc T_\vep\{\widetilde{F}\}(z)$ for all $z\in \supp(\varphi)$ and all sufficiently small $\vep\in (0,\ol{\vep}(\chi)]$. Additionally, recall
\begin{align}
\begin{split}
F_n&\,\defeq\,
\begin{cases}
0,&\Bbb Z\setminus \Bbb Z_+,\\
G_q\{g\}(\vep z)+\mc T_\vep\{F_{n-1}\}(z) ,&n\in \Bbb Z_+,\\
G^{\llv^{1/2}\varphi_\vep}_q\{g\}(\vep z),&n=\infty,
\end{cases}\\
D_n&\,\defeq\, F_n-F_{n-1},\quad  n \in \mathbb{Z},
\end{split}
\label{def:Fn-new}
\end{align}
which restates \eqref{def:F}, \eqref{def:Fc}, and the definitions of $D_n$ in Proposition~\ref{prop:gbdd}. Note that these definitions use the following notations: as in  \eqref{def:Af}, \eqref{def:Gf}, \eqref{def:Gq0}, \eqref{def:G}, and \eqref{def:Tvep},
\begin{align*}
A_f(t)&\,\defeq\,\int_0^t f(W_s)\d s,\quad 
G^f_\nu\{g\}(z)\,\defeq \int_0^\infty \e^{-\nu t}\E^W_z[\e^{A_f(t)}g(W_t)]\d t,\quad G_\nu\{g\}\;\defeq\; G_\nu^0\{g\},\\
G_\nu(z,z')&\,\defeq \int_0^\infty \e^{-\nu t}P_t(z,z')\d t,\quad 
\mc T_\vep\{h\}(z)\,\defeq\,\lv^{1/2}\int G_q( \vep z,\vep z')\varphi(z')h( z')\d z',
\end{align*}
and $P_t(z,z')$, specified in \eqref{def:Pt},
denote the transition probability densities of the two-dimensional standard Brownian motion.\medskip  

\begin{proof}[Proof of (\ref{exp:10}) of Lemma~\ref{lem:expansion=0}] By the fundamental theorem of calculus,
\begin{align}
&\E^W_{\vep z}[\e^{A_{\lv^{1/2} \varphi_\vep}(t)}g(W_{t})]\notag\\
&\quad =\E^W_{\vep z}\left[\left(1+\int_0^t \e^{A_{\lv^{1/2} \varphi_\vep}(t)-A_{\lv^{1/2} \varphi_\vep}(s)}\d_s A_{\lv^{1/2} \varphi_\vep}(s)\right)g(W_t)\right]
\label{exp:12proof0}\\
&\quad =\E^W_{\vep z}[g(W_t)]+\lv^{1/2}\int_0^t\int_{\Bbb R^2} P_s(\vep z,\vep z')\varphi(z')\E^W_{\vep z'}[\e^{A_{\lv^{1/2} \varphi_\vep}(t-s)}g(W_{t-s})]\d z'\d s,\label{exp:12proof}
\end{align}
where \eqref{exp:12proof} is obtained by conditioning on $\sigma(W_r;r\leq s)$ and then removing $\vep$ in $\varphi_\vep( z')\,\defeq\,\vep^{-2}\varphi(\vep^{-1}z')$ using a change of variables. 

We now obtain the required expansion $F_\infty(z)=G_q\{g\}(\vep z)+ 
\mc T_\vep\{F_{\infty}\}(z)$ in \eqref{exp:10} by taking the Laplace transforms of the left-hand side of
 \eqref{exp:12proof0} and the right-hand side of
 \eqref{exp:12proof}. In more detail, the first term on the right-hand side of \eqref{exp:12proof} transforms to $G_q\{g\}(\vep z)$, and the last term in \eqref{exp:12proof} 
transforms to $\mc T_\vep\{F_{\infty}\}(z)$. 
\end{proof}

Next, we aim to establish \eqref{Finfty:expansion} and its generalizations for $\{F_n\}_{n\in \Bbb Z}$ and $\{D_n\}_{n\in \Bbb Z\cup \{\infty\}}$. The main tool is the following asymptotic expansion: as $\nu|z-z'|^2\searrow 0$,
\begin{align}\label{asymp:G}
\begin{split}
G_{\nu}(z,z')=\frac{1}{\pi}\biggl(\log\frac{2^{1/2}}{\nu^{1/2}|z-z'|} -\gamma_{\sf EM}\biggr)+ \mathcal O\biggl(\nu|z-z'|^2\log \frac{1}{\nu^{1/2} |z-z'|}\biggr),
\end{split}
\end{align}
where $\gamma_{\sf EM}$ is the Euler--Mascheroni constant. Note that \eqref{asymp:G} holds because, by comparing \eqref{def:G} with an integral representation of the Macdonald function $K_0$  \cite[(5.10.25) on p.119]{Lebedev}, we have the following formula:
\begin{align}\label{id:GK}
G_\nu(z,z')=\pi^{-1}K_0( \sqrt{2\nu}|z-z'|),
\end{align}
whereas a series expansion of $K_0$ \cite[(5.7.11) on p.110]{Lebedev} gives
\begin{align}\label{asymp:K0}
K_0(a)=\log (2/a)-\gamma_{\mathsf E\mathsf M}+\mathcal O(a^2\log a^{-1})\quad \mbox{as }a\searrow 0.
\end{align}

In the sequel, we use the notations defined at the end of Section~\ref{sec:intro}, as well as
the following shorthand notation:
\begin{align}\label{def:IX}
\begin{split}
I^F_{n}(z)&\,\defeq\, \begin{cases}
G_q\{g\}(\vep z),&n\in \Bbb Z_+\cup\{\infty\},\\
0,&\mbox{otherwise},
\end{cases}\\
 I^D_{n}(z)&\,\defeq\, I^F_{n}(z)-I^F_{n-1}(z)=\begin{cases}
G_q\{g\}(\vep z),&n=0,\\
0,&\mbox{otherwise,}
\end{cases}
\end{split}
\end{align}
where $n\in \Bbb Z\cup\{\infty\}$ and $\infty-1$ is understood as $\infty$.

\begin{lem}\label{lem:expX}
Let $X \in \{F, D\}$ and $n \in \mathbb{Z} \cup \{\infty\}$. Fix $q \in (0, \infty)$. If $n = \infty$, require that $q$ satisfies Assumption~\ref{ass:q} when $X = F$, or that $D_\infty$ exists for all sufficiently small $\vep\in (0,\ol{\vep}(\chi)]$ when $X = D$. Then for all $\vep\in (0,\ol{\vep}(\chi)]$, sufficiently small when $n=\infty$, and all $z\in \supp(\varphi)$,
\begin{align}
\begin{split}
X_{n+1}(z)&=I^X_{n+1}( z)+\frac{\llv^{1/2}}{\pi} \biggl(\log\frac{2^{1/2}}{q^{1/2}\vep}-\gamma_{\sf EM}\biggr)  \lla \varphi X_n\rra +\llv^{1/2}\mc K\{X_n\}(z) \\
&\quad+\llv^{1/2}\int\mathcal O\biggl(q\vep^2 |z-z'|^2\log \frac{1}{q^{1/2}\vep |z-z'|}\biggr)\varphi(z')X_{n}(z')\d z',\label{1replace:eq}
\end{split}
\end{align}
where $\lla f\rra\,\defeq\int f(z)\d z$ and $\mc K\{f\}(z)\,\defeq\int (\pi^{-1}\log|z-z'|^{-1})\varphi(z')f(z')\d z'$, as defined in \eqref{def:phig} and \eqref{def:K}. Moreover, we can choose $C>0$ independent of $n\in \Bbb Z\cup \{\infty\}$ such that $
C|q\vep^2 |z-z'|^2\log [q^{1/2}\vep |z-z'|]^{-1}|$
bounds the $\mathcal O$-term in \eqref{1replace:eq}, for any $X\in \{F,D\}$. 
 \end{lem}
\begin{proof}
By \eqref{def:Fn-new}, \eqref{1replace:eq} holds trivially whenever $n \notin \mathbb{Z}_+ \cup \{\infty\}$. For $X = F$ and $n \in \mathbb{Z}_+ \cup \{\infty\}$, the last three terms on the right-hand side of \eqref{1replace:eq} are obtained by expanding $G_q(\vep z, \vep z')$ as in \eqref{asymp:G}, and then applying this expansion to $F_n(z) = G_q\{g\}(\vep z) + \mathcal{T}_\vep\{F_{n-1}\}(z)$. (Recall that this equation for $F_n(z)$ holds for finite $n$ by the definition of $F_n$ and for $n = \infty$ by the expansion \eqref{exp:10} proved above.) The proof of \eqref{1replace:eq} for $X = D$ and $n \in \mathbb{Z}_+ \cup \{\infty\}$ follows similarly upon using linearity. Finally, the last assertion of the lemma holds simply because all the $\mathcal O$-terms arise from the asymptotic expansion of $G_q(\vep z,\vep z')$ for $z,z'\in \supp(\varphi)$.
\end{proof}

\begin{lem}\label{lem:Gexp}
For all $h\in \B_b(\Bbb R^2)$ and $q\in (0,\infty)$, it holds that 
\begin{align}\label{mod:G}
|G_{q}\{h\}(\vep z')-G_q\{h\}(\vep z)|
 \leq 2 \vep^{2/3}|z'-z|\|h\|_\infty+C\vep^{2/3}|z'-z|^{1/2}G_{q/2}\{|h|\}(\vep z)
\end{align}
whenever $z,z'\in \Bbb R^2$ and $\vep\in (0,1)$ satisfy $\vep^{2/3}|z-z'|^{1/2}\leq 1$.
In particular, 
\begin{align}
G_q\{h\}(\vep z)=G_q\{h\}(0)+\mathcal O_{\varphi}[(1+q^{-1})\vep^{2/3}\|h\|_\infty ]\quad\mbox{$\forall\;z \in \supp(\varphi)$}.\label{Texp1}
\end{align}
\end{lem}
\begin{proof}
We use the following bound from \cite[Lemma~4.16 (2$\cc$)]{C:DBG}:
for all $\vep_0,\delta_0\in (0,\infty)$ such that $\vep_0/\delta_0^{1/2}\leq 1$ 
and for all $M_0\in (0,\infty)$,
\begin{align}\label{gauss:vep}
&\sup_{|z_0|\leq M_0}|P_t(\vep_0 z_0+z_1)-P_t(z_1)|\leq C(M_0)(\vep_0/\delta_0^{1/2})P_{2t}(z_1),\; \forall\;z_1\in \Bbb R^2,\;t\geq \delta_0.
\end{align}
To see why this bound implies \eqref{mod:G}, write
\[
P_t(\vep z'-z'')-P_{t}(\vep z-z'')=P_t\left(\vep |z'-z|\cdot \frac{(z'-z)}{|z'-z|}  +\vep z-z''\right)-P_t(\vep z-z'')
\]
for all $z'\neq z$. Then, by taking $M_0=1$, $\vep_0=\vep |z'-z|$, $\delta_0=\vep^{2/3}|z'-z|$, $z_0=(z'-z)/|z'-z|$ and $z_1=\vep z-z''$, \eqref{gauss:vep} implies the following inequality whenever the sufficient condition $\vep_0/\delta_0^{1/2}=\vep^{2/3}|z'-z|^{1/2}\leq 1$ holds:
\begin{align*}
&|G_{q}\{h\}(\vep z')-G_q\{h\}(\vep z)|\\
&\quad \leq \int_0^{\infty}\e^{-qt}\left(|P_th(\vep z')-P_th(\vep z)|\1_{t< \delta_0}+C(M_0)(\vep_0/\delta_0^{1/2}) P_{2t}|h|(\vep z)\1_{t\geq \delta_0}\right)\d t.
\end{align*}
The required inequality \eqref{mod:G} now follows immediately.
\end{proof}

\begin{proof}[Proof of (\ref{Finfty:expansion}) of Lemma~\ref{lem:expansion=0}]
Combine \eqref{1replace:eq} and \eqref{Texp1}, with \eqref{1replace:eq} restricted to $(X,n)=(F,\infty)$. In this case, the last term in \eqref{1replace:eq} is $\mathcal O_{\varphi,\chi}\{q\vep^{2}[\log (q^{1/2}\vep)^{-1}]\}\|F_{\infty}\|_{\varphi}$.
\end{proof}

Finally, we turn to Lemma~\ref{lem:1replace00} and its generalizations to  $\{F_n\}_{n\in \Bbb Z}$ and $\{D_n\}_{n\in \Bbb Z\cup \{\infty\}}$. 

\begin{lem}\label{lem:1replace0}
Let $X \in \{F, D\}$ and $n \in \mathbb{Z} \cup \{\infty\}$. Fix $q \in (0, \infty)$. If $n = \infty$, require that $q$ satisfies Assumption~\ref{ass:q} when $X = F$, or that $D_\infty$ exists for all sufficiently small $\vep\in (0,\ol{\vep}(\chi)]$ when $X = D$. Then for all $\vep\in (0,\ol{\vep}(\chi)]$, sufficiently small when $n=\infty$, and all $z'',z'\in \supp(\varphi)$,
\begin{align}
\lla \varphi X_{n+1}\rra&=\llv^{1/2}\lla \mathcal E(\varphi)X_n\rra +\O^X_{n},\label{1replace}\\
X_{n+1}( z'')-X_{n+1}( z')&=\llv^{1/2}\mc K\{ X_n\}(z'')-\llv^{1/2} \mc K\{X_n\}(z')+\O^X_n,\label{F1replace}
\end{align}
where $\mathcal E\{f\}(z')\,\defeq\,\varphi(z')\int (\pi^{-1}\log|z'-z''|^{-1})\varphi(z'')\d z''$, as defined in \eqref{def:energy}, and for all $n\in\Bbb Z\cup \{\infty\}$,
\begin{align}\label{def:Om}
\begin{split}
\O^F_n&\;\defeq\, \mathcal O_{\varphi,\chi}[(1+q^{-1})\vep^{2/3}\|g\|_\infty]+
\mathcal O_{\varphi,\chi}\{q\vep^2[\log (q^{1/2}\vep)^{-1}] \}\|F_{n}\|_\varphi,\\
\O^D_n&\;\defeq\, \1_{n+1=0}\mathcal O_{\varphi,\chi}[(1+q^{-1})\vep^{2/3}\|g\|_\infty]+
\mathcal O_{\varphi,\chi}\{q\vep^2[\log (q^{1/2}\vep)^{-1}] \}\|D_{n}\|_\varphi.
\end{split}
\end{align}
\end{lem}

\begin{proof}
To get \eqref{1replace}, note that integrating both sides of \eqref{1replace:eq} against $\varphi(z)\d z$ gives
\begin{align*}
\begin{split}
\lla \varphi F_{n+1}\rra&=\mathcal O_\varphi[(1+q^{-1})\vep^{2/3}\|g\|_\infty]+\llv^{1/2}\lla \mathcal E(\varphi)F_n\rra 
+ \mc O_{\varphi,\chi}\{q\vep^2[\log (q^{1/2}\vep)^{-1}]\}\|F_{n}\|_\varphi
,\\
\lla \varphi D_{n+1}\rra&=\1_{n+1=0}\mathcal O_\varphi[(1+q^{-1})\vep^{2/3}\|g\|_\infty]+\llv^{1/2}\lla \mathcal E(\varphi)D_n\rra 
\\
&\quad+
 \mc O_{\varphi,\chi}\{q\vep^2[\log (q^{1/2}\vep)^{-1}]\}\|D_{n}\|_\varphi.
\end{split}
\end{align*}
In more detail, the first terms on the right-hand sides follow since $G_q\{g\}(\vep z)$ integrates to $\mathcal O_{\varphi}[(1+q^{-1})\vep^{2/3}\|g\|_\infty]$ by \eqref{Texp1} and the assumption of $\lla \varphi\rra=0$. The use of $\1_{n+1=0}$ considers $I^D_n=0$ for $n\neq 0$ by \eqref{def:IX}. The third term on the right-hand side of \eqref{1replace:eq} integrates to $\lv^{1/2}\lla \mathcal E(\varphi)X_n\rra$ by the definition $\mc E(\varphi)(z')$. The sum of the other terms on the right-hand side of \eqref{1replace:eq} integrates to $
 \mc O_{\varphi,\chi}\{q\vep^2[\log (q^{1/2}\vep)^{-1}]\}\|X_{n}\|_\varphi$. In particular, the constant term vanishes after the integration, using the property $\lla \varphi\rra=0$.  

The proof of \eqref{F1replace} is almost identical. We now use \eqref{1replace:eq} to get an expansion of $X_{n+1}( z'')-X_{n+1}( z')$. Then the change we need is that by \eqref{mod:G}, $|G_q\{g\}(\vep z'')-G_q\{g\}(\vep z')|=\mathcal O_{\varphi}[(1+q^{-1})\vep^{2/3}\|g\|_\infty]$ for all $z'',z'\in \supp(\varphi)$. 
\end{proof}

\begin{proof}[Proof of Lemma~\ref{lem:1replace00}]
Take $X=F$ and $n=\infty$ in \eqref{1replace} and \eqref{F1replace}.
\end{proof}

\subsection{Higher-order expansions}\label{sec:multi}
We now proceed to the proof of Proposition~\ref{prop:Fvep=0}, again along with its generalizations to $\{F_n\}_{n\in \Bbb Z}$ and $\{D_n\}_{n\in \Bbb Z\cup\{\infty\}}$. Recall that $\lv$ satisfies \eqref{def:llv} throughout Section~\ref{sec:=0}. Also, the operator $\mc K\{f\}(z)\,\defeq\,\int (\pi^{-1}\log|z-z'|^{-1})\varphi(z')f(z')\d z'$ and the functions $I^F_n,I^D_n$ have been defined in \eqref{def:K} and \eqref{def:IX}, respectively. In the sequel, it is understood that $\mathcal K^2\{\1\}=\mathcal K^2\{\1\}(z)$ and $\mathcal K\{\1\}^2=\mathcal K\{\1\}(z)^2$.

\begin{lem}\label{lem:signexpansion}
Let $X \in \{F, D\}$ and $n \in \mathbb{Z} \cup \{\infty\}$. Fix $q \in (0, \infty)$. If $n = \infty$, require that $q$ satisfies Assumption~\ref{ass:q} when $X = F$, or that $D_\infty$ exists for all sufficiently small $\vep\in (0,\ol{\vep}(\chi)]$ when $X = D$.\medskip 

\noindent {\rm (1$\cc$)}  For all $\vep\in (0,\ol{\vep}(\chi)]$, sufficiently small when $n=\infty$, and all $z\in \supp(\varphi)$, we have
\begin{align}
\begin{split}
X_{n+1}( z)&=I^X_{n+1}(0)+A^o_1(z)X_n( z)+ A^o_{2}(z)X_{n-1}(z) + A_{3}^o(z)X_{n-2}(z)\\
&\quad+
A_{4}^o(z)X_{n-3}(z) 
 +\mc O^{X\prime}_{n,n-4},\label{1replace:eq1goal}
 \end{split}
\end{align} 
where $A^o_j=A^o_j(z)$ are defined as follows: 
\begin{align}
A^o_{1}&\,\defeq\,
\llv^{1/2} \mc K\{\1\},\label{def:A1o}\\ 
\begin{split}\label{def:A2o}
A^o_{2}&\,\defeq\,\frac{\llv \lla \mathcal E(\varphi)\rra\log \vep^{-1}}{\pi} \\
&\quad\;+ \lv\biggl\{\frac{1}{\pi} \biggl(\log \frac{2^{1/2}}{q^{1/2} }-\EM\biggr) \lla\mathcal E(\varphi)\rra  + [\mc K^2\{\1\}-\mc K\{\1\}^2]\biggr\},
\end{split}\\
A^o_{3}&\,\defeq\, \frac{\llv^{3/2} \lla \mc E(\varphi)\rra\log \vep^{-1}}{\pi}\left(- \mc K\{\1\}+\frac{\lla \mc E(\varphi)\mc K\{\1\}\rra}{\lla \mc E(\varphi)\rra}\right),\label{def:A3o}\\
\begin{split}
A^o_{4}&\,\defeq\,\frac{\llv^{2}\lla \mc E(\varphi)\rra \log \vep^{-1}}{\pi}\\
&\quad \times \biggl(
\frac{\lla \mc E(\varphi)\mc K^2\{\1\}\rra -\lla \mc E(\varphi)\mc K\{\1\}\rra \mc K\{\1\}}{\lla \mc E(\varphi)\rra}-[ \mc K^2\{\1\}-\mc K\{\1\}^2]\biggr),\label{def:A4o}
\end{split}
\end{align}
and with the $\O$-terms $\O^X_m$ defined by \eqref{def:Om},
\begin{align}
\begin{split}
\mathcal O^{X\prime}_m&\,\defeq\, \frac{\llv^{1/2}}{\pi}\biggl|\log \frac{2^{1/2}}{q^{1/2}\vep}-\EM\biggr|\O^X_m+\O^X_m+
\frac{\llv}{\pi}\biggl|\log \frac{2^{1/2}}{q^{1/2}}-\EM\biggr|
\O_m^X\\
&\quad\; +\frac{\llv^{3/2}}{\pi}\biggl|\log \frac{2^{1/2}}{q^{1/2}}-\EM\biggr|\O_{\varphi,\chi}(1)\|X_m\|_\varphi\\
&\quad\;+(\log \vep^{-1})^{-3/2}\O_{\varphi,\chi}(1)\|X_m\|_\varphi,\quad m\in \Bbb Z\cup \{\infty\},\label{def:Om'}
\end{split}\\
\mc O^{X\prime}_{m_2,m_1}& \,\defeq 
\begin{cases}
\mathcal O^{X\prime}_{m_2}+\cdots+\mathcal O^{X\prime}_{m_1},& m_2\geq m_1,\; m_1,m_2\in \Bbb Z,\label{def:O'n1n2}\\
\O^{X\prime}_\infty,&m_2=\infty, \;m_1=\infty-m,\;m\in \Bbb N.
\end{cases}
\end{align}

\noindent {\rm (2$\cc$)} Set $\mu = 1$ in the definition \eqref{def:llv} of $\llv$. For all $\varepsilon \in (0, \ol{\vep}(\chi)]$, sufficiently small when $n = \infty$, and for all $z \in \supp(\varphi)$, we have
\begin{align}
\begin{split}\label{rec:F}
X_{n+1}(z)&=I^X_{n+1}(0)+A_1(z)X_{n}(z)+A_2(z)X_{n-1}(z)
 +A_3(z)X_{n-2}(z) \\
 &\quad+A_4(z)X_{n-3}(z)+\mc O^{X\prime}_{n,n-4},
\end{split}
\end{align}
where $A_j=A_j(z)$ are defined as follows: 
\begin{align}
A_1&\;\defeq\,\frac{\frac{\pi^{1/2}\mc K\{\1\}}{\slla \mc E(\varphi)\srra^{1/2}}}{\log^{1/2} \vep^{-1}}+\frac{\frac{\lambda'\pi^{1/2}\mc K\{\1\}}{2\slla \mc E(\varphi)\srra ^{1/2}}}{\log \vep^{-1}},\label{def:a}\\
\begin{split}
A_2&\;\defeq\,1+\frac{\lambda'}{\log^{1/2} \vep^{-1}} +\frac{\lambda +\log \frac{2^{1/2}}{q^{1/2}}-\EM+\frac{\pi [\mc K^2\{\1\}-\mc K\{\1\}^2]}{\slla \mc E(\varphi)\srra }}{\log \vep^{-1}},\label{def:b}
\end{split}\\
\begin{split}
A_3&\;\defeq\,\frac{\frac{\pi^{1/2}}{\slla \mc E(\varphi)\srra^{1/2}}\Bigl(-\mc K\{\1\}+\frac{\slla \mc E(\varphi)\mc K\{\1\}\srra}{\slla \mc E(\varphi)\srra }\Bigr)}{\log^{1/2} \vep^{-1}} +\frac{ \frac{3\lambda'\pi^{1/2}}{2\slla \mc E(\varphi)\srra^{1/2}}\Bigl(-\mc K\{\1\}+\frac{\slla \mc E(\varphi)\mc K\{\1\}\srra}{\slla \mc E(\varphi)\srra }\Bigr) 
}{\log \vep^{-1}},\label{def:c}
\end{split}\\
A_4&\;\defeq\,\frac{\frac{\pi\slla \mc E(\varphi)\mc K^2\{\1\}\srra-\pi\slla \mc E(\varphi)\mc K\{\1\}\srra \mc K\{\1\}}{\slla \mc E(\varphi)\srra^2}
-\frac{\pi[\mc K^2\{\1\}-\mc K\{\1\}^2]}{\slla \mc E(\varphi\srra}}{ \log \vep^{-1}}.\label{def:d}
\end{align}
\end{lem}

\begin{proof}
All the $\mathcal O$-terms below are understood to be uniform over $z\in \supp(\varphi)$. \medskip 

\noindent (1$\cc$) The main equations for proving \eqref{1replace:eq1goal} can be summarized as the expansions in the next three displays. First, we will show in Step~1 below that
\begin{align}
 X_{n+1}( z)
&=I^X_{n+1}(0)+\frac{\llv}{\pi} \biggl(\log \frac{2^{1/2}}{q^{1/2} }-\EM\biggr) \lla\mathcal E(\varphi)\rra X_{n-1}( z)\notag\\
&\quad +\frac{\llv\log \vep^{-1}}{\pi} \lla \mathcal E(\varphi)\rra X_{n-1}( z)+\frac{\llv\log \vep^{-1}}{\pi}\lla \mathcal E(\varphi) [X_{n-1}-X_{n-1}(z)]\rra \notag\\
&\quad +\llv^{1/2} \mc K\{\1\}(z) X_n( z)+\llv^{1/2} \mc K \{X_n-X_n(z)\}(z)
 +\mc O^{X\prime}_{n,n-2}.\label{1replace:eq1}
\end{align} 
In Step~2 below, the fourth term on the right-hand side of \eqref{1replace:eq1} will be shown to satisfy
\begin{align}\label{1repace:step2-final1}
\begin{split}
\frac{\llv \log \vep^{-1}}{\pi} \lla \mathcal E(\varphi) [X_{n-1}-X_{n-1}(z)]\rra 
&=A^o_{3}(z)X_{n-2}(z)+ A^o_{4}(z)X_{n-3}(z)\\
&\quad\;+\mathcal O^{X\prime}_{n-2,n-4},
\end{split}
\end{align}
where $A_3^o(z)$ and $A_4^o(z)$ are defined in \eqref{def:A3o} and \eqref{def:A4o}. 
Finally, the sixth term on the right-hand side of \eqref{1replace:eq1} satisfies, by \eqref{F1replace} with $n+1$ replaced by $n$ and $n-1$, 
\begin{align}
\llv^{1/2} \mc K \{X_n-X_n(z)\}(z)
&=\llv \mc K\{\mc K\{X_{n-1}\}-\mc K\{X_{n-1}\}(z)\}(z)+\mc O^X_{n-1}\notag\\
&=\llv [\mc K^2\{\1\}(z)-\mc K\{\1\}(z)^2]X_{n-1}(z)+\mc O^X_{n-1}\notag\\
&\quad\;+\lv^{3/2}\mc O_{\varphi,\chi}(1)\|X_{n-2}\|_\varphi+\mc O^X_{n-2}\notag\\
&=\llv [\mc K^2\{\1\}(z)-\mc K\{\1\}(z)^2]X_{n-1}(z)+\mc O^{X\prime}_{n-1,n-2}.\label{1repace:step2-final2}
\end{align}
Applying \eqref{1repace:step2-final1} and \eqref{1repace:step2-final2} to the right-hand side of \eqref{1replace:eq1} yields \eqref{1replace:eq1goal}. \medskip

\noindent {\bf Step 1.} To obtain \eqref{1replace:eq1},
first, we show that \eqref{1replace:eq} implies the following expansion: 
\begin{align}
X_{n+1}( z)&=I^X_{n+1}(0)+(\1_{X=F}\1_{n+1\geq 0}+\1_{X=D}\1_{n+1=0})\mathcal O_{\varphi}[(1+q^{-1})\vep^{2/3}\|g\|_\infty]\notag\\
&\quad+\frac{\llv}{\pi}\biggl(\log \frac{2^{1/2}}{q^{1/2}\vep}-\EM\biggr)\lla \mathcal E(\varphi)X_{n-1}\rra\notag\\
&\quad+\frac{\llv^{1/2}}{\pi}\biggl|\log \frac{2^{1/2}}{q^{1/2}\vep}-\EM\biggr|\O^X_{n-1}  
+\llv^{1/2} \mc K\{ X_n\}(z)+\O^X_n.\label{1replace:eq100}
\end{align}
To prove \eqref{1replace:eq100}, we proceed as follows. Recall \eqref{def:IX}, and expand the first term on the right-hand side of \eqref{1replace:eq} by \eqref{Texp1} in the case of nonzero $I^X_{n+1}(z)$ to justify the first and second terms on the right-hand side of \eqref{1replace:eq100}. Next, expand $\lla \varphi X_n\rra$ in the second term on the right-hand side of \eqref{1replace:eq} by  \eqref{1replace} to justify the third and fourth terms on the right-hand side of \eqref{1replace:eq100}. Also, the third term on the right-hand side of  \eqref{1replace:eq} is the same as the fifth term on the right-hand side of \eqref{1replace:eq100}. Finally, the last term of \eqref{1replace:eq} gives the last term of \eqref{1replace:eq100} by the definition \eqref{def:Om} of $\O^X_n$. 

Next, we show why \eqref{1replace:eq100} can be rewritten as the following equation:
\begin{align}
\begin{split}
X_{n+1}( z)
&=I^X_{n+1}(0)+\frac{\llv}{\pi}\biggl(\log \frac{2^{1/2}}{q^{1/2}}-\EM\biggr)\lla \mathcal E(\varphi)X_{n-1}\rra+\\
&\quad +\frac{\llv \log \vep^{-1}}{\pi}\lla \mathcal E(\varphi)X_{n-1}\rra +
\llv^{1/2} \mc K\{ X_n\}(z)+\mc O^{X\prime}_{n}+\mathcal O^{X\prime}_{n-1}.\label{1replace:eq10}
\end{split}
\end{align}
First, the sum of the second and last terms on the right-hand side of \eqref{1replace:eq100}
 is an $\O^{X\prime}_n$  by \eqref{def:Om} and \eqref{def:Om'}; this gives the
  $\O^{X\prime}_n$ in \eqref{1replace:eq10}. Second, the fourth term on the right-hand side of \eqref{1replace:eq100} gives the $\O^{X\prime}_{n-1}$ in \eqref{1replace:eq10}. Finally, by rewriting the third term on the right-hand side of \eqref{1replace:eq100} according to $(\log \frac{2^{1/2}}{q^{1/2}\vep}-\EM)=(\log \frac{2^{1/2}}{q^{1/2}}-\EM)+\log\vep^{-1}$, we obtain the second and third terms on the right-hand side of \eqref{1replace:eq10}.

Now we can obtain \eqref{1replace:eq1} from \eqref{1replace:eq10} as follows.
Expand the second term on the right-hand side of \eqref{1replace:eq10}
by applying \eqref{F1replace}, with $n+1$ replaced by $n-1$, to replace $X_{n-1}$ under $\lla\cdot\rra $ with $X_{n-1}( z)$. Hence, 
\begin{align*}
&\frac{\llv}{\pi}\biggl(\log \frac{2^{1/2}}{q^{1/2}}-\EM\biggr)\lla \mathcal E(\varphi)X_{n-1}\rra\\
&\quad =\frac{\llv}{\pi}\biggl(\log \frac{2^{1/2}}{q^{1/2}}-\EM\biggr)\lla \mathcal E(\varphi)\rra X_{n-1}(z)+\frac{\llv^{3/2}}{\pi}\biggl|\log \frac{2^{1/2}}{q^{1/2}}-\EM\biggr|\O_{\varphi,\chi}(1)\|X_{n-2}\|_\varphi\\
&\quad \quad +\frac{\llv}{\pi}\biggl|\log \frac{2^{1/2}}{q^{1/2}}-\EM\biggr|
\O_{n-2}^X\\
&\quad =\frac{\llv}{\pi}\biggl(\log \frac{2^{1/2}}{q^{1/2}}-\EM\biggr)\lla \mathcal E(\varphi)\rra X_{n-1}(z)+\mathcal O^{X\prime}_{n-2}.
\end{align*}
By this equality, the second, fifth and sixth terms on the right-hand side of \eqref{1replace:eq10} lead to the second and last terms on the right-hand side of 
\eqref{1replace:eq1}.
Also, let us just rewrite the third and fourth terms on the right-hand side of \eqref{1replace:eq10} according to $X_{j}=X_{j}( z)+[X_{j}-X_{j}( z)]$. This leads to the $m$-th terms on the right-hand side of  \eqref{1replace:eq1} for $m=3,4,5,6$. 
We have proved \eqref{1replace:eq1}. \medskip 

\noindent {\bf Step 2.} To obtain \eqref{1repace:step2-final1},  we first use \eqref{F1replace} and the definition \eqref{def:Om} of $\O^X_n$ to get
\begin{align}
\lla \mathcal E(\varphi)[X_{n-1}-X_{n-1}(z)]\rra&=\llv^{1/2} 
\lla \mc E(\varphi)[\mc K\{X_{n-2}\}-\mc K\{X_{n-2}\}(z)]\rra +\mathcal O^X_{n-2}\notag\\
\begin{split}\label{1replace:step2}
&=\llv^{1/2}\lla \mc E(\varphi) \mc K\{X_{n-2}-X_{n-2}(z)\}\rra\\
&\quad +\llv^{1/2}\lla \mathcal E(\varphi)\mc K\{\1\}\rra X_{n-2}(z) \\
&\quad  - \llv^{1/2}\lla \mathcal E(\varphi)\rra \mc K\{X_{n-2}\}(z) +\mathcal O^X_{n-2},
\end{split}
\end{align}
where the second term on the right-hand side of \eqref{1replace:step2} follows by writing $X_{n-2}=[X_{n-2}-X_{n-2}(z)]+X_{n-2}(z)$. 

Next, we expand the first and third terms on the right-hand side of \eqref{1replace:step2}. 
For the first term on the right-hand side of \eqref{1replace:step2}, we have
\begin{align}
&\llv^{1/2}\lla \mc E(\varphi) \mc K\{X_{n-2}-X_{n-2}(z)\}\rra\notag\\
&\quad =\llv\lla \mc E(\varphi) \mc K\{\mc K\{X_{n-3}\}-\mc K\{X_{n-3}\}(z)\}\rra+\mathcal O^X_{n-3}\notag\\
&\quad =\llv \lla \mc E(\varphi)\mc K^2\{X_{n-3}\}\rra-\llv\lla \mc E(\varphi)\mc K\{\1\}\rra \mc K\{X_{n-3}\}(z)+\mc O_{n-3}^X\notag\\
&\quad =\llv \lla \mc E(\varphi)\mc K^2\{\1\}\rra X_{n-3}(z) -\llv \lla \mc E(\varphi)\mc K\{\1\}\rra \mc K\{\1\}(z)X_{n-3}(z)+\mc O^X_{n-3} \notag\\
&\quad\quad  +(\log \vep^{-1})^{-3/2}\O_{\varphi,\chi}(1)\|X_{n-4}\|_\varphi+\O^X_{n-4}\notag\\
&\quad =\llv [\lla \mc E(\varphi)\mc K^2\{\1\}\rra -\lla \mc E(\varphi)\mc K\{\1\}\rra \mc K\{\1\}(z)]X_{n-3}(z)+\mc O^{X\prime}_{n-3}+\mc O^{X\prime}_{n-4},\label{1replace:step21}
\end{align}
where the first and third equalities use \eqref{F1replace} for $n+1$ replaced by $n-2$ and $n-3$, respectively.
Similarly, the third term on the right-hand side of \eqref{1replace:step2} satisfies 
\begin{align}
&-\llv^{1/2}\lla \mathcal E(\varphi)\rra \mc K\{X_{n-2}\}(z)\notag\\
&\quad =-\llv^{1/2}\lla \mathcal E(\varphi)\rra \mc K\{\1\}(z)X_{n-2}(z)\notag\\
&\quad \quad -\llv\lla \mathcal E(\varphi)\rra \mc K\{\mc K\{X_{n-3}\}-\mc K\{X_{n-3}\}(z)\}(z) +\mathcal O^{X}_{n-3}\notag\\
\begin{split}
&\quad =-\llv^{1/2} \lla \mc E(\varphi)\rra \mc K\{\1\}(z)X_{n-2}(z)\\
&\quad \quad -\llv {\lla \mc E(\varphi)\rra}  [\mc K^2\{\1\}(z)-\mc K\{\1\}(z)^2]X_{n-3}(z) +\mathcal O^{X}_{n-3}+\mathcal O^{X}_{n-4},  \label{1replace:step22}
\end{split}
\end{align}
where the first and second equalities use \eqref{F1replace} for 
$n+1$ replaced by $n-2$ and $n-3$.

Finally, applying \eqref{1replace:step21} and \eqref{1replace:step22} to \eqref{1replace:step2}
gives the following equivalent of \eqref{1repace:step2-final1}:
 \begin{align*}
&\frac{\llv \log \vep^{-1}}{\pi} \lla \mathcal E(\varphi) [X_{n-1}-X_{n-1}(z)]\rra \\
&\quad =\frac{ \llv^{3/2}\log \vep^{-1}}{\pi} \left[
 -\lla \mc E(\varphi)\rra \mc K\{\1\}(z)+\lla \mc E(\varphi)\mc K\{\1\}\rra\right]
X_{n-2}(z)\\
&\quad \quad +\frac{\llv^2 \log \vep^{-1}}{\pi}\Big(\lla \mc E(\varphi)\mc K^2\{\1\}\rra-\lla \mc E(\varphi)\mc K\{\1\}\rra \mc K\{\1\}(z)\\
&\quad\quad -\lla \mc E(\varphi)\rra \,[\mc K^2\{\1\}(z)- \mc K\{\1\}(z)^2]\Big)   X_{n-3}(z)+
\mathcal O^{X\prime}_{n-2}+\mathcal O^{X\prime}_{n-3}+\mathcal O^{X\prime}_{n-4}.
\end{align*}
We have proved \eqref{1repace:step2-final1}. \medskip 

\noindent (2$\cc$)  Now we assume $\mu=1$, and recall that $\llv$ is defined by \eqref{def:llv}. Hence, in this case,
\[
\llv \equiv \frac{\pi }{\lla \mc E(\varphi)\rra \log \vep^{-1}}+\frac{\lambda'\pi}{\lla \mc E(\varphi)\rra \log^{3/2} \vep^{-1}}+\frac{\lambda\pi}{\lla \mc E(\varphi)\rra \log^{2} \vep^{-1}}.
\]
Then 
\begin{align}
\frac{\llv\lla \mathcal E(\varphi)\rra \log\vep^{-1}}{\pi}&=1+\frac{\lambda'}{\log^{1/2} \vep^{-1}}+\frac{\lambda}{\log \vep^{-1}},\label{def:llvr}
\end{align}
and we have the following expansion: for $a>0$, 
 \begin{align}
 \begin{split}
\llv^{a}
&=\frac{\pi^{a}}{\lla \mc E(\varphi)\rra^{a}(\log \vep^{-1})^{a}}+\frac{a\pi^{a}\lambda'}{\lla \mc E(\varphi)\rra^{a}(\log \vep^{-1})^{a+1/2}}+ \frac{\mc O_{\varphi,\chi,a}(1)}{(\log \vep^{-1})^{a+1}},\label{def:llv1/2}
\end{split}
\end{align}
which follows by using the expansion below with $x\equiv \pi/(\lla \mc E(\varphi)\rra \log \vep^{-1})$:
\[
(x+xy)^{a}=x^{a}(1+y)^{a}=x^{a}+ax^{a}y+x^{a}\mathcal O_a(y^2)\quad\mbox{ as $y\to 0$}.
\]
Moreover, for $b>0$, taking \eqref{def:llv1/2} with $a=b+1$ implies
\begin{align}
\begin{split}
\frac{\llv^{b+1}\lla \mathcal E(\varphi)\rra \log\vep^{-1}}{\pi}&=\frac{\pi^{b}}{\lla \mc E(\varphi)\rra^{b}(\log \vep^{-1})^{b}}+\frac{(b+1)\pi^{b}\lambda'}{\lla \mc E(\varphi)\rra^{b}(\log \vep^{-1})^{b+1/2}}\\
&\quad+ \frac{\mc O_{\varphi,\chi,b}(1)}{(\log \vep^{-1})^{b+1}}.\label{def:llvb}
\end{split}
\end{align}

We prove the required expansion now by showing that
for each $j=1,2,3,4$, the term $A_j^o(z)X_{n+1-j}(z)$ in \eqref{1replace:eq1goal} can be rewritten as $A_j(z)X_{n+1-j}(z) + \O^{X\prime}_{n+1-j}$. When $j=1$, apply \eqref{def:llv1/2} with $a=1/2$ to \eqref{def:A1o}. For $j=2$, use \eqref{def:llvr} on \eqref{def:A2o}. For $j=3$ and $j=4$, apply \eqref{def:llvb} with $b=1/2$ and $b=1$ to \eqref{def:A3o} and \eqref{def:A4o}, respectively. This completes the proof of \eqref{rec:F}.
\end{proof}

\begin{proof}[End of the proof of Proposition~\ref{prop:Fvep=0}]
Recall that we assume bounded $g\geq 0$ and $\int g\neq 0$.
For the case of $\mu<1$, $X=F$ and $n=\infty$, \eqref{1replace:eq1goal} gives the next equality, followed by its equivalent due to Assumption~\ref{ass:q} and rearrangement: for $z\in \supp(\varphi)$,
\begin{align*}
F_\infty(z)
&=G_q\{g\}(0)+\mu F_\infty(z)+\mc O_{\varphi,\chi,q}\left(\frac{\|F_\infty\|_\varphi}{\log^{1/2} \vep^{-1}}\right)+ \mathcal O_{\varphi,\chi,q}\left(\frac{\|g\|_\infty}{\log^{1/2} \vep^{-1}}\right),\\
F_\infty(z)
&=\frac{G_q\{g\}(0)}{1-\mu}+\mc O_{\varphi,\chi,q}\left(\frac{\|F_\infty\|_\varphi}{\log^{1/2} \vep^{-1}}\right)+\mathcal O_{\varphi,\chi,q}\left(\frac{\|g\|_\infty}{\log^{1/2} \vep^{-1}}\right).
\end{align*}
More specifically, the term $\mu F_\infty(z)$ on the right-hand side of the first equality follows by using the first term in $A_2^o(z)$ and the definition \eqref{def:llv} of $\lv$.
Because $G_q\{g\}(0)>0$ and $F_\infty\geq 0$, the last equality  implies $\limsup_{\vep\to 0}\|F_\infty\|_\varphi<\infty$ and the required limit of Proposition~\ref{prop:Fvep=0} for $\mu<1$. 

The other two limits of Proposition~\ref{prop:Fvep=0} are for $\mu=1$. In this case, \eqref{rec:F} shows 
\begin{align}\label{rec:F1}
F_\infty(z)&=G_q\{g\}(0)+(A_1+A_2+A_3+A_4)(z)F_\infty (z)+\mc O^{X\prime}_{\infty,\infty},
\end{align}
where $A_1,A_2,A_3,A_4$ are defined in \eqref{def:a}--\eqref{def:d}. Note that 
\begin{align}
(A_1+A_2+A_3+A_4)(z)&=1-\frac{\mathfrak C_{\lambda'}}{\log^{1/2} \vep^{-1}}-\frac{\mathfrak C_{\lambda',\lambda,q}(z)}{\log \vep^{-1}}\label{rec:F2}
\end{align}
for $\mathfrak C_{\lambda'}$ defined by \eqref{def:Clambda'} and 
\begin{align*}
\mathfrak C_{\lambda',\lambda,q}(z)&\defeq\,
\left(\lambda'+\frac{\pi^{1/2}\lla \mc E(\varphi)\mc K\{\1\}\rra }{\lla \mc E(\varphi)\rra^{3/2}}\right)\frac{\pi^{1/2}\mc K\{\1\}(z)}{\lla \mc E(\varphi)\rra ^{1/2}}-\lambda -\log \frac{2^{1/2}}{q^{1/2}}+\EM\\
&\quad-\frac{3\lambda'\pi^{1/2}\lla \mc E(\varphi)\mc K\{\1\}\rra}{2\lla \mc E(\varphi)\rra^{3/2} }-\frac{\pi\lla \mc E(\varphi)\mc K^2\{\1\}\rra }{\lla \mc E(\varphi)\rra^2 }.
\end{align*}
Note that $\mathfrak C_{\lambda',\lambda,q}(z)$ reduces to $\mathfrak C_{\lambda,q}$ in \eqref{def:Clambdaq} when $\mathfrak C_{\lambda'}=0$. Hence,
applying \eqref{rec:F1}, \eqref{rec:F2} and Assumption~\ref{ass:q} to the procedure in \eqref{Finfty:X}--\eqref{Finfty:X2}, we obtain the two required limits of Proposition~\ref{prop:Fvep=0} for $\mu=1$. The proof of Proposition~\ref{prop:Fvep=0} is complete.
\end{proof}

\subsection{Geometric bounds}\label{sec:geo}
We proceed to prove Proposition~\ref{prop:gbdd}, relying on the higher-order expansions in Lemma~\ref{lem:signexpansion}. Our primary technical concern is the case $\mu=1$, for which the key equation is provided by the following restatement of \eqref{rec:F}: for  all  fixed $q\in (0,\infty)$ and $n\in \Bbb Z$, 
\begin{align}
\label{rec:F11}
\begin{split}
X_{n+1}&=I^X_{n+1}(0)+(A_1+A_3)X_n+(A_2+A_4)X_{n-1}-A_3(X_n-X_{n-2})\\
&\quad-A_4(X_{n-1}-X_{n-3})+\mc O^{X\prime}_{n,n-4}\quad\mbox{uniformly on $\supp(\varphi)$.}
\end{split}
\end{align}

The first notable aspect of \eqref{rec:F11} is as follows. Based on the definitions \eqref{def:a}--\eqref{def:d} for $A_1, A_2, A_3, A_4$, the coefficients for the second and third terms on the right-hand side of \eqref{rec:F11} can be explicitly determined as follows:  
\begin{align*}
A_1+A_3&=\frac{\frac{\pi^{1/2}\slla \mc E(\varphi)\mc K\{\1\}\srra}{\slla \mc E(\varphi)\srra^{3/2} }}{\log^{1/2} \vep^{-1}}+\frac{-\frac{\lambda'\pi^{1/2}\mc K\{\1\}}{\slla \mc E(\varphi)\srra^{1/2} }+ \frac{3\lambda'\pi^{1/2}\slla \mc E(\varphi)\mc K\{\1\}\srra}{2\slla \mc E(\varphi)\srra^{3/2} }
}{\log \vep^{-1}},\\
A_2+A_4&=1+\frac{\lambda'}{\log^{1/2} \vep^{-1}} +\frac{\lambda +\log \frac{2^{1/2}}{q^{1/2}}-\EM+
\frac{\pi\slla \mc E(\varphi)\mc K^2\{\1\}\srra -\pi\slla \mc E(\varphi)\mc K\{\1\}\srra \mc K\{\1\}}{\slla \mc E(\varphi)\srra^2 }
}{\log \vep^{-1}}.
\end{align*}
Examining the coefficients of $(\log \vep^{-1})^{-1/2}$ in the previous two equations then reveals an initial connection to the crucial constant $\mathfrak C_{\lambda'}$, defined in \eqref{def:Clambda'}. 

The following lemma advances the analysis of \eqref{rec:F11}. In part (1$\cc$), we establish asymptotic expansions for the fourth and fifth terms on the right-hand side of \eqref{rec:F11}. In part (2$\cc$), we derive a recursive inequality, which forms the basis for proving the geometric bounds in Proposition~\ref{prop:gbdd} (1$\cc$).

\begin{lem}\label{lem:recF}
Set $\mu=1$ in $\llv$ defined by \eqref{def:llv}, and let $X\in \{F,D\}$. \medskip 

\noindent {\rm (1$\cc$)} For any fixed $q\in (0,\infty)$, the following holds for all $n\in \Bbb Z$ and all $\vep\in (0,\ol{\vep}(\chi)]$:
\begin{align}
 X_{n}-X_{n-2}
&=I^X_n(0)+
\llv^{1/2}\mc K\{X_{n-1}\}\notag\\
&\quad +\Biggl(\frac{\lambda'}{\log^{1/2} \vep^{-1}}+\frac{\lambda+\log \frac{2^{1/2}}{q^{1/2}}-\EM}{\log \vep^{-1}}+\frac{\big|\log \frac{2^{1/2}}{q^{1/2}}-\EM \big|\O_{\varphi,\chi}\big(1)}{\log^{3/2} \vep^{-1}}\Biggr)X_{n-2}\notag\\
&\quad +\O^{X\prime}_{n-1,n-3}\quad\mbox{uniformly on $ \supp(\varphi)$.}\label{bdd:dF}
\end{align}

\noindent {\rm (2$\cc$)} Recall $\mathfrak C_{\lambda'}$ defined in \eqref{def:Clambda'}, and set
\begin{align}\label{def:alpha}
\alpha_n&\defeq\|X_n\|_\varphi,\\
\begin{split}\label{def:in}
i_{n+1}&\,\defeq\, |I^X_{n+1}(0)|+\|A_3\|_\varphi|I^X_n(0)|+\|A_4\|_\varphi|I^X_{n-1}(0)|\\
&\quad\;+\biggl(\1_{X=F}+\1_{X=D}\sum_{j=0}^4\1_{n+1-j=0}\biggr)\mc O_{\varphi,\chi,q}(\|g\|_\infty),
\end{split}\\
\delta_{\#}=\delta_{\#}(\vep)&\,\defeq\,\frac{\frac{\pi^{1/2}\slla \mc E(\varphi)\mc K\{\1\}\srra}{\slla \mc E(\varphi)\srra^{3/2} }}{\log^{1/2} \vep^{-1}}+\frac{|\mathcal O_{\varphi,\chi}(1)|}{\log \vep^{-1}}+
\frac{|\mc O_{\varphi,\chi,q}(1)|}{\log^{3/2} \vep^{-1}},\notag\\
\delta_{\#}'=\delta_{\#}'(\vep,q)&\,\defeq\,
\frac{\mathfrak C_{\lambda'}}{\log^{1/2} \vep^{-1}}-\frac{\log\frac{1}{q^{1/2}}+|\mathcal O_{\varphi,\chi}(1)|}{\log \vep^{-1}}-\frac{|\mc O_{\varphi,\chi,q}(1)|}{\log^{3/2} \vep^{-1}},\notag\\
\delta_{\#}''=\delta_{\#}''(\vep,q)&\,\defeq\,\frac{|\mathcal O_{\varphi,\chi}(1)|}{\log \vep^{-1}}+\frac{|\mathcal O_{\varphi,\chi,q}(1)|}{\log^{3/2} \vep^{-1}}\notag
\end{align}
for all $\vep\in (0,\ol{\vep}(\chi)]$.
Assume $\lla \mc E(\varphi)\mc K\{\1\}\rra\geq 0$ and $\mathfrak C_{\lambda'}\geq 0$, so that the coefficients of $\log^{-1/2}\vep^{-1}$ in $\delta_\#$ and $\delta'_\#$ are nonnegative. 

In the case of $n\in \Bbb Z$, there exists $q_0(\varphi,\chi)\in (0,\infty)$ such that 
 for all $q> q_0(\varphi,\chi)$, the following holds:
\begin{enumerate}
\item [\rm (a)] the coefficient of $\log^{-1}\vep^{-1}$ in $\delta'_\#$ is strictly positive;
\item [\rm (b)] for all sufficiently small $\vep\in (0,\ol{\vep}(\chi)]$ independent of $n\in \Bbb Z$, we have
\begin{align}\label{ineq:alpha}
\begin{split}
\alpha_{n+1}&\leq i_{n+1}+\delta_\#\alpha_n+(1-\delta_{\#}-\delta_{\#}')\alpha_{n-1}\\
&\quad+\delta_{\#}''\alpha_{n-2}+\delta_{\#}''\alpha_{n-3}+\delta_{\#}''\alpha_{n-4},\quad \forall\;n\in \Bbb Z.
\end{split}
\end{align} 
\end{enumerate}
Additionally, if $n=\infty$ and $q\in (0,\infty)$ is sufficiently large such that condition {\rm (a)} holds, and $q$ satisfies Assumption~\ref{ass:q} when $X=F$ or the condition that $D_\infty$ exists for all sufficiently small $\vep\in (0,\ol{\vep}(\chi)]$ when $X=D$, then,  for even smaller $\vep\in (0,\ol{\vep}(\chi)]$, \eqref{ineq:alpha} with $n=\infty$ and the corresponding $X$ also holds.  
\end{lem}
\begin{proof}
(1$\cc$) The proof follows a slight modification of the argument used in Lemma~\ref{lem:signexpansion} (1$\cc$). Specifically, we begin by considering the following expansion, which is implied by the definition \eqref{def:Om'} of $\mc O^{X\prime}_m$  and by
\eqref{1replace:eq} with $n+1$ replaced by $n-1$:
\begin{align}
X_{n}
=I^X_n(0)+\frac{\llv^{1/2}}{\pi}\biggl(\log \frac{2^{1/2}}{q^{1/2}\vep}-\EM\biggr)\lla \varphi X_{n-1}\rra+\llv^{1/2}\mc K\{X_{n-1}\}+\mathcal O^{X\prime}_{n-1}.\label{bdd:dF:1}
\end{align}
The second term on the right-hand side can be expanded further as follows, where the first equality 
uses \eqref{1replace}, the second equality uses \eqref{F1replace}, and the fourth equality uses \eqref{def:llvr} since we choose $\mu=1$ now:
\begin{align}
&\frac{\llv^{1/2}}{\pi}\biggl(\log \frac{2^{1/2}}{q^{1/2}\vep}-\EM\biggr)\lla \varphi X_{n-1}\rra\notag\\
&\quad =\frac{\llv}{\pi}\biggl(\log \frac{2^{1/2}}{q^{1/2}\vep}-\EM\biggr)\lla \mc E(\varphi) X_{n-2}\rra+\mc O^{X\prime}_{n-2}\notag\\
&\quad =\frac{\llv}{\pi}\biggl(\log \frac{2^{1/2}}{q^{1/2}\vep}-\EM\biggr)\lla \mc E(\varphi)\rra X_{n-2}+\mc O^{X\prime}_{n-2}+\mc O^{X\prime}_{n-3}\notag\\
&\quad =\frac{\llv \lla \mc E(\varphi)\rra\log \vep^{-1}}{\pi} X_{n-2}+\frac{\llv \lla \mc E(\varphi)\rra}{\pi} \biggl(\log \frac{2^{1/2}}{q^{1/2}}-\EM\biggr) X_{n-2}+\O^{X\prime}_{n-2,n-3}\notag\\
&\quad =\biggl(1+\frac{\lambda'}{\log^{1/2} \vep^{-1}}+\frac{\lambda+\log \frac{2^{1/2}}{q^{1/2}}-\EM}{\log \vep^{-1}}+\frac{\big|\log \frac{2^{1/2}}{q^{1/2}}-\EM \big|\O_{\varphi,\chi}\big(1)}{\log^{3/2} \vep^{-1}}\biggr)X_{n-2}\notag\\
&\quad\quad +\O^{X\prime}_{n-2,n-3}.\label{bdd:dF:2}
\end{align}
We obtain \eqref{bdd:dF} upon combining \eqref{bdd:dF:1} and \eqref{bdd:dF:2} and rearranging. \medskip

\noindent {\rm (2$\cc$)} In the case of $n\in \Bbb Z$,
by applying \eqref{bdd:dF} to $X_n-X_{n-2}$ and $X_{n-1}-X_{n-3}$ in \eqref{rec:F11}, we get
\begin{align}
X_{n+1}
&=[I^X_{n+1}(0)-A_3I^X_n(0)-A_4I^X_{n-1}(0)]+(A_1+A_3)X_n\notag\\
&\quad+(A_2+A_4)X_{n-1}-A_3\llv^{1/2}\mc K\{X_{n-1}\}\notag\\
&\quad-A_3\biggl(\frac{\lambda'}{\log^{1/2} \vep^{-1}}+\frac{\lambda+\log \frac{2^{1/2}}{q^{1/2}}-\EM}{\log \vep^{-1}}+\frac{\big|\log \frac{2^{1/2}}{q^{1/2}}-\EM \big|\O_{\varphi,\chi}\big(1)}{\log^{3/2} \vep^{-1}}\biggr)X_{n-2}\notag\\
&\quad-A_4\llv^{1/2}\mc K\{X_{n-2}\}\notag\\
&\quad-A_4\biggl(\frac{\lambda'}{\log^{1/2} \vep^{-1}}+\frac{\lambda+\log \frac{2^{1/2}}{q^{1/2}}-\EM}{\log \vep^{-1}}+\frac{\big|\log \frac{2^{1/2}}{q^{1/2}}-\EM \big|\O_{\varphi,\chi}\big(1)}{\log^{3/2} \vep^{-1}}\biggr)X_{n-3}\notag\\
&\quad+\mc O^{X\prime}_{n,n-4}.\label{bdd:dF:mainbdd0}
\end{align}
Next, we apply the formulas for $A_1+A_3$, $A_2+A_4$, $A_3$, and $A_4$ given below \eqref{rec:F11} and in \eqref{def:c}--\eqref{def:d}. Observe that $|A_2+A_4| =A_2+A_4$ as $\vep \to 0$, and that $q$ appears in $A_2$ but not in $A_1$, $A_3$, or $A_4$. By expressing $\mathcal O^{X\prime}_{n,n-4}$ in summation form as in \eqref{def:O'n1n2}, and using the definition of $\alpha_n$ from \eqref{def:alpha}, \eqref{bdd:dF:mainbdd0} yields the following inequality for all sufficiently small $\vep\in (0,\ol{\vep}(\chi)]$ independent of $n\in \Bbb Z$:
\begin{align}
\alpha_{n+1}&\leq |I^X_{n+1}(0)|+\|A_3\|_\varphi|I^X_n(0)|+\|A_4\|_\varphi|I^X_{n-1}(0)|\notag\\
&\quad+\biggl[\biggl(\frac{\frac{\pi^{1/2}\slla \mc E(\varphi)\mc K\{\1\}\srra}{\slla \mc E(\varphi)\srra^{3/2} }}{\log^{1/2} \vep^{-1}}+\frac{|\mathcal O_{\varphi,\chi}(1)|}{\log \vep^{-1}}\biggr)\alpha_n+\mathcal O^{X\prime}_n\biggr]\notag\\
&\quad+\biggl[ \biggl(1+\frac{\lambda'}{\log^{1/2} \vep^{-1}}+\frac{\log\frac{1}{q^{1/2}}+|\mathcal O_{\varphi,\chi}(1)|}{\log \vep^{-1}}\biggr)\alpha_{n-1}+\mathcal O^{X\prime}_{n-1}\biggr]\notag\\
&\quad+\biggl\{\frac{\mathcal O_{\varphi,\chi}(1)\big[1+|\log\frac{1}{q^{1/2}}|(\log \vep^{-1})^{-1/2}\big]}{\log \vep^{-1}}\alpha_{n-2}+\mathcal O^{X\prime}_{n-2}\biggr\}\notag\\
&\quad+\biggl\{\frac{\mathcal O_{\varphi,\chi}(1)\big[1+|\log\frac{1}{q^{1/2}}|(\log \vep^{-1})^{-1/2}\big]}{\log^{3/2} \vep^{-1}}\alpha_{n-3}+\mathcal O^{X\prime}_{n-3}\biggr\}+\mathcal O^{X\prime}_{n-4},\label{bdd:dF:mainbdd}
\end{align}
which is enough to get \eqref{ineq:alpha} for finite $n$ upon recalling the definitions \eqref{def:Clambda'}, \eqref{def:Om} and \eqref{def:Om'} of $\mathfrak C_{\lambda'}$, $\mathcal O^{X}_m$ and $\mathcal O^{X\prime}_m$. This proves our assertion for finite $n$. 

For $n=\infty$, \eqref{rec:F} still ensures \eqref{rec:F11}. In this case, $X_n - X_{n-2} = X_\infty - X_\infty = 0$ and $X_{n-1} - X_{n-3} = X_\infty - X_\infty = 0$. Thus, \eqref{bdd:dF:mainbdd} holds for $n=\infty$ by a more straightforward argument. This proves \eqref{ineq:alpha} with $n=\infty$ and the corresponding $X$  for all sufficiently small $\vep\in (0,\ol{\vep}(\chi)]$.
 \end{proof}

Next, we introduce an assumption and study the corresponding recursive inequalities. A detailed verification showing that this assumption generalizes the property of $\{\alpha_n\}_{n\in \Bbb Z}$ in Lemma~\ref{lem:recF} (2$\cc$) for the case $X=D$ will be provided when we end the proof of Proposition~\ref{prop:gbdd}.

\begin{ass}\label{ass:beta}
$\{\beta_n\}_{n\in \Bbb Z}$ is a sequence of nonnegative numbers satisfying 
the following inequality for all $n\in \Bbb Z$:
\begin{align}\label{def:beta}
\beta_{n+1}\leq j_{n+1}+ \delta\beta_n+(1-\delta-\delta')\beta_{n-1}+\delta''\beta_{n-2}+\delta''\beta_{n-3}+\delta''\beta_{n-4},
\end{align}
for a sequence $\{j_n\}_{n\in \Bbb Z}$ of nonnegative numbers and $\delta,\delta',\delta''\in [0,1]$, independent of $n$, such that $\delta+\delta'\leq 1$. 
 \qed 
\end{ass}

\begin{lem}\label{lem:abcd}
For all $n\in \Bbb Z$ and $k\in \Bbb Z_+$, it holds that 
\begin{align}\label{rec:a}
\beta_{n+1}\leq a_{n+1,k}+ b_k\beta_{n-k}+c_k\beta_{n-k-1}+d_k\beta_{n-k-2}+e_k\beta_{n-k-3}+f_k\beta_{n-k-4},
\end{align}
where $a_{n+1,0}\,\defeq\,j_{n+1}$, $b_0\,\defeq\,\delta$, $c_0\,\defeq\,1-\delta-\delta'$, $d_0=e_0=f_0\,\defeq\,\delta''$,  and the following recursive equations define $a_{n+1,k+1},b_{k+1},c_{k+1},d_{k+1},e_{k+1}$ for all integers $k\geq 0$: 
\begin{align}
\begin{aligned}\label{def:abcd}
a_{n+1,k+1}&\defeq\,a_{n+1,k}+b_kj_{n-k},  
&b_{k+1}&\defeq\,b_k\delta+c_k, 
& c_{k+1}&\defeq\,b_k(1-\delta-\delta')+d_k,\\
 d_{k+1}&\defeq\,b_k\delta''+e_k, 
 & e_{k+1}&\defeq\,b_k\delta''+f_k, &f_{k+1}&\defeq\, b_k\delta''.
 \end{aligned}
\end{align}
In particular, 
\begin{align}
a_{n+1,k+1}&=j_{n+1}+b_0j_{n}+\cdots+b_kj_{n-k},&& \forall\;k\geq 0,\label{rec:abcd'}\\
b_{k+1}&=b_k\delta+b_{k-1}(1-\delta-\delta')+b_{k-2}\delta''+b_{k-3}\delta''+b_{k-4}\delta'',&& \forall\;k\geq 4.\label{rec:abcd}
\end{align}
\end{lem}
\begin{proof}
We prove \eqref{rec:a} by induction on $k \in \Bbb Z_+$, for any fixed $n \in \Bbb Z$. When $k=0$, \eqref{rec:a} follows directly from \eqref{def:beta}. Suppose \eqref{rec:a} holds for some integer $k \geq 0$. This assumption yields the first inequality below, and the second inequality follows from the nonnegativity of $b_k$ and from applying \eqref{def:beta} with $n+1$ replaced by $n-k$: 
\begin{align*}
\beta_{n+1}&\leq a_{n+1,k}+ b_k\beta_{n-k}+c_k\beta_{n-k-1}+d_k\beta_{n-k-2}+e_k\beta_{n-k-3}+f_k\beta_{n-k-4}\\\
&\leq a_{n+1,k}+b_k[j_{n-k}+ \delta\beta_{n-k-1}+(1-\delta-\delta')\beta_{n-k-2}+\delta''\beta_{n-k-3}+\delta''\beta_{n-k-4}\\
&\quad +\delta''\beta_{n-k-5}]+c_k\beta_{n-k-1}+d_k\beta_{n-k-2}+e_k\beta_{n-k-3}+f_k\beta_{n-k-4}\\
&=(a_{n+1,k}+b_kj_{n-k})+(b_k\delta+c_k)\beta_{n-k-1}+[b_k(1-\delta-\delta')+d_k]\beta_{n-k-2}\\
&\quad+(b_k\delta''+e_k)\beta_{n-k-3}+(b_k\delta''+f_k)\beta_{n-k-4}+b_k\delta''\beta_{n-k-5}.
\end{align*}
By \eqref{def:abcd}, \eqref{rec:a} with $k$ replaced by $k+1$ holds. Hence, \eqref{rec:a}  holds for all $k\in \Bbb Z_+$ by induction.

Finally, to obtain \eqref{rec:abcd}, we note that by \eqref{def:abcd},
\begin{align*}
f_{k+1}&=b_k\delta'', &&\forall\;k\geq 0,\\
e_{k+1}&=b_k\delta''+b_{k-1}\delta'',&& \forall\;k\geq 1,\\
d_{k+1}&=b_k\delta''+(b_{k-1}\delta''+b_{k-2}\delta''),&&  \forall\;k\geq 2,\\
c_{k+1}&=b_k(1-\delta-\delta')+[b_{k-1}\delta''+(b_{k-2}\delta''+b_{k-3}\delta'')],&&  \forall\;k\geq 3,\\
b_{k+1}&=b_k\delta+\{b_{k-1}(1-\delta-\delta')+[b_{k-2}\delta''+(b_{k-3}\delta''+b_{k-4}\delta'')]\},&&  \forall\;k\geq 4.
\end{align*}
The last equation gives \eqref{rec:abcd}. 
\end{proof}

\begin{lem}\label{lem:geo}
Let $\{b_k\}_{k\in \Bbb Z_+}$ be the sequence defined in Lemma~\ref{lem:abcd}. Suppose that in addition to Assumption~\ref{ass:beta}, there exist $b_\star>0$ and $\theta\in (0,1)$ such that all of
the following conditions hold:
\begin{itemize}
\item [\rm (a)] $4\theta\delta'< 1$.  
\item [\rm (b)] $
14\delta''\leq (1-\theta)\delta'$.
\item [\rm (c)] $b_k\leq b_\star(1-2\theta\delta')^k$ for all $k=0,1,2,3,4$.
\end{itemize}
Then $ b_{k}\leq b_\star(1-\theta\delta'/2)^{k}$ for all $k\in \Bbb Z_+$.
\end{lem}

\begin{proof}
We prove the lemma by showing the following two implications:
\begin{align}\label{abc:imply}
[\mbox{(a), (b), (c)}]\Longrightarrow [\mbox{(a'), (b'), (c')}]\Longrightarrow b_k\leq b_\star (1-\theta\delta'/2)^k,\;\forall\; k\in \Bbb Z_+.
\end{align}
Here, conditions (a'), (b') and (c') are defined as follows: 
\begin{itemize}
\item [\rm (a')] $4\theta\delta'/(2-\delta)^2< 1$.  
\item [\rm (b')] $[(1-\rho)^{-1}+(1-\rho)^{-2}+(1-\rho)^{-3}]\delta''\leq (1-\theta)\delta'$, where
$\rho$ is a zero of $\rho^2-(2-\delta)\rho+\theta\delta'=0$, given by
\begin{align}\label{def:x0}
\rho\,\defeq\, \frac{(2-\delta)-\sqrt{(2-\delta)^2-4\theta\delta'}}{2}=\frac{(2-\delta)[1-\sqrt{1-4\theta\delta'/(2-\delta)^2}\,]}{2}.
\end{align}
\item [\rm (c')] $b_k\leq b_\star(1-\rho)^k$ for all $k=0,1,2,3,4$.
\end{itemize}
Since $\delta\in [0,1]$ by assumption, (a') implies that $\rho\in (0,1)$. 
We remark that the use of (a)--(c) is meant to save algebra when it comes to the forthcoming applications.  

For the remainder of this proof, we use the following inequalities:
\begin{gather}
1\leq 2-\delta\leq 2,\label{ineq:d}\\
\frac{x}{2}\leq 1-\sqrt{1-x}\leq x,\quad \forall\;0\leq x\leq 1,\label{ineq:x}\\
\label{rho:tbdd}
\frac{\theta\delta'}{2}\leq \frac{(2-\delta)(4\theta\delta'/(2-\delta)^2)}{4}\leq 
\rho\leq \frac{(2-\delta)(4\theta\delta'/(2-\delta)^2)}{2}=\frac{2\theta\delta'}{2-\delta}\leq 
2\theta\delta'.
\end{gather}
Here, \eqref{ineq:d} uses the condition $\delta\in [0,1]$ from Assumption~\ref{ass:beta}. Also, \eqref{rho:tbdd} holds because the first inequality uses the second inequality of \eqref{ineq:d}, the second and third inequalities are obtained from the second equality in \eqref{def:x0}, (a') and \eqref{ineq:x}, and the last inequality is obtained from the first inequality of \eqref{ineq:d}. 

Now, we show that (a), (b) and (c) imply (a'), (b') and (c').
By the first inequality of \eqref{ineq:d}, it is immediate that (a) implies (a'). By the last inequality of \eqref{rho:tbdd}, the bound $2\theta\delta'< 1/2$ due to (a), and the identity $2+2^2+2^3=14$, we see that (b) implies (b'). Note that (c) implies (c') also by the last inequality of  \eqref{rho:tbdd}. We have proved the first implication in \eqref{abc:imply}.  

To prove the second implication in \eqref{abc:imply}, by the first inequality in \eqref{rho:tbdd}, it suffices to
prove $b_\ell\leq b_\star(1-\rho)^\ell$ for all $0\leq \ell\leq k$  by an induction on integers $k\geq 4$. 
This assertion readily holds for $k=4$ by condition (c'). Suppose that for some $k\geq 4$, $b_{\ell}\leq b_\star(1-\rho)^{\ell}$ for all $0\leq \ell\leq k$. Then by the recursive equation of $b_{k+1}$ in \eqref{rec:abcd}, 
\begin{align}
b_{k+1}&\leq b_\star(1-\rho)^{k}\delta+b_\star(1-\rho)^{k-1}(1-\delta-\delta')+b_\star(1-\rho)^{k-2}\delta''+b_\star(1-\rho)^{k-3}\delta''\notag\\
&\quad\;+b_\star(1-\rho)^{k-4}\delta''\notag\\
&= b_\star(1-\rho)^{k-1}\Big\{(1-\rho)\delta+(1-\delta-\delta')+[(1-\rho)^{-1}+(1-\rho)^{-2}+(1-\rho)^{-3}]\delta''\Big\}
\notag\\
&\leq  b_\star(1-\rho)^{k-1}\left[(1-\rho)\delta+(1-\delta-\theta\delta')\right]=b_\star(1-\rho)^{k+1},\label{eq:geo}
\end{align}
where the second inequality uses condition (b'), and the last equality follows  
since the definition \eqref{def:x0} of $\rho$ shows that it solves the first equation below:
\begin{align*}
0=\rho^2-(2-\delta)\rho+\theta\delta'
&\Longleftrightarrow (1-\rho)\delta+(1-\delta-\theta\delta')=1-2\rho+\rho^2=(1-\rho)^2.
\end{align*}
By \eqref{eq:geo}, we get $b_\ell\leq b_\star(1-\rho)^\ell$ for all $0\leq \ell\leq k+1$. The required assertion 
that $b_\ell\leq b_\star(1-\rho)^\ell$ for all $0\leq \ell\leq k$ and $k\geq 4$ now follows from mathematical induction. 
\end{proof}

\begin{proof}[End of the proof of Proposition~\ref{prop:gbdd}]
(1$\cc$) We consider $\mu=1$ first, using Lemma~\ref{lem:recF} (2$\cc$) with $X=D$. We can find $q_0(\varphi,\chi)\in (0,\infty)$ large to satisfy: 
\begin{align}\label{delta:finalcond}
\begin{split}
&\delta_\#''(\vep,q)/\delta_\#'(\vep,q)\leq (3/4)/14,\quad \delta_\#(\vep),\delta_\#'(\vep,q),\delta_\#''(\vep,q)\in [0,1),\\
&\mbox{for all $q\geq q_0(\varphi,\chi)$ and, for some $\vep_0(q)\in (0,\ol{\vep}(\chi)]$, all $\vep\in  (0,\vep_0(q))$.}
\end{split}
\end{align}
By \eqref{ineq:alpha}, \eqref{def:beta} holds with $j_n=i_n$, $\beta_n=\alpha_n$, $\delta=\delta_{\#}(\vep),\delta'=\delta_{\#}'(\vep,q), \delta''=\delta_{\#}''(\vep,q)$, so that Lemma~\ref{lem:abcd} applies. Conditions (a), (b) and (c) of Lemma~\ref{lem:geo} for $\theta=1/4$ are also satisfied for some universal constant $b_\star$. Therefore,
\begin{gather}
\|D_{n+1}\|_\varphi\leq a_{n+1,n}+b_n\|D_0\|_\varphi,\quad\forall\;n\geq 4, \label{Dn:expbdd}\\
a_{n+1,n}=   b_{n-4}j_4+b_{n-3}j_3+b_{n-2}j_2+b_{n-1}j_1, \quad\forall\;n\geq 4,\label{an:expbdd}\\
b_k\leq b_\star (1-\delta'/8)^k,\quad \forall\; k\in \Bbb Z_+.\label{bk:expbdd}
\end{gather}
Here, \eqref{Dn:expbdd} uses \eqref{rec:a} for $k=n$ and the property that $\alpha_n=0$ for $n<0$  by \eqref{def:IX}; \eqref{an:expbdd} uses \eqref{rec:abcd'}, with $k=n-1$, and the property that  by \eqref{def:IX} and \eqref{def:in},
 $i_{\ell}=0$ for $\ell\geq 5$; \eqref{bk:expbdd} uses Lemma~\ref{lem:geo}. By \eqref{Dn:expbdd}--\eqref{bk:expbdd}, the bounds in Proposition~\ref{prop:gbdd} for $\mu=1$ follow as soon as we choose $q_1(\varphi,\chi)\in (q_0(\varphi,\chi),\infty)$ such that $\mathfrak C_{\lambda,q}>0$. 

For $\mu<1$, \eqref{1replace:eq1goal} with $X=D$ shows that
whether or not $\lla \mc E(\varphi)\mc K\{\1\}\rra\geq 0$ and $\mathfrak C_{\lambda'}\geq 0$,
 \eqref{ineq:alpha} applies by noting $A_2^o=\mu+\mathcal O_{\varphi,\chi}(1)/\log^{1/2}\vep^{-1}$ and
redefining $\delta_\#,\delta_\#',\delta_\#''$ to be $\delta_\#=\delta_\#''=|\O_{\varphi,\chi,q}(1)|/\log^{1/2} \vep^{-1}$ and $\delta_\#'=(1-\mu)-|\O_{\varphi,\chi,q}(1)|/\log^{1/2} \vep^{-1}$. These definitions provide \eqref{delta:finalcond} by enlarging the aforementioned $q_1(\varphi,\chi)$ if necessary. Hence, the remaining inequality of Proposition~\ref{prop:gbdd} (1$\cc$) follows as before. \medskip

\noindent {\rm (2$\cc$)}
To obtain the property that $\|F_\infty\|_\varphi<\infty$, note that for all $j\in \Bbb Z_+$,
\begin{align}
D_j(z)&=\mc T^j_\vep\{z'\mapsto G_q\{g\}(\vep z')\}(z)\notag\\
&=\int_0^\infty \frac{\e^{-qt}}{j!}\E^W_{\vep z}\left[\left(\int_0^t \lv^{1/2}\varphi_\vep(W_s)\d s\right)^{j}g(W_t)\right]\d t\label{Dj:rep}
\end{align}
by the definition \eqref{def:Tvep} of $\mathcal T_\vep$ since by the Markov property of planar Brownian motion,
\begin{align}
&\int_{(\R^2)^{j+1}}\biggl(\prod_{\ell=1}^{j}G_\nu(z_{\ell-1},z_\ell)f(z_\ell)  \biggr)G_\nu(z_{j},z_{j+1})h(z_{j+1}) \d (z_1,\cdots,z_{j+1})\notag\\
&\quad =\E^W_{z_0}\left[\int_{0<t_1<\cdots<t_{j+1}<\infty} \e^{-\nu t_{j+1}}\biggl(\prod_{\ell=1}^{j} f(W_{t_\ell})\biggr)h(W_{t_{j+1}}) \d (t_1,\cdots,t_{j+1})\right]\notag\\
&\quad =\int_0^\infty \frac{\e^{-\nu t_{j+1}}}{j!}\E_{z_0}^W\left[\left(\int_0^{t_{j+1}} f(W_{s})\d s\right)^{j}h(W_{t_{j+1}})\right]\d t_{j+1},\quad \forall\;j\in \Bbb Z_+.\label{symmI1}
\end{align}
Also, we have the following elementary bound: 
\begin{align}\label{ele-exp}
\e^{x}\leq \e+\sum_{n=0}^\infty (2n+1)\frac{x^{2n}}{(2n)!},\quad \forall\; x\in \R,
\end{align} 
which can be seen by noting that $\e^x\leq \sum_{n=0}^\infty x^{2n}/(2n)!$ for $x\leq 0$ and, for $x\geq 0$, 
\[
\e^x= \sum_{n=0}^\infty\frac{x^{2n}}{(2n)!}+\sum_{n=0}^\infty\frac{x^{2n+1}}{(2n+1)!}\leq \sum_{n=0}^\infty\frac{x^{2n}}{(2n)!}+\sum_{n=0}^\infty\frac{1}{(2n+1)!}+\sum_{n=0}^\infty(2n+2)\frac{x^{2n+2}}{(2n+2)!}.
\]
Therefore, we have the following bound of $F_\infty$, where the second inequality uses \eqref{def:F}, \eqref{ele-exp} and monotone convergence, and the equality uses \eqref{Dj:rep}:
\begin{align}
0\leq F_\infty(z)
&\leq \sum_{n=0}^\infty (2n+1)\int_0^\infty \e^{-qt}\E_{\vep z}^W\left[\frac{1}{(2n)!}\left(\int_0^{t} \lv^{1/2}\varphi_\vep (W_{s})\d s\right)^{2n}g(W_{t})\right]\d t\notag\\
&=\sum_{n=0}^\infty (2n+1)D_{2n}( z).\label{Finfty:finalbdd}
\end{align}
By \eqref{Finfty:finalbdd} and the geometric bounds in Proposition~\ref{prop:gbdd} (1$\cc$), the required property $\|F_\infty\|_\varphi<\infty$ follows upon using the fact that $\sum_{n\geq 0}n\alpha^n<\infty$ for any $\alpha\in (0,1)$. 

Finally, to show that $F = F_\infty$ on $\supp(\varphi)$, observe that $\|F_\infty\|_{\varphi} < \infty$ implies \eqref{exp:10} for all sufficiently small $\vep\in (0,\ol{\vep}(\chi)]$. Because $F$ also satisfies this recursive equation and $\|F - F_\infty\|_{\varphi} < \infty$, we can choose $D_\infty \equiv F - F_\infty$. To proceed, we apply the result for $n = \infty$ and $X = D$ from Lemma~\ref{lem:recF} (2$\cc$), a case implying $i_\infty = 0$, and rearrange \eqref{ineq:alpha} for this case. This yields
$0\leq  -\delta'_\#\alpha_\infty+3\delta''_\# \alpha_\infty$. By enlarging $q_1(\varphi,\chi)$ and decreasing $\vep(q)$ from (1$\cc$) if necessary, we see that for all $q\in (q_1(\varphi,\chi),\infty)$
 all $\vep\in (0,\vep(q)]$, $-\delta'_\#+3\delta''_\#< 0$, and so, $0=\alpha_\infty=\|F - F_\infty\|_\varphi $. That is, $F = F_\infty$ on $\supp(\varphi)$. Finally, the series representation $F = \sum_{j=0}^\infty \mc T^j_\vep\{z'\mapsto G_q\{g\}(\vep z')\}$ on $\supp(\varphi)$ follows from the definition of $F$ and (1$\cc$). This completes the proof.
\end{proof}

\addtocontents{toc}{\protect\setcounter{tocdepth}{1}}
\subsection{Range of $\mathfrak C\bs (\bs \varphi\bs )$}\label{sec:range}
\addtocontents{toc}{\protect\setcounter{tocdepth}{2}}
\addcontentsline{toc}{subsection}{\numberline{\thesubsection}{Range of $\mathfrak C(\varphi)$}}

Recall that $\mathfrak C(\varphi)$ is defined by \eqref{def:Cphi}.  We close Section~\ref{sec:=0}
by a specialized version of the property that the range of $\mathfrak C(\varphi)$ is $\R$ as $\varphi$ ranges over bounded, compactly supported functions with $\int\varphi=0$. For the proof, 
we work with the following class of functions:
\begin{align}\label{def:phiR}
\varphi_{R}(z)\,\defeq\,  \1_{B_R}(z)-R^2\1_{B_1}(z)=(1-R^2)\1_{B_1}(z)+\1_{A_{R}}(z),\quad  1<R<\infty,
\end{align}
where $B_r\,\defeq \,\{z\in \Bbb C;|z|\leq r\}$, and $A_{R}\,\defeq\, \{z\in \Bbb C;1<|z|\leq R\}$. 

\begin{prop}
$\{\mathfrak C( \pm \varphi_R);1<R<\infty\}=\R$.
\end{prop}
\begin{proof}
We start by computing $\lla \varphi_R (\kappa \star \varphi_R)^j \rra$ for $j = 1, 2$, where $\varphi_R$ is defined in \eqref{def:phiR}. This is motivated by the following reformulation of \eqref{def:Cphi}: 
\begin{align}\label{Cphi:formula}
\mathfrak C(\varphi)=\pi^{1/2}\lla \varphi(\kappa\star \varphi)^2\rra/\lla \varphi(\kappa\star\varphi)\rra^{3/2},
\end{align} 
 where $\kappa (z)\,\defeq\,\pi^{-1}\log |z|^{-1}$ is defined in \eqref{def:energy}. To find those 
 $\lla \varphi_R (\kappa \star \varphi_R)^j \rra$, recall 
\begin{align}\label{def:kint}
\int_{B_r}\kappa(z-z')\d z'=\frac{(r^2-|z|^2)^+}{2}+r^2\log \left(\frac{1}{|z|\vee r}\right),\quad z\in \Bbb C
\end{align}
(cf. \cite[Proposition~4.9, p.75]{PS:BMP}). Hence, 
\begin{align}
\lla \varphi_{R} (\kappa\star \varphi_{R})^j\rra&=(1-R^2)\int_{|z|\leq 1}\kappa\star \varphi_R(z)^j\d z+\int_{1<|z|\leq R}\kappa\star \varphi_R(z)^j\d z\label{Cphi:1}\\
\begin{split}\label{Cphi:2}
&=(1-R^2)\int_{|z|\leq 1}\left(\frac{(R^2-1)|z|^2}{2}+R^2\log \left(\frac{1}{R}\right)\right)^j\d z\\
&\quad +\int_{1<|z|\leq R}\left(\frac{R^2-|z|^2}{2}+R^2\log \left(\frac{|z|}{R}\right)\right)^j\d z.
\end{split}
\end{align}
Specifically, we obtain \eqref{Cphi:1} by applying the second equality of \eqref{def:phiR}, and \eqref{Cphi:2} by \eqref{def:phiR} and \eqref{def:kint}.
To solve the right-hand side of \eqref{Cphi:2}, note that, by the polar coordinates, $\int f(|z|)\d z=\int 2\pi r f(r)\d r$ for any $f$, so some algebra shows
\begin{align}
\lla \varphi_{R} (\kappa\star \varphi_{R})\rra&=\frac{1}{2}\pi (R^2-R^4)+\pi R^4\log R,\label{eg:1}\\
\lla \varphi_{R} (\kappa\star \varphi_{R})^2\rra&=\frac{\pi}{8}R^2(R^4-1)+\frac{R^4}{2}\pi\left(R^2-2\right)\log (R)\notag\\
&\quad -\pi (R^6-R^4+R^2)\log^2(R).\label{eg:2}
\end{align}
In more detail, to obtain \eqref{eg:2}, we first establish the following formulas:
\begin{align*}
&(1-R^2)\int_0^1 2\pi r\left(\frac{(R^2-1)r^2}{2}+R^2\log \left(\frac{1}{R}\right)\right)^2\d r\\
&\quad=(1-R^2)\left(\frac{\pi R^4}{12}+\pi R^4\log^2 R-\frac{\pi}{2}R^4\log R-\frac{\pi R^2}{6}+\frac{\pi}{2}R^2\log R+\frac{\pi}{12}\right),\\
&\int 2\pi r\left(\frac{R^2-r^2}{2}+R^2\log \left(\frac{r}{R}\right)\right)^2 \d r\\
&\quad =\frac{\pi r^2}{24}\left(24 R^2\log^2\left(\frac{r}{R}\right)+6R^4-3R^2r^2-12 R^2r^2\log\left(\frac{r}{R}\right)+2r^4\right)+C.
\end{align*}

To prove the required property that $\{\mathfrak C( \pm \varphi_R);1<R<\infty\}=\R$, observe $\mathfrak C(\lambda \varphi) = \sgn(\lambda)\mathfrak C(\varphi)$ for any $\lambda \in \R \setminus \{0\}$. Thus, it is enough to show that $\{\mathfrak C(\varphi_R);1 < R < \infty \} \supset (-\infty, 0]$. We then apply \eqref{Cphi:formula} and note the following properties: (1) by  \eqref{eg:1} and \eqref{eg:2}, $\mathfrak C(\varphi_R)\sim \pi^{1/2}(-\pi R^6\log^2R)/(\pi R^4\log R)^{3/2}\to-\infty$ as $R\to\infty$; (2) $\lla \varphi_R(\kappa\star \varphi_R)^2\rra=0$ for some $R>1$ since by \eqref{eg:2}, $\lim_{R\searrow 1}\lla \varphi_R(\kappa\star \varphi_R)^2\rra=0$, 
$\lla \varphi_R(\kappa\star \varphi_R)^2\rra|_{R=1.01}\approx 4.18754\times 10^{-6}$, so the intermediate value theorem ensures the existence of a root of $\lla \varphi_R(\kappa\star \varphi_R)^2\rra=0$ in $(1,\infty)$. By the intermediate value theorem again, these two properties imply $\{\mathfrak C( \varphi_R);1<R<\infty\}\supset(-\infty,0]$. The proof is complete.
\end{proof}

\section{Asymptotics of time-changed functionals}\label{sec:>0}
In this section, we turn to the proofs of the main results in Section~\ref{sec:KRreg}. The proof of Theorem~\ref{thm:>0} (1$\cc$) is obtained at the end of Section~\ref{sec:level} as we will prove \eqref{def:Gf} and \eqref{kernel:expexp} for more general Markov processes first. Regarding Theorem~\ref{thm:>0} (2$\cc$), the expansion is an immediate consequence of Proposition~\ref{prop:5}, which gives more precise asymptotic representations and satisfies a stronger mode of convergence. 

{\bf Throughout this section, unless otherwise mentioned, we will work with $\varphi,\lv,g$ such that $\lla \varphi\rra>0$, (\ref{def:lvo}) holds, and $g$ is nonnegative and bounded.} Additional conditions on $\varphi$ will be imposed when needed. 

\subsection{Excursions in general radial-type Green's functions}\label{sec:level}
Our goal in the remainder of Section~\ref{sec:level} is to prove Theorem~\ref{thm:excursion}. The proof of \eqref{formula:excursion} is central. It will extend a method in \cite[pp.120--121]{Bertoin}, which shows how the excursion measure of $\{X_t\}$ at level $b$ can represent the $\E_b$-expectation of $\e^{A_f(t)}$ \emph{after the random time change} by the inverse local time at level $b$. In particular, we will use excursions at \emph{all} levels $b$, which contrasts with the proof of \eqref{Kac} in \cite{JPY} that uses excursions at the level $b=0$ and represents the Green's function in \eqref{Kac} accordingly. Throughout Section~\ref{sec:level}, we write $\P_a$ and $\E_a$ for $\P^X_a$ and $\E^X_a$, respectively. 

More specifically, the proof of  \eqref{formula:excursion} 
relies on the inverse local times and excursion processes at all levels $b\in E$. To state the basic properties of these tools, we write $\mathscr E_b$ for the set of $\epsilon\in D(\R_+,E)$ such that $\epsilon(t)\neq b$ for all $b\in (0,\zeta)$ and $\epsilon(t)=b$ for all $t\geq \zeta$, where $\zeta=\zeta(\epsilon)=\inf\{t\geq 0;\epsilon(t)=b\}$. Then we recall the following properties:
\begin{itemize}
\item Recall that under $\P_a$ for any $a\in E$, the inverse local time at level $b$ $\tau^b_\ell$ is defined as in \eqref{def:invLT}. Under $\P_b$, $\{\tau^b_\ell;\ell\geq 0\}$ is a subordinator killed at the independent exponential random variable $L^b_\infty$ with mean $\E_b[L^b_\infty]=\N_b(\zeta=\infty)^{-1}$
such that 
\begin{align}\label{Phib:Lap}
-\log\E_b[\e^{-\nu\tau^b_1}]= \Phi_b(\nu)\,\defeq\,\mathtt d(b)\nu+\N_b(1-\e^{-\nu\zeta}),\quad \forall\; \nu\in (0,\infty),
\end{align}
where, for fixed $\partial\notin D(\R_+,E)$, $\N_b$ is the excursion measure on $\mathscr E_b\cup \{\partial\}$ satisfying $\N_b(\{\partial\})=0$ and $\N_b(\zeta>\delta)<\infty$ for all $\delta>0$ \cite[Theorem~8 on p.114]{Bertoin}. Note that $\mathtt d(b)$ in \eqref{Phib:Lap} is also the unique constant satisfying 
\begin{gather}
\textstyle \int_0^t\1_{\{b\}}(X_s)\d s =\mathtt d(b)L^b_t\quad\forall\;t\geq 0\quad\mbox{$\P_b$-a.s.}\label{def:drift}
\end{gather}  
\cite[Corollary~6 on p.112]{Bertoin}.

\item The excursion process $\{e^b_\ell;0<\ell<\infty\}$ at level $b$ is defined under $\P_b$ as an $\ms E_b\cup \{\partial\}$-valued point process such that
\[
e^b_\ell=\{e^b_{\ell}(t);t\geq 0\}\,\defeq 
\begin{cases}
\big\{X_{\tau^b_{\ell-}+t\wedge (\tau^b_\ell-\tau^b_{\ell-})};t\geq 0\big\},& \mbox{if }\tau^b_\ell-\tau^b_{\ell-}>0,\\
 \partial, &\mbox{ otherwise}.
\end{cases}
\]
The law can be characterized according to the following dichotomy: 
\begin{itemize}
\item If $b$ is recurrent (i.e. $\{t;X_t=b\}$ is a.s. unbounded), then $\N_b(\zeta=\infty)=0$ and
$\{e^b_\ell\}$ is a Poisson point process with characteristic measure $\N_b$.  
\item If $b$ is transient (i.e. $\{t;X_t=b\}$ is a.s. bounded), then $\N_b(\zeta=\infty)\in (0,\infty)$ and a chronological concatenation of a sequence of i.i.d. copies of $\{e^b_{\ell\wedge L^b_\infty}\}$ defines a Poisson point process with characteristic measure $\N_b$.
\end{itemize}
See \cite[Theorem~10 on p.118]{Bertoin}.
\end{itemize}
For the following proofs, we recall the notations in \eqref{def:Uf}, Assumption~\ref{ass:X} and \eqref{def:Pf}.\medskip 

\begin{proof}[Proof of Theorem~\ref{thm:excursion}~(1$\cc$)]
The key equation for the proof is given by
\begin{align}
&\E_a\left[\int_0^\infty \e^{-\nu t}\e^{A_f(t)}g(X_t)\d t\right]\notag\\
&\quad =\int_EQ^f_{\nu,b}(a)\left(\int_0^{\infty}   \E_b[ \e^{-\nu\tau^b_\ell+ A_f(\tau^b_\ell)};\ell<L^b_\infty]\d\ell\right) g(b)\mathbf m(\d b).\label{excursion:id}
\end{align}
To obtain \eqref{excursion:id}, first, we compute
\begin{align*}
\int_0^\infty \e^{-\nu t}\e^{A_f(t)}g(X_t)\d t
&=\int_E\int_0^\infty \e^{-\nu t+A_f(t)}\d L^b_tg(b)\mathbf m(\d b)\\
&=\int_E\int_0^{L^b_\infty} \e^{-\nu \tau^b_\ell+A_f(\tau^b_\ell)}\d\ell g(b)\mathbf m(\d b),
\end{align*}
where the first equality follows from the monotone class theorem and Assumption~\ref{ass:X} (d), and the second follows by changing variables~\cite[Proposition~4.9 on p.8]{RY}.
To take the expectation of the last term, note that we have
the identity 
$\tau^b_\ell=T_b+\tau^b_\ell\circ\theta_{T_{b}}$ on $\{T_b<\infty\}$ ($\theta_t$ is the shift operator). Hence, we get \eqref{excursion:id} upon taking expectations of both sides of the last equality and using the strong Markov property of $\{X_t\}$ at time $T_b$ on $\{T_b<\infty\}$.
To get \eqref{formula:excursion} from \eqref{excursion:id}, it remains to
 show that, with $\Phi_b$ from \eqref{Phib:Lap}, 
\begin{align}\label{Psi:exp}
\E_b[ \e^{-\nu\tau^b_\ell+A_f(\tau^b_\ell)};\ell<L^b_\infty]=\e^{-\ell \Phi_b(\nu)[1-U_\nu\{fQ^f_{\nu,b}\}(b)]},\quad \ell\geq 0.
\end{align} 
This reduction follows since \eqref{oct} and \eqref{Phib:Lap} imply $U_\nu(b,b)=\Phi_b(\nu)^{-1}$; the strict positivity of $U_\nu(b,b)$ is shown in the proof of (2$\cc$) below. Also, when $1-U_\nu\{fQ^f_{\nu,b}\}(b)\leq 0$, the Lebesgue integral of the right-hand side of \eqref{Psi:exp} over $\ell\in [0,\infty)$ diverges, which is consistent with the convention mentioned below \eqref{formula:excursion}.

To obtain \eqref{Psi:exp}, we first show that
\begin{align}\label{Psi:exp1}
\E_b\big[ \e^{-\nu\tau^b_\ell+A_f(\tau^b_\ell)};\ell<L^b_\infty\big]=\e^{-\ell [\nu-f(b)]\mathtt d(b)-\ell \mathbf N_b(1-\e^{-\nu \zeta+A_f(\zeta)})},\quad \ell\geq 0,
\end{align}
with the convention that $\e^{-\nu \zeta+A_f(\zeta)}\,\defeq\, 0$ if $\zeta=\infty$. Note that
\begin{align}\label{Nb:dec}
\mathbf N_b(1-\e^{-\nu \zeta+A_f(\zeta)})= \mathbf N_b((1-\e^{-\nu \zeta+A_f(\zeta)})\1_{\{\zeta<\infty\}})+\mathbf N_b(\zeta=\infty),
\end{align}
and $\mathbf N_b(1-\e^{-\nu \zeta+A_f(\zeta)})$ must be well-defined as an $[-\infty,\infty)$-valued integral. 

Now, to get \eqref{Psi:exp1},
by monotone convergence, we can assume that $f$ is bounded above. This assumption allows, for all $0<L<R<\infty$, 
\begin{align}\label{Nb:finite}
\N_b(|1-\e^{-[\nu\zeta -A_f(\zeta)]\1_{\{L<\zeta\leq R\}}}|)=\N_b(|1-\e^{-\nu\zeta +A_f(\zeta)}|\1_{\{L<\zeta\leq R\}})<\infty,
\end{align}
since $\N_b(\zeta>\delta)<\infty$ for all $\delta>0$. Hence, 
we can evaluate the expectation in \eqref{Psi:exp1} as follows. 
Rewrite $-\nu \tau^b_\ell+ A_f(\tau^b_\ell)$ first: for $\ell<L^b_\infty$ and any partition $\{S_n\}_{n=0}^\infty$ of $(0,\infty)$ by intervals such that $ \overline{S}_n\subset (0,\infty)$, 
\begin{align}
-\nu \tau^b_\ell+ A_f(\tau^b_\ell)
&=\int_0^{\tau^b_\ell}-[\nu-f(b)]\1_{\{b\}}(X_s) \d s+\int_0^{\tau^b_\ell}-[\nu-f(X_s)]\1_{\{b\}^\complement}(X_s)\d s\notag\\
&=-[\nu-f(b)]\ell\mathtt d(b)+\sum_{n=0}^\infty\sum_{r:0< r\leq \ell}-[\nu \zeta(e^b_r)-A_f(e^b_r)]\1_{\{\zeta(e^b_r)\in S_n\}}.\label{excursion:dec}
\end{align}
Here, the first term in the last equality follows from \eqref{def:drift}, and the second term uses the definition of the excursion process $\{e^b_\ell\}$ and the assumed properties of $\{S_n\}$. 
Recall that Poisson random measures over disjoint sets are independent. Hence, by \eqref{Nb:finite} and \eqref{excursion:dec}, the exponential formula for Poisson point processes \cite[p.8]{Bertoin} applies and gives an equivalent form of \eqref{Psi:exp1} [recall the property of $\mathbf N_b(1-\e^{-\nu \zeta+A_f(\zeta)})$ mentioned right below \eqref{Nb:dec}], and hence, the justification of \eqref{Psi:exp1}:
\[
\E_b\big[ \e^{-\nu\tau^b_\ell+A_f(\tau^b_\ell)};\ell<L^b_\infty\big]=\e^{-\ell [\nu-f(b)]\mathtt d(b)}\Biggl(\prod_{n=0}^\infty \e^{-\ell \N_b(1-\e^{-[\nu\zeta -A_f(\zeta)]\1_{\{\zeta\in S_n\}}})}\Biggr)\e^{-\ell \N_b(\zeta=\infty)}.
\]

To show the equivalence between \eqref{Psi:exp1} and \eqref{Psi:exp}, we resume the use of general locally bounded $f$ and write
\begin{align}
&\N_b(1-\e^{-\nu\zeta+A_f(\zeta)})\notag\\
&\quad =\N_b(1-\e^{-\nu \zeta})+
\N_b(\e^{-\nu \zeta}-\e^{-\nu \zeta+A_f(\zeta)})\notag\\
&\quad =[\Phi_b(\nu)-\mathtt d(b)\nu]-
\N_b\left(\e^{-\nu \zeta}\int_0^\zeta  f(\epsilon_s) \e^{\int_s^\zeta f(\epsilon_r)\d r}\d s\1_{\{\zeta<\infty\}}\right)\label{Psi:exp2}
\end{align}
by using \eqref{Phib:Lap} and the fundamental theorem of calculus for $s\mapsto\e^{\int_s^\zeta f(\epsilon_r)\d r}$. Note that since $\mathbf N_b(1-\e^{-\nu \zeta+A_f(\zeta)})$ must be $[-\infty,\infty)$-valued, as explained below \eqref{Nb:dec}, we have
\begin{align}\label{Nbf-}
\N_b\left(\e^{-\nu \zeta}\int_0^\zeta  f^-(\epsilon_s) \e^{\int_s^\zeta f(\epsilon_r)\d r}\d s\1_{\{\zeta<\infty\}}\right)<\infty, 
\end{align}
where $f^-$ denotes the negative part of $f$, so the last term in \eqref{Psi:exp2} is well-defined and $(-\infty,\infty]$-valued. 
Moreover, this term can be written as
\begin{align}
&\N_b\left(\e^{-\nu \zeta}\int_0^\zeta  f(\epsilon_s) \e^{\int_s^\zeta f(\epsilon_r)\d r}\d s\1_{\{\zeta<\infty\}}\right)\notag\\
&\quad =\N_b\left(\int_0^\zeta \e^{-\nu s}f(\epsilon_s)\E_{\epsilon_s}[\e^{-\nu T_b+A_f(T_b)};T_b<\infty]\d s\right)\notag\\
&\quad =\Phi_b(\nu)\E_b\left[\int_0^\infty  \e^{-\nu s} f(X_s)\E_{X_s}[\e^{-\nu T_b+A_f(T_b)};T_b<\infty]\d s\right]\notag\\
&\quad\quad  -\Phi_b(\nu)\E_b\left[\int_0^\infty \e^{-\nu s}f(b)\1_{\{b\}}(X_s)\d s\right]\notag\\
&\quad =\Phi_b(\nu)U_\nu\{fQ^f_{\nu,b}\}(b)-f(b)\mathtt d(b).\label{Psi:exp3}
\end{align}
Here, the first equality follows from the Markov property of $\{\epsilon_t\}$ at $s$ under the excursion measure \cite[(3.28) Theorem on pp.102--103]{Blumenthal}; the second equality can be deduced from the compensation formula  for Poisson points processes as in \cite[p.120]{Bertoin}; \eqref{Psi:exp3} holds by the following computation where we apply \eqref{def:drift} and \eqref{Phib:Lap} in that order:
\begin{align*}
\E_b\left[\int_0^\infty \e^{-\nu s}\1_{\{b\}}(X_s)\d s\right]
&=\E_b\left[\int_0^\infty \e^{-\nu s}\mathtt d(b)\d L^b_s\right]=\frac{\mathtt d(b)}{\Phi_b(\nu)}.
\end{align*}
Now, the equivalence between \eqref{Psi:exp} and \eqref{Psi:exp2} is immediate upon applying \eqref{Psi:exp2} and \eqref{Psi:exp3} to \eqref{Psi:exp1} and noting that 
\[
[\nu-f(b)]\mathtt d(b)+[\Phi_b(\nu)-\mathtt d(b)\nu]-\Phi_b(\nu)U_\nu\{fQ^f_{\nu,b}(b)\}+f(b)\mathtt d(b)=\Phi_b(\nu)[1-U_\nu\{fQ^f_{\nu,b}\}(b)].
\]
We have proved \eqref{formula:excursion}. The property that $U_\nu\{fQ^f_{\nu,b}\}(b)$ is well-defined and $(-\infty,\infty]$-valued can be seen upon applying \eqref{Nbf-} to \eqref{Psi:exp3}.
\end{proof}

\begin{proof}[Proof of Theorem~\ref{thm:excursion}~(2$\cc$)]
First, to see $Q_{\nu,b}(a)=U_\nu(a,b)/U_\nu(b,b)$, note that   
$U_\nu(a,b)=Q_{\nu,b}(a)U_\nu(b,b)$ can be deduced from the strong Markov property of $\{X_t\}$ at $T_b$ on $\{T_b<\infty\}$ and the assumed joint continuity of the resolvent densities [c.f. the proof of \eqref{psiab:dec2} below]. Also, the property that $U_\nu(b,b)>0$ for all $b\in E$ follows from the same proof of \cite[Lemma~3.4.6 on p.81]{MR:Markov}, now that Assumption~\ref{ass:X} is imposed. 

Next, to prove \eqref{psiab:dec0}, it is enough to show that
\begin{gather}
Q^f_{\nu,b}(a)
=Q_{\nu,b}(a)+\E_a\left[\int_0^{T_b}\e^{-\nu s}f(X_s)Q^f_{\nu,b}(X_s)\d s\right],\label{psiab:dec1}\\
\E_a\left[\int_0^{T_b}\e^{-\nu s}f(X_s)Q^f_{\nu,b}(X_s)\d s\right]+Q_{\nu,b}(a)U_\nu\{fQ^f_{\nu,b}\}(b)=U_\nu\{fQ^f_{\nu,b}\}(a).
\label{psiab:dec2}
\end{gather}
This is so since \eqref{psiab:dec0} holds trivially if $U_\nu\{fQ^f_{\nu,b}\}(b)=\infty$. To see this property, note that the following proof shows that the last term of \eqref{psiab:dec1} must be $(-\infty,\infty]$-valued, and so, 
 by \eqref{psiab:dec2}, $U_\nu\{fQ^f_{\nu,b}\}(b)=\infty$ implies $U_\nu\{fQ^f_{\nu,b}\}(a)=\infty$.
 
Now, to obtain \eqref{psiab:dec1}, write
\begin{gather}\label{FOTC}
\e^{A_{f}(t)}=1+\int_0^{t}f(X_s)\e^{A_{f}(t)-A_{f}(s)}\d s.
\end{gather}
By using \eqref{FOTC} with $t=T_b$, 
multiplying both sides by $\e^{-\nu T_b}$, and taking expectations,
\begin{align}
Q^f_{\nu,b}(a)&=Q_{\nu,b}(a)+\int_0^\infty \E_a[\e^{-\nu s}f(X_s)\e^{-\nu (T_b-s)+A_{f}(T_b)-A_{f}(s)};s<T_b<\infty]\d s.
\end{align}
We obtain \eqref{psiab:dec1} by applying the strong Markov property at time $s$ to the last term. 
Note that the proof of \eqref{psiab:dec1} implies 
$\E_a[\int_0^{T_b}\e^{-\nu s}f^-(X_s)Q^f_{\nu,b}(X_s)\d s]<\infty$,
so that the last term in \eqref{psiab:dec1} is well-defined and $(-\infty,\infty]$-valued. 

Regarding \eqref{psiab:dec2},
we expand the last term in \eqref{psiab:dec1} by writing $\int_0^\infty=\int_0^{T_b}+\int_{T_b}^\infty$ and the strong Markov property of $\{X_t\}$ at $T_b$ on $\{T_b<\infty\}$, so that 
\begin{align}
\begin{split}
U_\nu\{fQ^f_{\nu,b}\}(a)&=
\E_a\left[\int_0^{T_b}\e^{-\nu s}f(X_s)Q^f_{\nu,b}(X_s)\d s\right]\\
&\quad +\E_a[\e^{-\nu T_b};T_b<\infty]U_\nu\{fQ^f_{\nu,b}\}(b),\label{resolvent:rec}
\end{split}
\end{align}
which is \eqref{psiab:dec2} by using the notations $Q_{\nu,b}$ and $U_\nu$. The proof is complete.
\end{proof}

\subsection{Asymptotic expansions under BES$^{ 2}$}\label{sec:3asymp}
In this subsection, we prove Proposition~\ref{prop:5}, which provides three asymptotic expansions; see \eqref{asymp:U}--\eqref{Qf:exp0final}. Together, these expansions not only imply but also refine the asymptotic expansion \eqref{lim:aux} in Theorem~\ref{thm:>0}, by establishing the stronger property of uniform convergence as defined below. We will use this stronger form of convergence in Section~\ref{sec:radial}, alongside $\little{o}$- and $\mathcal O$-terms.

\begin{defi}\label{def:unif}
We say that $h_\vep(t)$ converges to $h(t)$ uniformly in $t \in \Gamma_\vep$ as $\vep \to 0$ if $\sup_{t} |(h_\vep - h)(t) \1_{\Gamma_\vep}(t)| \to 0$. \qed 
\end{defi}

In the sequel, we fix a choice of $M_\varphi$ satisfying the following condition:
\begin{align}
 \label{def:M}
\begin{split}
\mbox{$M_\varphi\in(1,\infty)\mbox{ and }\supp(\varphi)\subseteq \{z\in \Bbb C;|z|\leq M_\varphi\}$}.
\end{split}
\end{align}
Additionally, recall $\oovarphi\uvep$, $R^f_{\nu,b}(a)$ and $\lv$ defined in \eqref{def:phiexp}, \eqref{def:TR} and \eqref{def:lvo}. 

\begin{prop}\label{prop:5}
Fix $q\in (0,\infty)$, and define
\begin{align}
\Gamma_{\geq }(\vep,M_\varphi)&\,\defeq\,\{(a,b)\in\R\times \R;M_\varphi\geq a\geq b\geq M_{\vep}\},\label{Ggeq}\\
\Gamma_{\leq }(M_\varphi)&\,\defeq\, \{(a,b)\in \R\times \R;0<a\leq b\leq M_\varphi\},\label{Gleq}\\
M_{\vep}&\,\defeq \,
(M_\varphi/2)\wedge [\log (\log \vep^{-1})]^{-1}.
\label{def:Mvep}
\end{align}
Then uniformly in $(a,b)\in \Gamma_{\geq}(\vep,M_\varphi)\cup \Gamma_{\leq}(M_\varphi)$ as $\vep\to 0$, 
\begin{align}
\begin{split}
V_{q\vep^2}(a,b)&= \frac{b}{\pi}\int_{-\pi}^\pi\biggl(\log \frac{2^{1/2}}{q^{1/2}\vep|a-b\e^{\i \theta}|}-\gamma_{\sf EM} \biggr)\d \theta\\
&\quad\;+ b\cdot \little{o}_{q,M_\varphi}\left(\frac{1}{\log \vep^{-1}}\right),\label{asymp:U}
\end{split}\\
\begin{split}
V_{q\vep^2}\{\lv\oovarphi\uvep R^{\lv\oovarphi\uvep}_{q\vep^2,b}\}(a)
& =\mu+\frac{A_{\mu,\lambda,q}(a,b)}{\log \vep^{-1}}
+\little{o}_{\varphi,\chi,q,M_\varphi}\left(\frac{1}{\log \vep^{-1}}\right),
\label{eq:finalI:rad}
\end{split}\\
\begin{split}
R_{q\vep^2,b}^{\lv\oovarphi\uvep}(a)&=1-\frac{(1-\mu)\log^-(b/a)+A_{\mu,\lambda,q}(b,b)- A_{\mu,\lambda,q}(a,b)}{\log \vep^{-1}}\\
&\quad\;+\little{o}_{\varphi,\chi,q,M_\varphi}\left(\frac{1}{\log \vep^{-1}}\right),\label{Qf:exp0final}
\end{split}
\end{align}
where $A_{\mu,\lambda,q}(a,b)$ is defined in \eqref{def:Azb}, and $\chi\,\defeq\,(\mu,\lambda)$.
\end{prop}

\begin{rmk}
The formulation of Definition~\ref{def:unif} and the selection of $M_\vep$ arise from limitations in some a-priori bounds for $R_{q\vep^2,b}^{\lv\oovarphi\uvep}(a)$, which are clarified in Lemma~\ref{lem:3asymp}.\qed 
\end{rmk}

\begin{proof}[Proof of (\ref{lim:aux}) of Theorem~\ref{thm:>0}]
For fixed $a,b,\mu,\ell\in (0,\infty)$ and $\lambda\in \R$, 
we have the following particular consequence of \eqref{kernel:expexp}, choosing $f\equiv \lv \oovarphi\uvep$ and $\nu\equiv q\vep^2$:
\begin{align*}
\E^{\rho}_{a}[\e^{-q\vep^2\tau^b_\ell+A_{\lv \oovarphi\uvep}(\tau^b_\ell)}]
&=R^{\lv \oovarphi\uvep}_{q\vep^2,b}(a)\exp\Biggl\{-\ell\Biggl(\frac{1-V_{q\vep^2}\{\lv \oovarphi\uvep V^{\lv \oovarphi\uvep}_{q\vep^2,b}\}(b)}{V_{q\vep^2}(b,b)}\Biggr)\Biggr\}.
\end{align*}
The required identity follows upon applying \eqref{asymp:U}--\eqref{Qf:exp0final} to the right-hand side.
\end{proof}

\begin{proof}[Proof of (\ref{asymp:U}) of Proposition~\ref{prop:5}]
Take $\nu=q\vep^2$ in \eqref{UG:polar} and apply \eqref{asymp:G}. The $\little{o}$-term in \eqref{asymp:U} holds since $(a,b)\in \Gamma_{\geq}(\vep,M_\varphi)\cup \Gamma_{\leq}(M_\varphi)$.
\end{proof}

To prove \eqref{eq:finalI:rad} and \eqref{Qf:exp0final}, we estimate $R_{q\vep^2,b}^{\lv\oovarphi\uvep}(a)$ using the following recursion, derived by 
applying Theorem~\ref{thm:>0} (1$\cc$) with the choice of  $U_\nu\equiv V_\nu$, $Q^f_{\nu,b}(a)\equiv R^f_{\nu,b}(a)$, $Q_{\nu,b}\equiv R_{\nu,b}$, $f\equiv \lv \oovarphi\uvep$ and $\nu\equiv q\vep^2$ for \eqref{psiab:dec}:
\begin{gather}\label{psiab:decr}
\begin{split}
R^{\lv \oovarphi\uvep}_{q\vep^2,b}(a)=R^0_{q\vep^2,b}(a)\times [1-V_{q\vep^2}\{\lv \oovarphi\uvep R^{\lv \oovarphi\uvep}_{q\vep^2,b}\}(b)]+V_{q\vep^2}\{\lv \oovarphi\uvep R^{\lv \oovarphi\uvep}_{q\vep^2,b}\}(a)\\
\mbox{under the condition of $V_{q\vep^2}\{\lv \oovarphi\uvep R^{\lv \oovarphi\uvep}_{q\vep^2,b}\}(b)<\infty$. }
\end{split}
\end{gather}
Here, for all $\nu>0$,  we have the following formula: 
\begin{align}\label{BES:time}
R^0_{\nu,b}(a)=
\begin{cases}
\displaystyle K_0(a\sqrt{2\nu})/K_0(b\sqrt{2\nu}),&a\geq b>0,\\
\displaystyle I_0(a\sqrt{2\nu})/I_0(b\sqrt{2\nu}),&0<a\leq b,
\end{cases}
\end{align}
where $K_0(\cdot)$ is the Macdonald function of order zero, and $I_0(\cdot)$ is the modified Bessel function of the first kind of order zero; see \cite{HM} for \eqref{BES:time}.

Lemma~\ref{lem:3asymp} sets the stage for applying the recursion in \eqref{psiab:decr}. It analyzes the coefficients involved and provides preliminary bounds and a first-order asymptotic formula for $ R_{q\vep^2,b}^{\lv \oovarphi\uvep}(a)$. The approach for establishing these bounds relies on the following simple estimates, which follow directly from the definition of $R^f_{\nu,b}(a)$ in \eqref{def:TR} and the selection of $M_\varphi$ in \eqref{def:M}: 
\begin{align}\label{eq:4-0}
R^{0}_{q\vep^2+\|(\lv \oovarphi\uvep)^-\|_\infty,b}(a)
\leq
 R_{q\vep^2,b}^{\lv \oovarphi\uvep}(a)\leq R_{0,b}^{\|(\lv \oovarphi\uvep)^+\|_\infty\1_{[0,M_\varphi]}(|\cdot|)}(a),
\end{align}
where $f^-$ and $f^+$ denote the negative and positive parts of a real-valued function $f$. In the sequel, $\beta$ denotes a one-dimensional standard Brownian motion, and $T_b(\beta)$ is defined as in \eqref{def:TbX}

\begin{lem}\label{lem:3asymp}
\noindent {\rm (1$\cc$)} For $\nu>0$,
\begin{align}\label{BES:time-ext}
\E^\rho_a[\e^{\nu T_b}]&=\frac{J_0(a\sqrt{2\nu })}{J_0(b\sqrt{2\nu})},\quad\, 0<a\leq b<\frac{j_{0,1}}{\sqrt{2\nu}},\\
 \E^\beta_a[\e^{\nu \int_0^{T_b(\beta)}\1_{(-\infty,0]}(\beta_s)\d s}]&=\frac{\cos( a\sqrt{2\nu})}{\cos(b\sqrt{2\nu})},
\;\;\;\displaystyle  0\geq a\geq b>-\frac{\pi}{2\sqrt{2\nu}},\label{formula:ub}
\end{align}
where $J_0(\cdot)$, satisfying $J_0(0)=1$, is the Bessel function of the first kind of order zero, and $j_{0,1}$ is the smallest positive zero of $J_0(\cdot)$.
\medskip  

\noindent {\rm (2$\cc$)} Fix $M\in(0,\infty)$. As $\nu\to 0+$, the following expansion holds uniformly over $0<a,b\leq M$, using the negative part $\log^-(x)\,\defeq -\log (x\wedge 1)$ of $\log x$:
\begin{align}
\begin{split}
&R^0_{\nu,b}(a)=1-\frac{\log^- (b/a)}{\log \frac{1}{\sqrt{2\nu}}}\\
&\quad\quad+\mathcal O_M\biggl((a\vee  b)^2\cdot 2\nu \log \frac{1}{(a\vee  b)\sqrt{2\nu}}+
\frac{(|\log b|\1_{a\geq b}+1)(\log^-(b/a)+1)}{\log^2 \sqrt{2\nu}}\biggr).\label{ET:exp}
\end{split}
\end{align}

\noindent {\rm (3$\cc$)} For all  $a,b>0$, 
\begin{align}\label{R:uppbdd}
R_{q\vep^2,b}^{\lv\oovarphi\uvep}(a)\leq \E_{\log a}^{\beta}\Big[\e^{\|(\lv \oovarphi\uvep)^+\|_\infty M_\varphi^2\int_0^{T_{\log b}(\beta)}\1_{(-\infty,\log M_\varphi]}(\beta_s)\d s}\Big].
\end{align}

\noindent {\rm (4$\cc$)} Fix $q\in (0,\infty)$. Then uniformly in $(a,b)\in \Gamma_{\geq}(\vep,M_\varphi)\cup \Gamma_{\leq}(M_\varphi)$ as $\vep \to 0$, 
\begin{align}
\begin{split}\label{Qf:exp0}
R_{q\vep^2,b}^{\lv\oovarphi\uvep}(a)=1+\mathcal O_{\varphi,\chi,q,M_\varphi}\left(\frac{1}{\log^{1/2}\vep^{-1}}\right)\quad (\chi\,\defeq\,(\mu,\lambda)).
\end{split}
\end{align}
\end{lem}

\begin{rmk}\label{rmk:nu<0}
For the extension of \eqref{formula:ub} to $\nu<0$, see \cite{JY} and \cite[Example 2.1, p. 437 with $\mu=0$]{Knight}. \qed
\end{rmk}

\begin{proof}[Proof of Lemma~\ref{lem:3asymp}]
(1$\cc$) To derive \eqref{BES:time-ext}, consider the case $0 < a \leq b$ in \eqref{BES:time}. By comparing the series expansions of $J_0(\cdot)$ and $I_0(\cdot)$~\cite[(5.3.2) p.102; (5.7.1) p.108]{Lebedev}, we find $J_0(z) = I_0(\i z)$. By analytically continuing the second formula in \eqref{BES:time} with respect to the variable $\nu$, we obtain \eqref{BES:time-ext}. In more detail, this step is justified by a general theorem on analytic characteristic functions \cite[Theorem~2.3.5, p.33]{Bryc:Normal}. 

The validity of \eqref{formula:ub} can also be established via analytic continuation, extending the formula for $\nu<0$ found in the references cited in Remark~\ref{rmk:nu<0}. Alternatively, the result can be confirmed using stochastic calculus, as in \cite{JY}, which we outline below for completeness.

For fixed $\nu>0$ and $-\pi/(2\sqrt{2\nu}) < b < 0$, observe that the following problem has a unique continuous solution given by $h_b(a)\,\defeq \cos((a\wedge 0)\sqrt{2\nu})/\cos(b\sqrt{2\nu})$ for $b \leq a < \infty$:
\begin{align}\label{ODE:FK}
\begin{cases}
\displaystyle \frac{1}{2}u''(a)+\nu\1_{(-\infty,0]}(a) u(a)=0,\quad a\in (b,0)\cup (0,\infty),\\
\displaystyle \lim_{a\nearrow 0}u'(a)=\lim_{a\searrow 0}u'(a),\quad 
u(b)=1.
\end{cases}
\end{align}
The proof of \eqref{formula:ub} then proceeds through the following steps:
\begin{itemize}
\item [(i)] Expand $H^\nu_t\, \defeq \,h_b(\beta_t) \exp\{\nu \int_0^t \1_{(-\infty,0]}(\beta_s)\, \mathrm{d}s\}$ under $\P^\beta_a$ into a semimartingale by first applying the It\^{o}--Tanaka formula \cite[6.24, p.215]{KS:BM} to $h_b(\beta_t)$, followed by the product rule for stochastic integrals. In particular, the first and second equalities in \eqref{ODE:FK} ensure that the finite-variation part of $H^\nu_{t\wedge T_b(\beta)}$ vanishes. 
\item [(ii)] Take expectations of $H^\nu_{t \wedge T_b(\beta) \wedge T_n(\beta)}$, which are martingales by (i) for all fixed $n\in \Bbb N$. This yields $\E^\beta_a[H^\nu_{t \wedge T_b(\beta) \wedge T_n(\beta)}]=H^\nu_0$. 
\item [(iii)] By applying Fatou's lemma to $\E^{\beta}_a[H^{\nu_0}_{t \wedge T_b(\beta) \wedge T_n(\beta)}] = H^{\nu_0}_0$ for $\nu_0 > \nu$ as $n \to \infty$, we obtain the inequality $\E^{\beta}_a\big[\exp\{\nu_0 \int_0^{T_b(\beta)} \1_{(-\infty,0]}(\beta_s)\, \mathrm{d}s\}\big] \leq H_0^{\nu_0}$, since $h_b(a) \geq 1$ for all $b \leq a < \infty$.  Consequently, the family $\{H^\nu_{t \wedge T_b(\beta) \wedge T_n(\beta)}\}_{t \geq 0, n \in \mathbb{N}}$ is uniformly integrable. This property enables us to take the limits $t \to \infty$ and $n \to \infty$ inside the expectation $\E^\beta_a[H^\nu_{t \wedge T_b(\beta) \wedge T_n(\beta)}] = H^{\nu_0}_0$ using a standard limit theorem. This establishes identity \eqref{formula:ub} through the final equality in \eqref{ODE:FK}. 
\end{itemize}

\noindent (2$\cc$) This proof derives \eqref{ET:exp} from \eqref{BES:time}, using the following general rule:
 \begin{align}\label{tau:rule1}
\frac{1+\vep_a}{1+\vep_b}=1+(\vep_a-\vep_b)-\frac{\vep_b(\vep_a-\vep_b)}{(1+\vep_b)}\quad\mbox{whenever $\vep_b>-1$.}
\end{align}

First, by using \eqref{ET:exp} when $M\geq a\geq b>0$, 
\[
\vep_{\alpha}\equiv \frac{K_0(\alpha\sqrt{2\nu})}{\log \frac{1}{\sqrt{2\nu}}}-1,\quad \alpha\in \{a,b\}\Longrightarrow\;
R^0_{\nu,b}(a)=\frac{K_0(a\sqrt{2\nu})}{K_0(b\sqrt{2\nu})}=\frac{1+\vep_a}{1+\vep_b}.
\] 
Whenever $\nu<1/2$, $\vep_b>-1$ so that the rule in \eqref{tau:rule1} applies. To justify the $\mathcal{O}_M$-term in \eqref{ET:exp} when $M \geq a \geq b > 0$, we use the following estimate,  obtained from \eqref{asymp:K0}: 
\begin{align*}
\vep_\alpha
&=\frac{\log \frac{2}{\alpha}-\EM+\mathcal O_M\big(\alpha^2\cdot 2\nu \log \frac{1}{\alpha\sqrt{2\nu}}\big)}{\log\frac{1}{ \sqrt{2\nu}}},
\end{align*}
where the $\mathcal O_M$-term holds uniformly in $M \geq \alpha > 0$ as $\nu \to 0+$.
Since $x^2\log(1/x)$ is increasing for $0 < x \leq \e^{-1/2}$, the first term in the argument of $\mathcal{O}_M(\cdot)$ in \eqref{ET:exp} bounds the following term, which is a remainder in $\vep_a - \vep_b$:
\[
\frac{\mathcal O_M\big(a^2\cdot 2\nu \log \frac{1}{a\sqrt{2\nu}}\big)}{\log\frac{1}{ \sqrt{2\nu}}}-\frac{\mathcal O_M\big(b^2\cdot 2\nu \log \frac{1}{b\sqrt{2\nu}}\big)}{\log\frac{1}{ \sqrt{2\nu}}}.
\]
Additionally, the second term in the argument of $\mathcal O_M(\cdot)$ of \eqref{ET:exp} bounds $\vep_b(\vep_a-\vep_b)$. We have proved \eqref{ET:exp} when $M\geq a\geq b>0$.

For $0 < a \leq b \leq M$, we apply the corresponding formula from \eqref{BES:time}.
Recall the asymptotic expansion $I_0(a) = 1 + \mathcal{O}(a^2)$ as $a \to 0$, and note that $I_0(a)$ is a strictly positive, increasing function for $a \geq 0$. Both properties follow from the integral representation of $I_0(\cdot)$ in~\cite[(5.10.22), p.119]{Lebedev}.  Hence, the following inequalities and estimate hold, where the $\mathcal O_M$-term holds uniformly in $0< b\leq M$:
\begin{align}\label{I0ratio}
1\geq \frac{I_0(a\sqrt{2\nu})}{I_0(b\sqrt{2\nu})}\geq 
\frac{1}{I_0(b\sqrt{2\nu})}=\frac{1}{1+\mathcal O_M(b^2\nu)},\quad \nu\to 0+.
\end{align}
The expansion in \eqref{ET:exp} for $M \geq a \geq b > 0$ follows by applying \eqref{tau:rule1} to the rightmost side of \eqref{I0ratio} and observing that $b^2\nu$ is bounded by the first term in the argument of $\mathcal{O}_M(\cdot)$ in \eqref{ET:exp}. This completes the proof of (2$\cc$).\medskip 

\noindent (3$\cc$) Given $\rho_0 = a > 0$, the SDE for the two-dimensional Bessel process provides the representation $\rho_t = \exp\{\beta_{\int_0^t \d r / \rho_r^2}\}$, where $\{\beta_t\}$ is a one-dimensional Brownian motion starting at $\log a$; see \cite[(2.11) Theorem, p.193]{RY} for more details. Thus, \eqref{R:uppbdd} follows by considering the rightmost side of \eqref{eq:4-0} along with the bound established below:
\begin{align*}
A_{\|(\lv \oovarphi\uvep)^+\|_\infty\1_{[0,M_\varphi]}(|\cdot|)}(T_b(\rho))&=\|(\lv \oovarphi\uvep)^+\|_\infty \int_0^{T_b(\rho)} \1_{[0,M_\varphi]}(\exp\{\beta_{\int_0^s \d r/\rho_r^2}\})\d s\\
&=\|(\lv \oovarphi\uvep)^+\|_\infty\int_0^{T_{\log b}(\beta)}\1_{[0,M_\varphi]}(\exp\{\beta_{s'}\})\e^{2\beta_{s'}}\d s' \\
&\leq \|(\lv \oovarphi\uvep)^+\|_\infty M_\varphi^2\int_0^{T_{\log b}(\beta)}\1_{(-\infty,\log M_\varphi]}(\beta_{s'})\d s'.
\end{align*}
Here, in the second equality, we change variables by setting $s' = \int_0^s \d r / \rho_r^2$ and use the identity $\int_0^{T_b(\rho)} \d r / \rho_r^2 = T_{\log b}(\beta)$. 
\medskip 

\noindent (4$\cc$) The proof proceeds by proving that \eqref{Qf:exp0} holds when the equality is replaced by inequalities ``$\geq$'' and ``$\leq$'' , utilizing the bounds from \eqref{eq:4-0}:
\begin{itemize}
\item \emph{Equation~\eqref{Qf:exp0} with ``\,$=$'' replaced by ``\,$\geq $''}: The lower bound in \eqref{eq:4-0} is at least $R^0_{\nu, b}(a)$ for the following choice of $\nu$:
 $\nu=C(\varphi,\chi,q)/\log\vep^{-1}$. This provides  \eqref{Qf:exp0} with ``$=$'' replaced by ``$\geq$'', once we analyze the corresponding $R^0_{\nu,b}(a)$ using the expansion in \eqref{ET:exp}, the selection of $M_\vep$ in \eqref{def:Mvep} and the increasing monotonicity of $x^2\log(1/x)$ for $0 < x \leq \e^{-1/2}$.
\item \emph{Equation~\eqref{Qf:exp0} with ``\,$=$'' replaced by ``\,$\leq$''}: We begin by deriving two additional upper bounds that admit explicit formulas:
\begin{itemize}
\item For $(a,b)\in \Gamma_{\leq}(M_\varphi)$, the upper bound in \eqref{eq:4-0} satisfies: 
\[
R_{0,b}^{\|(\lv \oovarphi\uvep)^+\|_\infty\1_{[0,M_\varphi]}(|\cdot|)}(a)\leq R_{-\|(\lv \oovarphi\uvep)^+\|_\infty,b}^{0}(a).
\]
The right-hand side admits an explicit formula via \eqref{BES:time-ext}, using the choice $\nu\equiv \|(\lv \oovarphi\uvep)^+\|_\infty$. 
\item For $(a, b) \in \Gamma_{\geq}(\vep, M_\varphi)$, we use the upper bound in \eqref{R:uppbdd}. Specifically, observe the following estimate:
 \[
 |\log M_\vep-\log M_\varphi|\cdot 2 \sqrt{2\|(\lv \oovarphi\uvep)^+\|_\infty M_\varphi^2}\xrightarrow[\vep\to 0]{} 0.
 \]
Hence, the right-hand side of \eqref{R:uppbdd} can be evaluated using \eqref{formula:ub}
with $(a,b,\nu)$ replaced by $(\log a - \log M_\varphi, \log b - \log M_\varphi,\|(\lv \oovarphi\uvep)^+\|_\infty M_\varphi^2)$. 
\end{itemize}

Now, we derive \eqref{Qf:exp0} with ``$=$'' replaced by ``$\leq$'' by employing the explicit upper bounds obtained above and four key facts: (i) the second-order Taylor expansions of $J_0(z)$ and $\cos(z)$ as $z \to 0$; (ii) the general rule in \eqref{tau:rule1}; (iii) the definition of $M_\vep$ in \eqref{def:Mvep}; and (iv) the bound $\|(\lv \oovarphi\uvep)^\pm\|_\infty \leq C(\varphi, \chi)/\log \vep^{-1}$. Specifically, the second-order Taylor expansions as $z \to 0$ are $J_0(z) = 1 + J_0''(0)z^2/2! + \mathcal{O}(|z|^4)$ (see the proof of (1$\cc$)), and $\cos(z) = 1 - z^2/2! + \mathcal{O}(|z|^4)$.
\end{itemize}
We have proved \eqref{Qf:exp0}. 
\end{proof}

\begin{proof}[Proof of (\ref{eq:finalI:rad}) of Proposition~\ref{prop:5}]
The proof begins with the following asymptotic expansion, to be established in Step~1: uniformly in $b\in [M_\vep, M_\varphi]$ and $|z|\leq M_\varphi$ as $\vep\to 0$,  
\begin{align}
V_{q\vep^2}\{\lv \oovarphi\uvep R^{\lv \oovarphi\uvep}_{q\vep^2,b}\}(|z|)&=\mu X_\vep(b)+\frac{A^\sharp_{0}(z)}{\log\vep^{-1}}
+\little{o}_{\varphi,\chi,q,M_\varphi}\left(\frac{1}{\log \vep^{-1}}\right),\label{eq:X0}
\end{align}
where
\begin{align}
X_\vep(b)&\,\defeq \int \oophi\uvep (z') R^{\lv \oovarphi\uvep}_{q\vep^2,b}(|z'|)\d z',\\
A^\sharp_{0}(z)&\,\defeq\, \int \biggl(\mu\log \frac{2^{1/2}}{q^{1/2}|z-z''|}-\mu \gamma_{\sf EM}+\lambda\biggr)\overline{\varphi}(z'')\d z''.\label{def:A1}
\end{align}
Note that $A^\sharp_{0}(\cdot)$ is a radially symmetric function. In the remainder of this proof, $\little{o}(\cdot)=\little{o}_{\varphi,\chi,q,M_\varphi}(\cdot)$ and $\mathcal O(\cdot)=\mathcal O_{\varphi,\chi,q,M_\varphi}(\cdot)$

The remainder of the proof elaborates Step~1 and expands the right-hand side of \eqref{eq:X0} further, ultimately establishing \eqref{eq:finalI:rad}. 
This expansion proceeds
with the following steps: Step~2 shows that for all small $\vep$ and all $b \in [M_\vep, M_\varphi]$, 
\begin{align}
\begin{split}\label{Xvep:0}
X_\vep(b)&=\int \overline{\varphi}(z')R^0_{q\vep^2,b}(|z'|)\d z'\times [1-V_{q\vep^2}\{\lv \oovarphi\uvep  R^{\lv \oovarphi\uvep}_{q\vep^2,b}\}(b)]\\
&\quad +\int \overline{\varphi}(z')V_{q\vep^2}\{\lv \oovarphi\uvep  R^{\lv \oovarphi\uvep}_{q\vep^2,b}\}(|z'|)\d z'.
\end{split}
\end{align} 
Steps~3 and 4 then derive asymptotic expansions, uniformly in $b \in [M_\vep, M_\varphi]$ as $\vep \to 0$, for the following two terms from the right-hand side of \eqref{Xvep:0}:
\begin{gather}
\int \overline{\varphi}(z')R^0_{q\vep^2,b}(|z'|)\d z',\quad
\int \overline{\varphi}(z')V_{q\vep^2}\{\lv \oovarphi\uvep  R^{\lv \oovarphi\uvep}_{q\vep^2,b}\}(|z'|)\d z'.\label{Xvep:12term}
\end{gather}
In Step 5, we use \eqref{Xvep:0} to combine \eqref{eq:X0} and the asymptotic expansions of the two terms in \eqref{Xvep:12term}, yielding an asymptotic expansion for $\mu X_\vep(b)$ that holds uniformly in $b \in [M_\vep, M_\varphi]$ as $\vep \to 0$. In Step~6, we substitute this expansion of $\mu X_\vep(b)$ into \eqref{eq:X0}. After some algebraic manipulation, we obtain \eqref{eq:finalI:rad} since $b\in [M_\vep,M_\varphi]$ and $|z|\leq M_\varphi$ if and only if $(|z|,b)\in \Gamma_{\geq }(\vep,M_\varphi)\cup \Gamma_\leq (M_\varphi)$. \smallskip

\begin{rmk}[Asymptotic recursions]\label{rmk:AR}
In Step~4, the asymptotic expansion for the second term in \eqref{Xvep:12term} is proved using \eqref{eq:X0} again. Thus, this proof reflects the style of asymptotic recursions discussed in Section~\ref{sec:=0}. \qed 
\end{rmk}

\noindent {\bf Step~1.} To prove \eqref{eq:X0}, we begin by presenting a preliminary result, expressed through the following asymptotic expansion: uniformly in $b\in [M_\vep,M_\varphi]$ and $z\in \Bbb C$ as $\vep\to 0$,
\begin{align}
&V_{q\vep^2}\{\oophi\uvep R^{\lv \oovarphi\uvep}_{q\vep^2,b}\}(|z|)\notag\\
&\quad =\frac{\log \vep^{-1}}{\pi}\int \oophi\uvep(z')R^{\lv \oovarphi\uvep}_{q\vep^2,b}(|z''|)\d z''\notag\\
&\quad\quad +
\int \oophi\uvep(z'')R^{\lv \oovarphi\uvep}_{q\vep^2,b}(|z''|)\frac{1}{\pi}\biggl(\log\frac{2^{1/2}}{q^{1/2} |z-z''|} -\gamma_{\sf EM}\biggr)\d z'' \notag\\
&\quad\quad +\int \oophi\uvep(z'')R^{\lv \oovarphi\uvep}_{q\vep^2,b}(|z''|)\mathcal O\left(q\vep^2|z-z''|^2\log \frac{1}{q^{1/2}\vep |z-z''|}\right)\d z''.\label{Xvep:10}
\end{align}
To see this, recall that $\oophi\uvep(\cdot)$ and $R^{\lv \oovarphi\uvep}_{q\vep^2,b}(\cdot)$ are radially symmetric, $\oophi\uvep(\cdot)$ has a compact support,
and $G_\nu\{\cdot\}$ is introduced in \eqref{def:Gf}. Since $V_\nu\{h\}(|z|)=G_\nu\{h\}(z)$ for radially symmetric $h$, \eqref{Xvep:10} follows from \eqref{asymp:G} by choosing $\nu\equiv q\vep^2$. 
Moreover, by Lemma~\ref{lem:3asymp} (4$\cc$), each term on the right-hand side of \eqref{Xvep:10} is well-defined whenever $\vep$ is small.

We prove \eqref{eq:X0} now. By the definition \eqref{def:lvo} of $\lv$,
\begin{align}
&V_{q\vep^2}\{\lv \oovarphi\uvep R^{\lv \oovarphi\uvep}_{q\vep^2,b}\}(|z|)\notag\\
&\quad =\frac{\mu\pi}{\log\vep^{-1}}V_{q\vep^2}\{\oophi\uvep R^{\lv \oovarphi\uvep}_{q\vep^2,b}\}(|z|)+\frac{\pi\lambda}{\log^2\vep^{-1}}V_{q\vep^2}\{\oophi\uvep R^{\lv \oovarphi\uvep}_{q\vep^2,b}\}(|z|) \notag\\
&\quad =\mu X_\vep(b)+\frac{\mu}{\log \vep^{-1}}\int \oophi\uvep (z'')R^{\lv \oovarphi\uvep}_{q\vep^2,b}(|z''|)\biggl(\log\frac{2^{1/2}}{q^{1/2} |z-z''|} -\gamma_{\sf EM}\biggr)\d z''\notag\\
&\quad \quad +\frac{\lambda}{\log \vep^{-1}}\int \oophi\uvep (z'')R^{\lv \oovarphi\uvep}_{q\vep^2,b}(|z''|)\d z''+\mc O\left(\frac{1}{\log^2 \vep^{-1}}\right),\label{Xvep:1-1}
\end{align}
where \eqref{Xvep:1-1} uses \eqref{Xvep:10}, and
 the $\mc O$-term in \eqref{Xvep:1-1} follows from Lemma~\ref{lem:3asymp} (4$\cc$).
Applying \eqref{Qf:exp0} to the right-hand side of \eqref{Xvep:1-1} allows a replacement of $\oophi\uvep(z'')R^{\lv \oovarphi\uvep}_{q\vep^2,b}(|z''|)$ by $\overline{\varphi}(z'')$, yielding the following asymptotic expansion that proves \eqref{eq:X0}:
\begin{align}
V_{q\vep^2}\{\lv \oovarphi\uvep R^{\lv \oovarphi\uvep}_{q\vep^2,b}\}(|z|)
&=\mu X_\vep(b)+\frac{\mu}{\log \vep^{-1}}\int \overline{\varphi}(z'')\biggl(\log\frac{2^{1/2}}{q^{1/2} |z-z''|} -\gamma_{\sf EM}\biggr)\d z''\notag\\
&\quad +\frac{\lambda}{\log \vep^{-1}}\int \overline{\varphi} (z'')\d z''+\little{o}\left(\frac{1}{\log \vep^{-1}}\right).
\end{align}

\noindent {\bf Step~2.}
To justify \eqref{Xvep:0} for all sufficiently small $\vep$ and all $b \in [M_\vep, M_\varphi]$, note that the condition in \eqref{psiab:decr} is satisfied by applying \eqref{Qf:exp0} for these values of $\vep$ and $b$. As a result, \eqref{Xvep:0} follows by integrating both sides of the equation in \eqref{psiab:decr} against $ \oophi\uvep(z')\d z'$, provided that $a$ in \eqref{psiab:decr} is replaced by $|z'|$. \medskip 

\noindent {\bf Step~3.}
By \eqref{def:phiexp} and \eqref{ET:exp}, the first term in \eqref{Xvep:12term} satisfies the following asymptotic expansion: uniformly in $b\in [M_\vep,M_\varphi]$ as $\vep\to 0$,
\begin{align}
 \int\overline{\varphi}(z')R^0_{q\vep^2,b}(|z'|)\d z'
&= \int \left[\overline{\varphi}(z')+\frac{\mu\pi \widecheck{\varphi}(z')}{\log \vep^{-1}}\right]\bigg[1-\frac{\log^-(b/|z'|)}{\log \frac{1}{\sqrt{2q\vep^2}}}
\bigg]\d z'\notag\\
&\quad +\mc O\left(\frac{\log^2b+1}{\log^2 \vep^{-1}}\right)\notag\\
& =1+\frac{A^\sharp_{1}(b)}{\log \vep^{-1}}+\little{o}\left(\frac{1}{\log \vep^{-1}}\right),\label{Xvep:1}
\end{align}
where the second equality uses the identity $\int\ol{\varphi}(z')\d z'=1$ and the notation
\begin{align}\label{def:A0}
A^\sharp_{1}(b)\,\defeq \int\left[\mu\pi \widecheck{\varphi}(z')-\overline{\varphi}(z')\log^-(b/|z'|)\right]\d z',
\end{align}
as well as the selection of $M_\vep$ in \eqref{def:Mvep}.\medskip

\noindent {\bf Step 4.} 
We show that the second term in \eqref{Xvep:12term} satisfies 
the following asymptotic expansion: uniformly in $b\in [M_\vep,M_\varphi]$ and $|z|\leq M_\varphi$ as $\vep\to 0$,
\begin{align}
&\int \oovarphi\uvep(z') V_{q\vep^2}\{\lv \oovarphi\uvep  R^{\lv \oovarphi\uvep}_{q\vep^2,b}\}(|z'|)\d z'\notag\\
&\quad =\biggl(1+\frac{A^\sharp_{2}}{\log \vep^{-1}}\biggr)
\mu X_\vep(b)+\frac{A^\sharp_{3}}{\log \vep^{-1}}+\little{o}\left(\frac{1}{\log \vep^{-1}}\right).\label{Xvep:2}
\end{align}
Here, $A^\sharp_{2}$ and $A^\sharp_{3}$ are constants defined below; we also present alternative forms for them using the definitions of $\widecheck{\varphi}(\cdot)$ and $A^\sharp_{0}(\cdot)$ in \eqref{def:cphi} and \eqref{def:A1}:
\begin{align}
A^\sharp_{2}&\,\defeq\, \int \mu\pi\widecheck{\varphi}(z')\d z'\notag\\
&=\int_0^\infty
\mu 2\pi^2 r\cdot \frac{1}{2\pi}\int_{-\pi}^\pi \left(-\frac{1}{\pi}\widehat{\varphi}(r\e^{\i \theta})\int_{\Bbb C}\widehat{\varphi}(z'')\log |r\e^{\i \theta}-z''|\d z''\right)\d \theta \d r
 \notag\\
 &=\mu\dint \widehat{\varphi}(z')\widehat{\varphi}(z'')\log \frac{1}{|z'-z''|}\d z'\d z'',\label{def:A2}\\
A^\sharp_{3}&\,\defeq\,\int \overline{\varphi}(z')A^\sharp_{0}(z')\d z'\notag\\
&=\dint \overline{\varphi}(z')\overline{\varphi}(z'')\biggl(\mu\log \frac{2^{1/2}}{q^{1/2}|z'-z''|}-\mu\EM+\lambda\biggr)\d z'\d z''.\label{def:A3}
\end{align}

To derive \eqref{Xvep:2}, we apply \eqref{def:phiexp}, \eqref{def:lvo}, and \eqref{eq:X0}, which yields:
\begin{align}
&\int \overline{\varphi}(z')V_{q\vep^2}\{\lv \oovarphi\uvep  R^{\lv \oovarphi\uvep}_{q\vep^2,b}\}(|z'|)\d z'\notag\\
&\quad =\int \left[\overline{\varphi}(z')+\frac{\mu\pi \widecheck{\varphi}(z')}{\log \vep^{-1}}\right]\biggl[\mu X_\vep(b)+\frac{A^\sharp_{0}(z')}{\log\vep^{-1}}\biggr]\d z' +\little{o}\left(\frac{1}{\log \vep^{-1}}\right)\notag.
\end{align}
Equation~\eqref{Xvep:2} now follows from the previous equality after some algebra.\medskip 

\noindent {\bf Step 5.} In this step, we derive an asymptotic expansion for $\mu X_\vep(b)$ by applying the expansions in \eqref{eq:X0}, \eqref{Xvep:1}, and \eqref{Xvep:2} to \eqref{Xvep:0}. The resulting expansion is given by the last equality below, which holds uniformly in $b \in [M_\vep, M_\varphi]$ as $\vep \to 0$:
\begin{align}
\mu X_\vep(b)
&=\mu\biggl[1+\frac{A^\sharp_{1}(b)}{\log \vep^{-1}}+\little{o}\left(\frac{1}{\log \vep^{-1}}\right)\biggr] \biggl[1-\mu X_\vep(b)-\frac{A^\sharp_{0}(b)}{\log \vep^{-1}}-\little{o}\left(\frac{1}{\log \vep^{-1}}\right)\biggr]\notag\\
&\quad +\mu\biggl[\biggl(1+\frac{A^\sharp_{2}}{\log \vep^{-1}}\biggr)\mu X_\vep(b)+\frac{A^\sharp_{3}}{\log \vep^{-1}}+\little{o}\biggl(\frac{1}{\log \vep^{-1}}\biggr)\biggr]\notag\\
& =\mu+\mu\biggl(\frac{A^\sharp_{1}(b)-A^\sharp_{0}(b)+A^\sharp_{3}}{\log\vep^{-1}}\biggr)+\mu\biggl(-1-\frac{A^\sharp_{1}(b)}{\log \vep^{-1}}+1+\frac{A^\sharp_{2}}{\log \vep^{-1}}\biggr)\mu X_\vep(b)\notag\\
&\quad +\little{o}\left(\frac{1}{\log \vep^{-1}}\right)\notag\\
\begin{split}
&=\mu+\frac{\mu(-\mu+1)A^\sharp_{1}(b)-\mu A^\sharp_{0}(b)+\mu^2 A^\sharp_{2} +\mu A^\sharp_{3}}{\log \vep^{-1}}+\little{o}\left(\frac{1}{\log \vep^{-1}}\right).
\label{X:eq1}
\end{split}
\end{align}
In more detail, $X_\vep(b)=1+\mathcal O(\log^{-1/2} \vep^{-1})$ uniformly in $b\in [M_\vep,M_\varphi]$ as $\vep\to 0$, justified by \eqref{def:phiexp} and \eqref{Qf:exp0}, is applied in the last two equalities. Additionally, the selection of $M_\vep$ in \eqref{def:Mvep} has been applied repeatedly to obtain the terms $\little{o}((\log \vep^{-1})^{-1})$ in the last two equalities. \medskip 

\noindent {\bf Step 6.}
We complete the proof of \eqref{eq:finalI:rad} in this step. First, applying \eqref{eq:X0} and \eqref{X:eq1}, we obtain that uniformly in $b\in [-M_\vep,M_\vep]$ and $|z|\leq M_\varphi$ as $\vep\to 0$,
\begin{align}\label{V:final}
V_{q\vep^2}\{\lv \oovarphi\uvep R^{\lv \oovarphi\uvep}_{q\vep^2,b}\}(|z|)
& =
\mu+\frac{A_{\mu,\lambda,q}^\sharp(z,b)}{\log\vep^{-1}}+\little{o}\left(\frac{1}{\log \vep^{-1}}\right),
\end{align}
where
\begin{align}
\begin{split}\label{def:Aab}
A_{\mu,\lambda,q}^\sharp(z,b)&\,\defeq\, \mu(-\mu+1)A^\sharp_{1}(b)+A^\sharp_{0}(z)-\mu A^\sharp_{0}(b)+\mu^2 A^\sharp_{2} +\mu A^\sharp_{3}.
\end{split}
\end{align}

To complete the derivation of \eqref{eq:finalI:rad}, it remains to show the following identity:
\[
A_{\mu,\lambda,q}^\sharp(z,b) = A_{\mu,\lambda,q}(|z|,b), 
\]
where $A_{\mu,\lambda,q}(a,b)$ is defined in \eqref{def:Azb}. This follows by combining the identities below:
\begin{align}
A_{\mu,\lambda,q}^\sharp(z,b)
&=(-\mu^2+\mu)[A^\sharp_{1}(b)-A^\sharp_{2}]+[A^\sharp_{0}(z)-\mu A^\sharp_{0}(b)]+\mu (A^\sharp_{2} + A^\sharp_{3}),\notag\\
(-\mu^2+\mu)[A^\sharp_{1}(b)-A^\sharp_{2}]
&=
(\mu^2-\mu)
\int \log^-(b/|z'|)  \overline{\varphi}(z')\d z',\label{Vfinal:1}\\
\begin{split}
A^\sharp_{0}(z)-\mu A^\sharp_{0}(b)
&=\int \log \biggl(\frac{|b-z''|^{\mu^2}}{|z-z''|^\mu}\biggr)
\overline{\varphi}(z'')\d z''\\
&\quad\;+(\mu-\mu^2)\biggl(\log \frac{2^{1/2}}{q^{1/2}} - \gamma_{\sf EM}\biggr)
+(1-\mu)\lambda,\label{Vfinal:2}
\end{split}\\
\begin{split}
\mu (A^\sharp_{2}+A^\sharp_{3})
&=-\mu^2\dint \varphi(z')\varphi(z'')\log |z'-z''|\d z'\d z''\\
&\quad\;+\mu^2\biggl(\log \frac{2^{1/2}}{q^{1/2}}-\EM\biggr)+\mu\lambda.\label{Vfinal:3}
\end{split}
\end{align}
Here, \eqref{Vfinal:1} is a consequence of \eqref{def:A0} and the first equality in \eqref{def:A2}; \eqref{Vfinal:2} is justified by using the definition \eqref{def:A1} of $A^\sharp_{0}(\cdot)$; and \eqref{Vfinal:3} follows from \eqref{def:A2} and \eqref{def:A3}, since $\dint \overline{\varphi}(z')\widehat{\varphi}(z'')\log |z'-z''|\d z'\d z'' = 0$ by virtue of the identity $\int h_1(|z''|) \widehat{h}_2(z'') \d z''= 0$ for all $h_1, h_2$, where $\widehat{h}_2$ is as defined in \eqref{def:hphi}. The proof is complete.
\end{proof}

\begin{proof}[Proof of (\ref{Qf:exp0final})  of Proposition~\ref{prop:5}]
Equation~\eqref{Qf:exp0final} is obtained by applying \eqref{def:Mvep}, \eqref{eq:finalI:rad} and \eqref{ET:exp} to \eqref{psiab:decr}, using also the algebra that follows: we have
$(1-A)(1-\mu-B)+(\mu+C)=1-[(1-\mu)A+B-C]+AB$.
\end{proof}

\subsection{Application to the Green's functions at criticality}\label{sec:radial}
In this section, we prove Theorem~\ref{thm:critres}, using Theorem~\ref{thm:>0} (2$\cc$) as the basic tool. Our starting point is the following expansions:
\begin{align}
\E^W_z[\e^{ A_f(t)}g(W_t)]&=\E^W_z[g(W_t)]+\E^W_z\left[\int_0^tf(W_s)\e^{A_f(t-s)}\d s g(W_t)  \d r\right]\notag\\
&=\E^W_z[g(W_t)]+\E^W_z\left[\int_0^tf(W_s)\int_0^s \e^{ A_f(r)}f(W_r) g(W_t)  \d r\right].\label{expansion:2}
\end{align}
Take $f\equiv \lv \varphi_\vep(z)$. Then, by an application of the Markov property and a standard property of the Laplace transform, we expect that for any $z'$ of order $1$, 
\begin{align}
G^{\lv f}_q\{g\}(z)&\approx G_q\{g\}(z)+G_q(z)\lv G_{q}^{\lv(\oovarphi\uvep)_\vep}\{g\} \label{expansion:3}\\
&\approx G_q\{g\}(z)+G_q(z)\lv^2 G_{q}^{\lv(\oovarphi\uvep)_\vep}\{(\oovarphi\uvep)_\vep\} (\vep z') G_q\{g\}(0).\label{expansion:4}
\end{align}
These informal approximations, previously used and made rigorous in \cite[Section~5]{C:DBG}, remain central to the proof that follows.\medskip 

\begin{proof}[Proof of Theorem~\ref{thm:critres}] We begin by establishing the following a priori bound, which, as in the proof of \cite[Proposition~5.2]{C:DBG}, suffices to justify \eqref{expansion:4} as an approximation uniform in $|z'|\leq M$ for any $M>0$:
\begin{align}\label{final:q}
\sup_{z:|z|\leq M}\lv^2 G_{q}^{\lv(\oovarphi\uvep)_\vep}\{(\oovarphi\uvep)_\vep\} (\vep z)<\infty\mbox{\;\; $\forall$ large }q>\beta.
\end{align}
To establish \eqref{final:q}, we extend the argument presented in \cite[Lemmas~5.4 and~5.5, especially (5.38)]{C:DBG}. In particular, the technique used in \cite[(5.38)]{C:DBG} remains applicable to \eqref{final:q} because: (1) $\inf_z\oovarphi\uvep (z)\geq 0$ for all small $\vep > 0$, given that $\varphi \geq 0$ and $|\widecheck{\varphi}| \leq C(\varphi) \overline{\varphi}$ by the definition of $\widecheck{\varphi}$ in \eqref{def:hphi}; and (2) $\oovarphi\uvep$
 is composed of $\overline{\varphi}$, which is nonnegative and satisfies $\int \overline{\varphi} = 1$, and a function uniformly bounded by $C(\lambda,\varphi)(\log \vep^{-1})^{-1}$.

It remains to show that the Green's functions in \eqref{final:q} converge uniformly in $|z|\leq M$ to $2\pi/[\log (q/\beta)]$. To simplify these functions, observe that 
\begin{gather}
\mbox{$(\int_0^t\vep^{-2}f(\vep^{-1}W_s)\d s,\vep^{-1}W_t)$ under $\P_{\vep z}^W$}\stackrel{\rm (d)}{=}
\mbox{$(\int_0^{\vep^{-2}t}f ( W_{s})\d s, W_{\vep^{-2} t})$ under $\P^W_{z}$}\notag\\
\Longrightarrow \lv^2 G_{q}^{\lv(\oovarphi\uvep)_\vep}\{(\oovarphi\uvep)_\vep\} (\vep z)=\lv^2 G_{q\vep^2}^{\lv\oovarphi\uvep}\{\oovarphi\uvep\} (z).\label{reduce:1}
\end{gather}
Thus, the proof reduces to showing that $\lv^2 G_{q\vep^2}^{\lv\oovarphi\uvep}\{\oovarphi\uvep\}(z)$ converge uniformly in $|z| \leq M$ to $2\pi/[\log (q/\beta)]$. This is demonstrated in Steps 1 and 2 below. \medskip  

\noindent {\bf Step~1.}
Pick radially symmetric functions $\upsilon_\vep$ such that $0\leq \upsilon_\vep\leq 1$, $\upsilon_\vep(z)=0$ for all $|z|\leq  M_\vep$, and $\upsilon_\vep(z)=1$ for all $|z|\geq 2M_\vep$. In this step, we show that
\begin{align}\label{reduce:2}
\sup_{z:|z|\leq M}\left|\lv^2 G_{q\vep^2}^{\lv\oovarphi\uvep}\{\oovarphi\uvep(1-\upsilon_\vep)\} (z)\right|\xrightarrow[\vep\to 0]{}0.
\end{align}

To justify \eqref{reduce:2}, we use the following analogue of \eqref{exp:10} when the terminal condition is chosen to be $\oovarphi\uvep (1-\upsilon_\vep)$:
\[
\lv^2G_{q\vep^2}^{\lv\oovarphi\uvep}\{\oovarphi\uvep (1-\upsilon_\vep)\}=\lv^2 G_{q\vep^2}\{\oovarphi\uvep (1-\upsilon_\vep)\}+ G_{q\vep^2}\{\lv\oovarphi\uvep\cdot \lv^2 G_{q\vep^2}^{\lv\oovarphi\uvep}\{\oovarphi\uvep (1-\upsilon_\vep)\} \}.
\]
This allows Lemma~\ref{lem:Dyson} stated below, using the following choice of parameters:
\begin{align*}
H&\equiv \lv^2 G_{q\vep^2}^{\lv\oovarphi\uvep}\{\oovarphi\uvep(1-\upsilon_\vep)\},\\
 H_0&\equiv \lv^2G_{q\vep^2}\{\oovarphi\uvep (1-\upsilon_\vep)\},\\
 f&\equiv \lv \oovarphi\uvep,\\
 D&\equiv \{z\in \Bbb C;|z|\leq M\vee M_\varphi\},\\
 \nu&\equiv q\vep^2.
\end{align*}
Note that $\sup_{z \in D}\lv G_{q\vep^2}\{\oovarphi\uvep (1 - \upsilon_\vep)\}(z)$ tends to zero by using \eqref{asymp:G}, our choice of $\upsilon_\vep$, and the definition \eqref{def:Mvep} of $M_\vep$. Additionally, $fH_0=\lv^2 \oovarphi\uvep\cdot \lv G_{q\vep^2}\{\oovarphi\uvep (1-\upsilon_\vep)\}$. Thus, \eqref{reduce:2} follows from Lemma~\ref{lem:Dyson}, \eqref{final:q} and the necessary condition in \eqref{reduce:1}.
\medskip 

\noindent {\bf Step~2.}
In this step, we establish the following limit, which, together with \eqref{reduce:2}, completes the proof of Theorem~\ref{thm:critres}: for all $q$ from \eqref{final:q} with $q>\beta$,
 \begin{align}
\sup_{z:|z|\leq M}\left|\lv^2 G_{q\vep^2}^{\lv\oovarphi\uvep}\{\oovarphi\uvep \upsilon_\vep\} (z)-\frac{2\pi }{\log(q/\beta)}\right|\xrightarrow[\vep\to 0]{}0.\label{finalgoal}
\end{align}

To prove \eqref{finalgoal}, we use the formula below, which follows from \eqref{kernel:expexp} and \eqref{dec:Gf}:
\begin{align}
\lv^2 G_{q\vep^2}^{\lv\oovarphi\uvep}\{\oovarphi\uvep\upsilon_\vep\} (z)&=\int_0^\infty  \frac{\lv^2 R^{\lv\oovarphi\uvep}_{q\vep^2,b}(|z|)V_{q\vep^2}(b,b)}{\max\{1-V_{q\vep^2}\{\lv\oovarphi\uvep R^{\lv\oovarphi\uvep}_{q\vep^2,b}\}(b),0\}}\oovarphi\uvep(b)\upsilon_\vep(b)\d b.\label{eq2:main}
\end{align}
With $\beta$ defined by \eqref{lim:beta} and with $\mu=1$, the kernel on the right-hand side satisfies: 
\begin{align*}
 \frac{ \lv^2\cdot R^{{\lv\oovarphi\uvep}}_{q\vep^2,b}(|z|)\cdot V_{q\vep^2}(b,b)}{\max\{1-V_{q\vep^2}\{{\lv\oovarphi\uvep} R^{{\lv\oovarphi\uvep}}_{q\vep^2,b}\}(b),0\}}
  =  \frac{\frac{\pi^2+\little{o}_{\varphi,\chi,q,M_\varphi}(1)}{(\log\vep^{-1})^2}\cdot [1+\little{o}(1)]\cdot \log\vep^{-1}[2b+ b\cdot \little{o}(1)]}{\max\{[\frac{1}{2}\log (q/\beta)+\little{o}_{\varphi,\chi,q,M_\varphi}(1)]/\log\vep^{-1},0\}},
\end{align*}
where the right-hand side follows from \eqref{def:lvo},
\eqref{asymp:U}, \eqref{eq:finalI:rad}, and \eqref{Qf:exp0}. Moreover, the previous equality holds uniformly for $(|z|,b) \in \Gamma_{\geq}(\vep, M_\varphi) \cup \Gamma_{\leq}(M_\varphi)$ as $\vep \to 0$. Since $\log(q/\beta) > 0$, we obtain the following limit, holding uniformly in $0 < |z| \leq M$:
\begin{align}
&\lv^2 G_{q\vep^2}^{\lv\oovarphi\uvep}\{\oovarphi\uvep \upsilon_\vep\} (z) \xrightarrow[\vep\to 0]{}\int_0^\infty \frac{4\pi^2b}{\log (q/\beta)} \overline{\varphi}( b)\d b=\frac{2\pi}{\log(q/\beta)},\label{finalintegral}
\end{align}
where the equality comes from switching to the polar coordinates. This is enough to prove \eqref{finalgoal}.
\end{proof}

\begin{lem}\label{lem:Dyson}
Let $f,H,H_0$ be nonnegative functions defined on $\R^2$ with $\|H\|_f<\infty$ and $\|H_0\|_f<\infty$, with $\|\cdot\|_h$ defined in \eqref{def:Hh}. Suppose $D\supseteq \supp(f)$ and $\nu \in (0, \infty)$ so that $G^f_\nu\{f\} < \infty$ in $D$. Then the following property holds:
\begin{align}\label{H:ineq}
\begin{split}
H&\leq H_0+G_\nu \{fH\}\mbox{ in $D$}\Longrightarrow\mbox{$H\leq  H_0+G^f_\nu \{fH_0\}$ in $D$.}
\end{split}
\end{align}
\end{lem}

\begin{proof}
We iterate $H\leq H_0+G_\nu \{fH\}$. For all integer $n\geq 1$ and $1\leq j\leq n-1$,
applying \eqref{symmI1} with $h \in \{fH_0, fH\}$ and $k \in \{j, n-1\}$ shows the following inequality for any fixed $z_0\in D$:
\begin{align*}
H(z_0)
&\leq H_0(z_0)+\sum_{j=1}^{n-1}\int_0^\infty \frac{\e^{-\nu t}}{(j-1)!}\E^W_{z_0}\left[\left(\int_0^t f(W_{s})\d s\right)^{j-1} (fH_0)(W_t) \right]\d t\\
&\quad+\int_0^\infty \frac{\e^{-\nu t_{n}}}{(n-1)!}\E_{z_0}^W\left[\left(\int_0^{t_{n}} f(W_{s})\d s\right)^{n-1}(fH)(W_{t_{n}})\right]\d t_{n}.
\end{align*}
Since $\|H\|_f,\|H_0\|_f,G^f_\nu\{f\}$ are finite by assumption, passing $n\to\infty$ for the right-hand side of the last inequality leads to $H(z_0)\leq H_0(z_0)+G^f_\nu\{fH_0\}(z_0)$, as required. 
\end{proof}

\subsection{Alternative derivation by asymptotic recursions}\label{sec:another}
In this section, we provide an alternative proof of the following limit, 
which assumes \eqref{def:lvoointro} and, as demonstrated in \cite[Proposition~5.2]{C:DBG}, is crucial for deriving \eqref{finalgoal00}: for all $M\in (0,\infty)$ and $q \in (q_0,\infty)$ with sufficiently large $q_0 \in (0, \infty)$ independent of $\vep$,
\begin{align}\label{lim:finfty}
 \lim_{\vep\to 0}\lv f_\infty(z)&=\lim_{\vep\to 0}\frac{\pi}{\log \vep^{-1}} f_\infty(z)
=\frac{2\pi }{\log (q/\beta)}G_q\{g\}(0)\;\;\mbox{uniformly in }|z|\leq M,
\end{align}
where $f_\infty(z)\;\defeq \;G^{\llv\varphi_\vep}_q\{g\}(\vep z)$ for nonnegative $g\in \B_{b}$, and $\beta$ is defined in \eqref{lim:beta}. Recall Remark~\ref{rmk:betadiff} for a discussion of the notational differences between this paper and \cite{C:DBG}. The proof of \eqref{lim:finfty} presented below assumes $\varphi \geq 0$ and utilizes the method of asymptotic recursions. The argument we thus use parallels the one in Section~\ref{sec:ar-der}, but is now significantly simplified due to the assumption $\varphi\geq 0$. This allows us to avoid several subtleties. In particular, we do not need to verify an analogue of Assumption~\ref{ass:q}, since \cite[Lemmas~5.3 and 5.5]{C:DBG} ensure that there exists $q_1 \in (\beta, \infty)$, independent of $\vep$, such that $\sup_{z\in \R^2}f_\infty(z) < \infty$ for all $q\in ( q_1,\infty)$.

The main equation used in this proof of \eqref{lim:finfty} is the following straightforward extension of \eqref{Finfty:expansion}: for all $q\in (q_1,\infty)$,
\begin{align}
f_\infty(z)&= G_q\{g\}(0)+\frac{\lv}{\pi}\biggl(\log \frac{2^{1/2}}{q^{1/2}\vep}-\EM\biggr)\lla \varphi f_\infty\rra+\lv \mathcal K\{f_\infty\}(z) +\mathcal O^f_\infty,\label{expansion:varphi=1-1}
\end{align}
holding uniformly in $|z|\leq M$. Here, $\mathcal K\{\cdot\}$ is as defined in \eqref{def:K}, and
we define 
\[
\mathcal O^f_\infty\,\defeq\, \mathcal O_{\varphi,\lambda,g,q}(\vep^{2/3})(1+\|f_\infty\|_\varphi).
\]
To proceed, observe that by \eqref{expansion:varphi=1-1}, the following expansion holds uniformly for all $|z|,|z'|\leq M$:
\begin{gather}
f_\infty(z')=f_\infty(z)+\lv \mathcal K\{f_\infty\}(z')-\lv \mathcal K\{f_\infty\}(z)+\mathcal O^f_\infty,\label{expansion:varphi=1-2}
\end{gather}
which is analogous to \eqref{1replace0} in Lemma~\ref{lem:1replace00}. Since $\lla\varphi\rra = 1$, \eqref{expansion:varphi=1-2} can be used to expand $\lla \varphi f_\infty\rra$ in \eqref{expansion:varphi=1-1}. Specifically, integrating both sides of \eqref{expansion:varphi=1-2} against $\varphi(z')\d z'$ yields
\begin{align*}
\lla \varphi f_\infty\rra&=f_\infty(z)+\lv \int \varphi(z')\mathcal K\{f_\infty\}(z')\d z'- \lv \mathcal K\{f_\infty\}(z)+\mathcal O^f_\infty\\
&=f_\infty(z)+\lv \iint \varphi(z')\kappa(z'-z'')\varphi(z'')\\
&\quad \times \left[f_\infty(z)+\lv \mathcal K\{f_\infty\}(z'')-\lv \mathcal K\{f_\infty\}(z)+\mathcal O^f_\infty\right]\d z'\d z''  - \lv \mathcal K\{f_\infty\}(z)+\mathcal O^f_\infty\\
&=f_\infty(z)+\lv \iint \varphi(z')\kappa(z'-z'')\varphi(z'')\d z'\d z''f_\infty(z) - \lv \mathcal K\{f_\infty\}(z)+\wt{\mathcal O}^f_\infty,
\end{align*}
where the second equality writes out the definition of $\mathcal K\{\cdot\}$ in \eqref{def:K} and then applies \eqref{expansion:varphi=1-2} again with $z'$ replaced by $z''$, and 
\[
\wt{\mathcal O}^f_\infty\,\defeq\, \mathcal O_{\varphi,\lambda,g,q}(\log^{-2}\vep^{-1})(1+\|f_\infty\|_\varphi).
\]

We now complete the proof of \eqref{lim:finfty}. Recall that we set $\mu = 1$ in \eqref{def:lvoointro}. Substituting the right-hand side for $\lla \varphi f_\infty \rra$ in \eqref{expansion:varphi=1-1} yields:
\begin{align}
f_\infty(z)&=G_q\{g\}(0)+\frac{1}{\pi} \left(\frac{\pi}{\log \vep^{-1}}+\frac{ \lambda\pi}{\log^2 \vep^{-1}}\right)\times\biggl(\log \frac{2^{1/2}}{q^{1/2}\vep}-\EM\biggr)\notag\\
&\quad \times \biggl(f_\infty(z)+\lv \iint \varphi(z')\kappa(z'-z'')\varphi(z'')\d z'\d z'' f_\infty(z)- \lv \mathcal K\{f_\infty\}(z)+\wt{\mathcal O}^f_\infty\biggr)\notag\\
&\quad +
\lv \mathcal K\{f_\infty\}(z)+\mathcal O^f_\infty\notag\\
&=G_q\{g\}(0)+f_\infty(z)+\Biggl(\frac{\pi \lla \mc E(\varphi)\rra+\log \frac{2^{1/2}}{q^{1/2}}+\lambda-\EM}{\log \vep^{-1}} \Biggr)f_\infty(z)+\wt{\mathcal O}^f_\infty,
\label{expansion:varphi=1-3}
\end{align}
where we use the definition \eqref{def:energy} of $\mathcal E(\varphi)$ in the last equality. 
By \eqref{lim:beta}, the prefactor of the third term in \eqref{expansion:varphi=1-3} equals $-\frac{1}{2}\log(q/\beta) / \log \vep^{-1}$. Since $\lv \sim \pi / \log \vep^{-1}$, rearranging both sides of the previous display and letting $\vep \to 0$ shows that we can select a sufficiently large $q_0 \in (q_1, \infty)$, independent of $\vep$, so that \eqref{lim:finfty} holds for all $q \in (q_0, \infty)$.\medskip 

\section{AI disclosure} 
The paper has used AI for language editing. All ideas and proofs in this work were developed by the author.

\end{document}